\documentclass[11pt,english,a4paper]{article}
\usepackage[T1]{fontenc}
\usepackage[latin9]{inputenc}
\usepackage{geometry}
\usepackage{color}
\definecolor{note_fontcolor}{rgb}{0.53125, 0.53125, 0.53125}
\usepackage{refstyle}
\usepackage{mathtools}
\usepackage{amsmath}
\usepackage{amsthm}
\usepackage{amssymb}
\usepackage{enumitem}
\usepackage{float}
\PassOptionsToPackage{normalem}{ulem}
\usepackage{ulem}
\usepackage[unicode=true,pdfusetitle,
 bookmarks=true,bookmarksnumbered=false,bookmarksopen=false,
 breaklinks=false,pdfborder={0 0 1},backref=false,colorlinks=false]
 {hyperref}

\makeatletter

\AtBeginDocument{}
\AtBeginDocument{\providecommand\eqref[1]{\ref{eq:#1}}}

\RS@ifundefined{subsecref}
  {\newref{subsec}{name = \RSsectxt}}
  {}
\RS@ifundefined{thmref}
  {\def\RSthmtxt{theorem~}\newref{thm}{name = \RSthmtxt}}
  {}
\RS@ifundefined{lemref}
  {\def\RSlemtxt{lemma~}\newref{lem}{name = \RSlemtxt}}
  {}

\theoremstyle{definition}
\newtheorem*{example*}{\protect\examplename}
\theoremstyle{plain}
\newtheorem{thm}{\protect\theoremname}
\theoremstyle{plain}
\newtheorem{rem}[thm]{\protect\remarkname}
\theoremstyle{plain}
\newtheorem{example}[thm]{\protect\examplename}
\theoremstyle{remark}
\newtheorem*{rem*}{\protect\remarkname}
\theoremstyle{plain}
\newtheorem{lem}[thm]{\protect\lemmaname}
\theoremstyle{plain}

\theoremstyle{plain}
\newtheorem{prop}[thm]{\protect\propositionname}
\theoremstyle{remark}

\def\L{L}
\def\F{F}
\def\R{\mathbb R}
\def\Z{\mathbb Z}
\def\N{\mathbb N}
\def\ep{\varepsilon}

\global\long\def\Ff{\mathcal F}%

\global\long\def\Lf{\mathcal L}%

\definecolor{green}{rgb}{0.1,0.6,0}

\def\RSlemtxt{Lemma~}

\makeatother

\usepackage{babel}
\providecommand{\examplename}{Example}
\providecommand{\claimname}{Claim}
\providecommand{\lemmaname}{Lemma}
\providecommand{\propositionname}{Proposition}
\providecommand{\remarkname}{Remark}
\providecommand{\theoremname}{Theorem}
\providecommand{\corollaryname}{Corollary}

\title{Populations facing a {\it nonlinear} environmental gradient: steady states and pulsating fronts}
\date{}
\begin{document}
\global\long\def\red#1{\textcolor{red}{#1}}%
\global\long\def\blue#1{\textcolor{blue}{#1}}%
\global\long\def\green#1{\textcolor{green}{#1}}%
\global\long\def\Ff{\mathcal{F}}%
\global\long\def\Fmu{\mathcal{F}^{\mu}}%
\global\long\def\Lf{\mathcal{L}}%
\global\long\def\Lmu{\mathcal{L}^{\mu}}%
\global\long\def\ep{\varepsilon}%
\global\long\def\R{\mathbb{R}}%
\global\long\def\magenta#1{\textcolor{magenta}{#1}}%
\global\long\def\Re{\Re}%
\global\long\def\Im{\Im}%
\global\long\def\ind{\ind}%

\renewcommand\Re{\operatorname{Re}}
\renewcommand\Im{\operatorname{Im}}
\renewcommand\ind{\operatorname{ind}}

\maketitle

\begin{center}
{\large\bf Matthieu Alfaro\footnote{Universit\'e de Rouen Normandie, CNRS, Laboratoire de Math\'ematiques Rapha\"el Salem, Saint-Etienne-du-Rouvray, France \& BioSP, INRAE, 84914, Avignon, France. e-mail: {\tt matthieu.alfaro@univ-rouen.fr}} and Gwena\"el Peltier\footnote{IMAG, Univ. Montpellier, CNRS, Montpellier, France. e-mail: {\tt gwenael.peltier@umontpellier.fr}}} \\
[2ex]
\end{center}


\tableofcontents{}

\vspace{10pt}

\begin{abstract} We consider a population structured by a space
variable and a phenotypical trait, submitted to dispersion,
mutations, growth and nonlocal competition. This population is facing an {\it environmental gradient}: to survive at location $x$, an individual must have a trait close to some optimal trait $y_{opt}(x)$. Our main focus is to understand the effect of a {\it nonlinear} environmental gradient. 

We thus consider a nonlocal parabolic equation for the distribution of the population, with $y_{opt}(x)=\ep\theta(x)$, $0<\vert \ep \vert \ll 1$. We construct steady states solutions and, when $\theta$ is periodic, pulsating fronts. This requires the combination of rigorous perturbation techniques based on a careful application of the implicit function theorem in rather intricate function spaces. To deal with the phenotypic trait variable $y$ we take advantage of a Hilbert basis of $L^{2}(\R)$ made of eigenfunctions of an underlying Schr\"odinger operator, whereas to deal with the space variable $x$ we use the Fourier series expansions. 

Our mathematical analysis reveals, in particular, how both the steady states solutions and the fronts (speed and profile) are distorted by the nonlinear environmental gradient, which are important biological insights.
\\


\end{abstract}

\section{Introduction}\label{s:intro}

This paper is concerned with the nonlocal parabolic equation
\begin{equation}
\label{eq}
\partial_t u=\partial_{xx} u+\partial_{yy} u+u\displaystyle \left(1-A^2\left(y-y_{opt}(x)\right)^2-\int_{\R}u(t,x,y')\,dy'\right), \quad t>0, x\in \R, y\in \R,
\end{equation}
with
\begin{equation}
\label{def-yopt}
y_{opt}(x):=\ep \theta(x),\quad \theta \in C_b(\R),
\end{equation}
which serves as a model in evolutionary biology. Here $u=u(t,x,y)$ denotes the distribution of a population which, at each time $t>0$, is structured by a space variable $x\in\mathbb R$, and a phenotypic trait $y\in\mathbb R$. This population  is  submitted to  spatial dispersion, mutations, growth and
competition. The spatial dispersion and the mutations are
modeled by diffusion operators, namely $\partial_{xx}u$ and $\partial_{yy}u$. The intrinsic {\it per capita growth rate}
of the population depends
on both the location $x$ and the phenotypic trait
$y$. It is modeled by the confining term $1-A^2\left(y-y_{opt}(x)\right)^2$, where $A>0$ is a constant that measures the strength of the selection. This corresponds to a population living in an {\it environmental gradient}: to survive at location $x$,
an individual must have a trait close to the optimal trait
$y_{opt}(x)=\ep \theta(x)$. Finally, we consider a logistic
regulation of the population distribution that is local in the spatial
variable $x$  and nonlocal in the phenotypic trait $y$. In other words,  we consider
that there exists, at each location, an intra-specific competition which takes place with all individuals whatever their trait.

The main input of this work is to analyze the case of a {\it nonlinear} environmental gradient. To do so, we consider that the optimal trait  is described by (\ref{def-yopt})
with  $0<\vert \ep\vert \ll 1$, which corresponds to a nonlinear perturbation of the linear case $\ep=0$. First,  under some natural assumptions, we construct steady states solutions, shedding light on how Gaussian solutions (corresponding to $\ep=0$) are distorted by the nonlinear perturbation. Next, we consider the case of a periodic perturbation, $\theta\in C(\R/L\Z)$ for some $L>0$, for which we construct {\it pulsating fronts}  with a semi infinite interval of admissible speeds.

\medskip

In ecology, an {\it environmental gradient} refers to a gradual change in various factors in space that determine the favoured phenotypic traits. Environmental gradients can be related to factors such as altitude, temperature, and other environment characteristics. It is now well
documented that invasive species need to evolve during their range
expansion to adapt to local
conditions \cite{Ett_08}, \cite{Kel_08}.  Such issues are highly relevant in the context of the  global
warming \cite{Dav_05}, \cite{Dup_12}, or of the evolution of resistance of bacteria to antibiotics \cite{Her_12}, \cite{Bay_16}. Theoretical models therefore need to incorporate  evolutionary factors  \cite{Gri_06}, \cite{Key_06}, \cite{Her_12}. In this context,
let us mention the so-called \lq\lq cane toad equation'' which has led to rich  mathematical results  \cite{BenCalMeuVoi_12}, \cite{BouCalMeuMirPerRao_12},  \cite{BerMouRao_15},  \cite{BouHenRhy_16}. On the other hand, equations having the form of (\ref{eq}) were developed  in \cite{Pec_98}, \cite{Pre_04}, \cite{Pol_05}, \cite{MirRao_13}.

Before discussing propagation phenomena in (\ref{eq}), let us briefly recall that  {\it  traveling fronts} are particular  solutions that consist  of a constant profile connecting zero to \lq\lq a non-trivial state'' and shifting at a constant speed. This goes back to the seminal works \cite{Fis_37}, \cite{KolPetPis_37} on the  Fisher-KPP equation
\begin{equation*}
 \partial_t u=\Delta u+u(1-u), \quad t>0, x\in \R^N,
\end{equation*}
and, among so many others, \cite{AroWei_75,AroWei_78}, \cite{FifMac_77}. The construction of such solutions is much harder when the equation does not enjoy the comparison principle. One then usually needs to use topological degree arguments and the identification of the \lq\lq non-trivial state'' is typically missing, see  e.g.
\cite{BerNadPerRyz_09}, \cite{AlfCov_12}, \cite{HamRyz_14} on the nonlocal Fisher-KPP equation. 
 
As far as the mathematical analysis of (\ref{eq}) is concerned, one has to deal with the fine interplay between the space variable $x$ and the phenotypic trait $y$, the fact that the phenotypic space is unbounded, and the nonlocal competition term. Because of the latter, equation (\ref{eq}) does not enjoy the comparison principle and its analysis is quite involved since
many techniques based on the comparison principle --- such as some
monotone iterative schemes or the sliding method --- are unlikely to be used.

Despite of that,  the linear environmental gradient case, namely 
\begin{equation}
\label{def:yopt-linear}
y_{opt}(x)=Bx, \quad \text{ for some }  B\in \R,
\end{equation}
is now rather well understood. In this case, depending on the sign of an underlying principal eigenvalue \cite{AlfCovRao_13}, either the population gets extinct, or it is able to adapt progressively to uncrowded zones and  invade the environment. When propagation occurs, known results are the following. First, the $B=0$ case allows a separation of variables trick, from which a rather exhaustive analysis can be performed \cite{BerJinSil_16}. Roughly speaking, traveling fronts can be written in the form $\Gamma_0(y)U(x-ct)$, where $\Gamma _0(y)$ is an underlying {\it ground state} or {\it principal eigenfunction} and $U(z:=x-ct)$ a Fisher-KPP traveling wave with speed $c$. This fact will be precised and exploited later in the present work. On the other hand, when $B\neq 0$, variables cannot be separated and careful estimates of the nonlocal competition term are required.  Thanks to rather sharp {\it a priori} estimates, Harnack and Bernstein type refined inequalities, traveling fronts are constructed in \cite{AlfCovRao_13} and the determinacy of the spreading speed in the associated Cauchy problem is obtained in \cite{AlfBerRao_17}.  Very recently, accelerating invasions induced by  initial {\it heavy tails} of the population distribution --- see 
\cite{HamRoq_10} and \cite{Gar_11} for related results in absence of evolution---  have been analysed in \cite{Pel_20}. 

Last, let us mention that the case of a moving optimum
$$
y_{opt}(t,x)=B(x-c_{s}t), \quad \text{ for some } B\in \R, c_{s}>0,
$$
is also analyzed in \cite{AlfBerRao_17}. This case serves as a model to study, e.g., the effect of global warming on the survival and propagation of a species: the favorable areas are shifted by the climate change at a given speed $c_s>0$. The outcome is that there is an identified critical climate speed $c_s^*\geq 0$ such that $c>c_s^*$ implies extinction, whereas $c_s<c_s^*$ implies survival and invasion. 

\medskip

Nevertheless, the case of nonlinear environmental gradients is of great importance for applications, for instance in the context of development of resistance of pathogens to antibiotics. In this respect, let us mention the  experimental set up of \cite{Bay_16} where, thanks to mutation, E. Coli bacteria are able to cross a  four feet long petri dish on which the antibiotic concentration sharply increases\footnote{see the striking movie at  https://www.youtube.com/watch?v=yybsSqcB7mE}.

As far as we know no significant mathematical results exist for model (\ref{eq}) when the environmental gradient $y_{opt}(x)$ is nonlinear. The reasons are, at least, threefold. First of all, it is much harder, if possible, to relate the issue to a underlying eigenvalue problem. Second, it is expected that the population may survive while being blocked in a restricted zone (so that invasion does not occur). Last, if invasion occurs, tracking the propagation of the solution is far from obvious since, among others, geometrical effects (via curvature) may appear along the optimal curve $y=y_{opt}(x)$.

\medskip

Thus, in order to understand the situation where the optimal trait no longer depends linearly on space, our strategy is to consider the case (\ref{def-yopt}) with $0< \vert \ep\vert \ll 1$, which we see as a nonlinear perturbation of the case (\ref{def:yopt-linear}) with $B=0$ studied in \cite{BerJinSil_16}.

Our first goal is to construct steady states, which we denote $n=n^{\ep}(x,y)$, to (\ref{eq}). To do so, we will rely on rigorous perturbation techniques based on the implicit function theorem. We will  also take advantage of the orthonormal basis of $L^2(\R)$ consisting of eigenfunctions of the underlying operator
$$
-\frac{d^2}{dy^2}-(1-A^2y^2).
$$
This requires to work in rather intricate function spaces. Besides this rigorous theoretical construction, asymptotic expansions combined with numerical explorations enable to capture the distortion of the steady state by the nonlinear perturbation of the environmental gradient.

Our second goal is to analyze the propagation phenomena arising from model (\ref{eq}). To do so, for $\theta$ being  $L$-periodic, we construct {\it pulsating fronts}. These particular solutions were first introduced by \cite{ShiKawTer_86} in a biological context, and by Xin  \cite{Xin_91, Xin_93, Xin_00} in the framework of flame
propagation, as  natural extensions, in the periodic
framework, of the aforementioned traveling fronts. By definition, a pulsating front is a speed $c_{\ep}\in \R$ and a profile $\tilde{u}^{\ep}(z,x,y)$, that is $L$-periodic in the $x$ variable, such that 
$$
u^{\ep}(t,x,y):=\tilde{u}^{\ep}(x-c_{\ep}t,x,y)
$$
solves equation (\ref{eq}) and such that, as $z\to \pm \infty$, $\tilde{u}^{\ep}(z,x,y)$ connects zero to a \lq\lq non-trivial periodic state'', a natural candidate being the steady state $n^{\ep}(x,y)$ constructed previously. Equivalently, a pulsating front is a 
solution  connecting zero to a \lq\lq non-trivial periodic state'', and that satisfies the constraint
$$
u^{\ep}\left(t+\frac L {c_{\ep}},x,y\right)=u^{\ep}(t,x-L,y), \quad \forall (t,x,y)\in \R ^3.
$$
 As far as monostable pulsating fronts are
concerned, we refer among others to the seminal works of Weinberger \cite{Wei_02},
Berestycki and Hamel \cite{BerHam_02}. Let us also mention \cite{HudZin_95},
\cite{BerHamRoq_05_2}, \cite{Ham_08}, \cite{HamRoq_11} for
related results. In contrast with these results and as mentioned above, model (\ref{eq}) does not enjoy the comparison principle. In such a situation, construction of pulsating fronts in a Fisher-KPP situation was recently achieved in \cite{AlfGri_18} (see \cite{ConDomRoqRyz_06}, \cite{Hen_14} for an ignition type nonlinearity and a different setting). Another inherent difficulty of the present situation is to deal with both variables $x$ (space) and $y$ (phenotypic trait). To do so, we will first  use the orthonormal basis of $L^2(\R)$ mentioned above to deal with $y$ and then  use the Fourier series expansions to deal with $x$. Again, this is combined with a careful use of rigorous perturbation techniques based on the implicit function theorem.  As far as we know, such perturbation arguments to construct pulsating fronts are rather used in the ignition  \cite{BagMar_09} or bistable cases
\cite{DinHamZha_14}. Besides this rigorous theoretical construction, our analysis reveals how the speed and profile of the fronts are modified  by the nonlinear perturbation of the environmental gradient, which are very relevant for biological applications.

\section{Main results}\label{s:results}

Letting 
\begin{equation}\label{def:r}
r(y):=1-A^2y^2,\quad A>0,
\end{equation}
equation (\ref{eq}) is recast
\begin{equation}
\label{eq-r}
\partial_t u=\partial_{xx} u+\partial_{yy} u+u\displaystyle \left(r\left(y-\ep\theta(x)\right)-\int_{\R}u(t,x,y')\,dy'\right).
\end{equation}

\begin{rem}[Quadratic choice] If $r:\R\to\R$ is continuous and confining, that is $
\lim_{\vert y\vert\to \infty} r(y)=-\infty$, then the operator $
- \frac{d^{2}}{dy^{2}}-r(y)$
is essentially self-adjoint on $C^{\infty}_{c} (\R)$, and has discrete spectrum. There exists an orthonormal basis $\{\Gamma _k\}_{k\in \N}$ of $L^2(\R)$ consisting of eigenfunctions, namely
$$
-\Gamma _k ''-r(y)\Gamma _k=\lambda_k \Gamma _k, \quad \Vert \Gamma _k\Vert _{L^{2}}=1,
$$
with corresponding eigenvalues $
\lambda_0<\lambda_1\leq \lambda_2\leq \cdots \leq \lambda_k\to +\infty$ of finite multiplicity. Assuming that the confinement is, say, polynomial we may handle such per capita growth rate $r$ as in \cite{AlfVer_19}. For the clarity of the exposition (in particular some  relations  between the eigenfunctions are helpful, see subsection \ref{ss:linear}) we have nonetheless decided to consider the quadratic case (\ref{def:r}) which, anyway, reveals all the possible features of the model.
\end{rem}

In the sequel, we denote by $(\lambda_{0},\Gamma_{0}(y))$ the principal eigenelements of $-\frac{d^{2}}{dy^{2}}-r(y)$, namely 
\[
\begin{cases}
-\Gamma_{0}^{\prime\prime}-(1-A^2y^{2})\Gamma_{0}=\lambda_{0}\Gamma_{0} & \text{in }\mathbb{R},\\
\Gamma_{0}>0 & \text{in }\mathbb{R},\\
||\Gamma_{0}||_{L^2}=1,
\end{cases}
\]
that is
\begin{equation*}
\label{expression-eigen}
\lambda_0= A-1, \quad \Gamma _0(y)=\left(\frac{A}{\pi}\right)^{\frac 14}e^{-\frac 12  A y^2}.
\end{equation*}

We first state that, as soon as $\lambda _0 >0$ and $\vert \ep \vert \ll 1$, extinction
of the population occurs for rather general initial data, including in particular the case of continuous compactly supported ones.

\begin{prop}[Extinction] \label{prop:extinction}  Assume $\lambda_0>0$. Let us fix  $0<\mu_0<\lambda _0$ and $\mu _0+1<a<\lambda_0+1$.  Let $\theta \in C_b(\R)$. Then there is $\ep_0>0$ such that, for any $\vert \ep\vert < \ep_0$, the following holds: any global nonnegative solution $u=u^\ep (t,x,y)$ of (\ref{eq-r}), starting from a 
initial data $u_0=u_0(x,y)$ such that
\begin{equation}\label{def:M}
M=M(u_0):=\sup_{(x,y)\in \R^2} u_0(x,y)e^{\frac 12 a y^2}<+\infty,
\end{equation}
goes extinct exponentially fast as $t\to+\infty$. More precisely, we have
\begin{equation*}
0\leq u(t,x,y)\leq Me^{-\mu _0  t}e^{-\frac 1 2 a y^2}, \quad \text{ for all } t\geq 0, x\in\R, y\in \R. 
\end{equation*}
\end{prop}

When $\lambda _0 \geq 0$, extinction in the linear case  $\ep=0$ is easily proved thanks to the comparison principle since the nonlocal term is \lq\lq harmless'' when searching an estimate {\it from above}. Hence, when $\lambda _0>0$, the proof of  Proposition \ref{prop:extinction} follows from a rather direct perturbation argument. Notice that the critical case $\lambda_0=0$ is much more subtle, since more sensitive to perturbations, and left open here.

\medskip

We now focus on the case $\lambda_{0}<0$, for which survival is expected when $\vert \ep\vert \ll 1$. We thus inquire for nonnegative and nontrivial steady state $n=n^\ep(x,y)$ solving
\begin{equation}
\label{eq:statio_state}
\partial_{xx} n+\partial_{yy} n+n\displaystyle \left(r\left(y-\ep\theta(x)\right)-\int_{\R}n(x,y')\,dy'\right)=0.
\end{equation}
Notice that, in this paper, we reserve the notations $n=n^\ep(x,y)$ to steady states and $u=u^\ep(t,x,y)$ to time dependent solutions. Observe first that, when $\ep=0$, an appropriate renormalization of the ground state $\Gamma _0=\Gamma _0(y)$ provides a positive solution: it is obvious that
\begin{equation}
\label{def:n0}
n^0(y):=\eta \Gamma_0(y), \quad \eta:=\frac{-\lambda_{0}}{\Vert \Gamma_{0}\Vert_{L^1}}>0,
\end{equation}
solves (\ref{eq:statio_state}) when $\ep=0$. Our first main result is concerned with the construction of steady states when $\vert \ep\vert$ is small enough. 

\begin{thm}[Steady states]
\label{thm:steady_state_eps} Assume $\lambda_{0}<0$. Let $\theta\in C_{b}(\mathbb{R})$. Let us fix $\beta >\frac{13}4$.

Then there is $\ep_0>0$ such that, for any $\vert \ep \vert < \ep_0$, 
\begin{equation}\label{objectif1}
\text{ 
 there is a unique $n^{\varepsilon}\in Y$ such that $n^{\ep}$ solves (\ref{eq:statio_state})},
\end{equation}
where the function space $Y$ is given by (\ref{eq:Y_space}). Additionally, we have 
 \begin{equation}\label{terme-un}
 n^{\varepsilon}=n^{0}+\ep n^{1}+o(\ep) \quad \text{  in $Y$, as } \varepsilon\rightarrow 0,
 \end{equation}
where $Y$ is equipped with the norm (\ref{eq:Y_norm}), and where
\begin{equation}
\label{def-terme-un}
n^{1}=n^{1}(x,y):=A (\rho_A*\theta)(x) yn_0(y),
\end{equation}
with $\rho_A$ the probability density given by
\begin{equation}
\label{def:proba}
\rho_A(z):=\frac 12 \sqrt{2A}e^{-\sqrt{2A}\vert z\vert}.
\end{equation}

If we assume further that $\theta\in C_{b}^{m}(\mathbb{R})$ for some $m\geq 1$, then the same conclusions hold true with $Y$ replaced by  $Y_{m}$ given by (\ref{eq:Ym_space}).
\end{thm}

The proof relies on rigorous perturbation techniques, and involves rather intricate function spaces, such as $Y\subsetneq C^{2}_b(\R^{2})$, $Y_m\subsetneq C^{2+k}_b(\R^{2})$, which are precisely defined in (\ref{eq:Y_space}), (\ref{eq:Ym_space}).

The positivity of the constructed steady state is not provided by our proof. Nevertheless, it can be proved {\it a posteriori} in some prototype situations for the perturbation $\theta=\theta(x)$, in particular when it is periodic or {\it localized}. To state this, we denote
$$
C_{per}^{L}(\R):=\{f \in C(\R): \forall x \in \R, f(x+L)=f(x)\}, \quad \text{ for $L>0$},
$$
and
$$
C_0(\R):=\{f \in C(\R): \lim _{\vert x\vert \to +\infty} f(x)=0\}.
$$

\begin{thm}[Positive steady states, the periodic case]\label{thm:steady-per} Let the conditions of Theorem \ref{thm:steady_state_eps} hold and assume further that $\theta\in C_{per}^{L}(\R)$ for some $L>0$. Then, the steady states $n^\ep=n^{\ep}(x,y)$ constructed in Theorem \ref{thm:steady_state_eps} are $L$-periodic with respect to the $x$ variable. Furthermore, up to reducing $\ep _0>0$, there holds
\begin{equation}
\label{positivite-exp-per}
\forall a >0, \exists C>0, \forall \vert \ep \vert < \ep _0, \forall (x,y)\in \R^2, \quad 0<n^\ep (x,y) \leq Ce^{-a \vert y\vert }.
\end{equation}
\end{thm}

\begin{thm}[Positive steady states, the localized case]\label{thm:steady-C0} Let the conditions of Theorem \ref{thm:steady_state_eps} hold and assume further that $\theta\in C_0(\R)$. Then,  the steady states $n^\ep=n^{\ep}(x,y)$ constructed in Theorem \ref{thm:steady_state_eps} satisfy $n^\ep -n_0 \in \tilde{Y}$, where the function space $\tilde{Y}$ is given by (\ref{eq:Y_C0}) and equipped with the norm
(\ref{norme-Y-tilde}). In particular,
\begin{equation}
\label{uniform}
n^{\ep}(x,y)\to n^{0}(y), \quad \text{ as } \vert x\vert \to +\infty, \text{ uniformly w.r.t. $y\in \R$}.
\end{equation}
Furthermore, up to reducing $\ep _0>0$, there holds
\begin{equation}
\label{positivite-exp-per-bis}
\forall a >0, \exists C>0, \forall \vert \ep \vert < \ep _0, \forall (x,y)\in \R^2, \quad 0<n^\ep (x,y) \leq Ce^{-a \vert y\vert }.
\end{equation}
\end{thm}

The distortion of the positive steady state by the  nonlinear (periodic or localized) perturbation $\theta=\theta(x)$ is encoded in (\ref{terme-un})---(\ref{def-terme-un}) and will be discussed in details in subsections \ref{ss:bio-loc} an \ref{ss:bio-steady}.

\medskip

Next, still assuming $\lambda_{0}<0$, we enquire on the existence
of fronts for equation (\ref{eq-r}). To deal with the $\ep=0$ case, let us recall the well-known fact  concerning the Fisher-KPP traveling fronts: for any
\[
c_{0}\geq c^{*}_0:=2\sqrt{-\lambda_{0}}>0,
\]
there is a unique (up to translation) profile $U=U(z)$ solving
\begin{align}
\begin{cases}
U^{\prime\prime}+c_{0}U^{\prime}-\lambda_{0}U(1-U)=0 & \text{on }\mathbb{R},\\
U(-\infty)=1,\\
U(+\infty)=0.
\end{cases}\label{eq:FKPP_front}
\end{align}
which moreover satisfies $U'<0$. Equipped with a Fisher-KPP front $(c_0,U)$, a straightforward computation shows that, when $\ep=0$,
\[
u^{0}(x-c_{0}t,y):=U(x-c_{0}t)n^{0}(y)
\]
solves (\ref{eq-r}), where $n^{0}$ is the ground state given by
(\ref{def:n0}). As explained above, this corresponds to
a separation of the variables $z=x-c_{0}t$ and $y$. In other words, the profile
$n^{0}(y)$ invades the trivial state along the $x$ axis at the spreading
speed $c_{0}$.

Our second main result is concerned with the case $\theta\in C_{per}^{L}(\mathbb{R})$,
for which we construct fronts when $|\varepsilon|$ is small enough.
Because of the periodic term $\varepsilon\theta(x)$ in (\ref{eq-r}),
we look for a pulsating front of the form $u^{\varepsilon}(x-c_{\ep}t,x,y)$
with $u^{\ep}=u^{\ep}(z,x,y)$ satisfying
\begin{equation}
\begin{cases}
u^{\varepsilon}(z,\cdot,y)\in C_{per}^{L}(\mathbb{R}) & \forall(z,y)\in\mathbb{R}^{2},\\
u^{\varepsilon}(-\infty,x,y)=n^{\varepsilon}(x,y) & \text{uniformly w.r.t. }(x,y)\in\mathbb{R}^{2},\\
u^{\ep}(+\infty,x,y)=0 & \text{uniformly w.r.t. }(x,y)\in\mathbb{R}^{2},
\end{cases}\label{eq:ueps_cond}
\end{equation}
where $n^{\varepsilon}=n^{\ep}(x,y)$ is the (periodic in $x$) steady
state provided by Theorem \ref{thm:steady-per}. That is, $u^{\varepsilon}$
spreads at the perturbed speed $c_{\varepsilon}$ and connects the
steady state $n^{\varepsilon}$ to the trivial one. 

\begin{thm}[Pulsating fronts]
\label{thm:pulsating} Assume $\lambda_{0}<0$. Let us fix $\beta>\frac{19}{4}$
and $\gamma>3$. Assume $\theta\in C^{k,\delta}(\mathbb{R})\cap C_{per}^{L}(\R)$
with $L>0$ and where $k\in\mathbb{N}$, $0\leq\delta<1$ satisfy $k+\delta>\gamma+\frac{1}{2}$. Let $n^{\varepsilon}$ be the steady state solving (\ref{eq:statio_state}) and obtained from Theorem \ref{thm:steady-per}.

Then there is $\overline{\ep}_{0}>0$ such that, for any $\vert\ep\vert<\overline{\ep}_{0}$,
there are a speed $c_{\varepsilon}>0$ and a profile $u^{\varepsilon}=u^\ep(z,x,y)\in C_{b}^{2}(\mathbb{R}^{3})$
such that
\begin{equation*}
\begin{cases}
	u^{\ep}\textup{ satisfies } (\ref{eq:ueps_cond}),\\ 
	(t,x,y)\mapsto u^{\ep}(x-c_{\ep}t,x,y)
\textup{ solves }(\ref{eq-r}).\\
\end{cases}
\end{equation*}
Additionally, we have
\begin{equation}
|c_\ep-c_0|+\sup_{(z,x,y)\in \R^{3}}\; \Big| (1+y^2)e^{b|z|} \Big( u^{\varepsilon}(z,x,y)-U(z)n^{\varepsilon}(x,y) \Big) \Big|  \to 0, \text{ as } \ep \to 0.\label{eq:ueps_CV}
\end{equation}
for some $b>0$. 
\end{thm}

An inherent difficulty to the construction of pulsating fronts is that the underlying elliptic operator, see (\ref{op-deg}), is degenerate.  This requires to consider a regularization, see (\ref{op-reg}), via a parameter $0<\mu\ll 1$. For a fixed such $\mu$, we use rigorous perturbation techniques (from the $\ep=0$ situation), that involve very intricate function spaces, which are precisely defined in Section \ref{s:pulsating}. 
 To deal with the phenotypic trait variable $y$ we take advantage of a Hilbert basis of $L^{2}(\R)$ made of eigenfunctions of an underlying Schr\"odinger operator, whereas to deal with the space variable $x$ we use the Fourier series expansions. Last, thanks to a judicious choice of function spaces, we can let the regularization parameter $\mu \to 0$ and then catch the desired pulsating front solution for a {\it nontrivial} range of small $\vert \ep \vert$. We refer to Remark \ref{rem:spaces-2} for more technical and precise details. 
 
Let us comment on the issue of the positivity of the constructed pulsating front which is not provided by our proof.  One might be tempted to adapt the argument of subsection \ref{ss:positivity} which proves {\it a posteriori} the positivity of the constructed steady state, but this would require a precise control of the tail of the front as $z\to +\infty$, which is not reachable by our construction, nor by an adjustment of it. Nevertheless, we believe that a precise {\it a priori} argument, in the spirit of \cite{Ham_08}, may connect the \lq\lq positivity issue'' with some \lq\lq minimal speed issue'' denoted $c_\ep ^{*}$.  Equipped with this, we conjecture that, up to reducing $\ep_0>0$, one may prove {\it a posteriori} the positivity of the constructed pulsating front as soon as $c_0>c_0^*=2\sqrt{-\lambda_0}$. In other words, the positivity should not be lost, at least when we perturb from a {\it super-critical} traveling front. This is a very delicate issue, that would require lengthy arguments, and left here as an open question.
 
 Our analysis, see subsection \ref{ss:bio-puls}, reveals that 
 \begin{equation}\label{c1-egal-zero}
c_{\ep}=c_{0}+o(\ep),\quad\text{ as }\ep\to 0.
\end{equation}
In other words, the perturbation of the speed of the front by the periodic nonlinearity $\theta=\theta(x)$ vanishes {\it at the first order with respect to $\ep$}, which is a relevant biological information. Notice, however, that this could be guessed from the fact that, in view of the model, the sign of $\ep$ is expected to be irrelevant for the speed issue.  On the other hand, the distortion of the profile of the front is less predictable, but our mathematical analysis provides some clues. We refer to Example \ref{ex:periodic2} in subsection \ref{ss:bio-puls}.



\medskip

The organization of the paper is as follows. In Section \ref{s:preliminaries}, we prove the extinction result, namely Proposition \ref{prop:extinction}, and present some useful tools for the following, in particular some spectral properties. The steady states are constructed in Section \ref{s:steady} through the proofs of Theorem \ref{thm:steady_state_eps}, Theorem \ref{thm:steady-per} and Theorem \ref{thm:steady-C0}. In Section \ref{s:pulsating}, we construct pulsating fronts by proving Theorem \ref{thm:pulsating}. Last, in Section \ref{s:numerics},  we present some biological insights of our results, together with some numerical explorations.

\section{Preliminaries}\label{s:preliminaries}

\subsection{Extinction result}\label{ss:extinction}

We here consider the case $\lambda_0>0$ for which we prove extinction, as stated in Proposition \ref{prop:extinction}.

\begin{proof}[Proof of Proposition \ref{prop:extinction}] For $M\geq 0$ given by (\ref{def:M}), we consider
$$
\phi(t,y):=Me^{-\mu _0  t}e^{-\frac 1 2 a y^2}
$$
which satisfies
$$
\partial_t \phi-\partial_{xx}\phi-\partial_{yy} \phi-r(y-\ep\theta(x))\phi
=\left[(A^2-a^2)y^2-2 \ep A^2\theta(x)y+\ep^2A^2 \theta^2(x)+a-1-\mu _0\right]\phi.
$$
The discriminant of the quadratic polynomial in $y$ is $4A^2a^2\theta ^2(x) \ep ^2 -4(A^2-a^2)(a-1-\mu_0)$, which is uniformly (with respect to $x\in \R$) negative for $\vert \ep\vert \leq\ep _0$ for $\ep_0>0$ sufficiently small (recall that $\theta$ is bounded). As a result
$$
\partial_t \phi-\partial_{xx}\phi-\partial_{yy} \phi-r(y-\ep\theta(x))\phi\geq 0 \geq \partial _t u -\partial_{xx} u-\partial _{yy} u-r(y-\ep\theta(x))u.
$$
Since we know from (\ref{def:M}) that $\phi(0,y)\geq u_0(x,y)$,  we deduce from the comparison principle that $u(t,x,y)\leq \phi(t,y)$, which concludes the proof.
\end{proof}

\subsection{Implicit Function Theorem}\label{ss:ift}

We recall the {\it Implicit Function Theorem}, see \cite[Theorem 4.B]{Zei_86} for instance.

\begin{thm}[Implicit Function Theorem]
\label{thm:IFT}Let $X,Y,Z$ be Banach spaces over $\mathbb{K}=\mathbb{R}$
or $\mathbb{K}=\mathbb{C}$ with their respective norms $||.||_{X}$,
$||.||_{Y}$ and $||.||_{Z}$. Let $U$ be a open neighborhood of $(0,0)$
in $X\times Y$. Let $F\colon U\rightarrow Z$ be a map.
Suppose that
\begin{enumerate}[label=(\roman*)]
\item $F(0,0)=0$, and $F$ is continuous at $(0,0)$,\label{enu:continuite}
\item $D_{y}F$ exists as a partial Fréchet derivative on $U$,
and $D_{y}F$ is continuous at $(0,0)$,\label{enu:__IFT_diff_exists}
\item $D_{y}F(0,0)\colon Y\rightarrow Z$ is bijective.\label{enu:bij}
\end{enumerate}
Then the following are true:
\begin{itemize}
\item There are $r_{0}>0$ and $r_{1}>0$ such that, for
every $x\in X$ satisfying $||x||_{X}< r_{0}$, there is a
unique $y(x)\in Y$ for which $||y(x)||_{Y}\leq r_{1}$ and $F(x,y(x))=0$.
\item If $F$ is $C^{k}$ on $U$ with $0\leq k\leq\infty$,
then $x\mapsto y(x)$ is also  $C^{k}$ on a neighborhood of
$0$.
\end{itemize}
\end{thm}

\subsection{Linear material}\label{ss:linear}

In this subsection we fix $A>0$ and consider the operator $\mathcal{H}w:=-w^{\prime\prime}-(1-A^{2}y^{2})w$,
which corresponds to the harmonic oscillator. The following is well-known.

\begin{prop}[Eigenelements of the harmonic oscillator]
\label{prop:basis_eigenfunctions} The operator $\mathcal{H}w:=-w^{\prime\prime}-(1-A^{2}y^{2})w$
admits a family of eigenelements $(\lambda_i,\Gamma_i)_{i\in\N}$, where 
\begin{equation}
\lambda_{i}=-1+(2i+1)A,\label{eq:eigenval}
\end{equation}
and $\Gamma_{i}(y)=C_{i}H_{i}(\sqrt{A}y)e^{-\frac{1}{2}Ay{}^{2}}$. Here  $(H_{i})_{i\in\mathbb{N}}$ denotes the family of Hermite polynomials,
that is the unique family of real polynomials satisfying
\begin{equation*}
\int_{\mathbb{R}}H_{i}(x)H_{j}(x)e^{-x^{2}}dx=2^{i}i!\sqrt{\pi}\delta_{ij},\qquad\deg H_{i}=i,\label{eq:eigenfunc_hermite}
\end{equation*}
and 
\begin{equation}
C_{i}=\left(\frac{A}{\pi}\right)^{1/4}\sqrt{\frac{1}{2^{i}i!}}\label{eq:eigenfunc_cst}
\end{equation}
a normalization constant so that $||\Gamma_{i}||_{L^2}=1$.
	
Additionally, the family $(\Gamma_{i})_{i\in\mathbb{N}}$ forms a Hilbert
basis of $L^{2}(\mathbb{R})$.
\end{prop}

We now present some relations between the eigenfunctions, which will
prove useful in our proofs.

\begin{lem}[Some linear relations]\label{lem:relations} For any integer $i$, there holds
\begin{equation}\label{eq:yGamma}
y\Gamma_{i}(y)=p_{i}^{+}\Gamma_{i+1}(y)+p_{i}^{-}\Gamma_{i-1}(y), \qquad 
p_{i}^{+}:=\sqrt{\frac{i+1}{2A}},\qquad p_{i}^{-}:=\sqrt{\frac{i}{2A}}.
\end{equation}
and
\begin{equation}\label{eq:dGamma}
\Gamma_{i}^{\prime}(y)=q_{i}^{+}\Gamma_{i+1}(y)+q_{i}^{-}\Gamma_{i-1}(y), \qquad q_{i}^{+}=-\sqrt{\frac{(i+1)A}{2}},\qquad q_{i}^{-}=\sqrt{\frac{iA}{2}}.
\end{equation}
with the conventions $p_0^-\Gamma_{-1}(y)\equiv q_0^-\Gamma_{-1}(y)\equiv 0$.
\end{lem}

\begin{proof}
The Hermite polynomials are known to satisfy  the recursion relation 
\[
2xH_{i}(x)=H_{i+1}(x)+2iH_{i-1}(x).  
\]
Multiplying this by $C_{i}$ and setting $x=\sqrt{A}y$,
we get
$
2\sqrt{A}y\Gamma_{i}(y) =\frac{C_{i}}{C_{i+1}}\Gamma_{i+1}(y)+\frac{2iC_{i}}{C_{i-1}}\Gamma_{i-1}(y)$
which, combined with (\ref{eq:eigenfunc_cst}), proves (\ref{eq:yGamma}).

The Hermite polynomials are known to satisfy the relations
\[
H_{i}^{\prime}(x)=2iH_{i-1}(x),\qquad 2xH_{i}(x)=H_{i+1}(x)+2iH_{i-1}(x).
\]
Differentiating the expression $\Gamma_{i}(y)=C_{i}H_{i}(\sqrt{A}y)e^{-\frac{1}{2}Ay{}^{2}}$ and using the above relations, we reach
$$
\Gamma_{i}^{\prime}(y) =C_{i}\left(i\sqrt A H_{i-1}(\sqrt A y)-\frac{1}{2}\sqrt A H_{i+1}(\sqrt A y)\right)e^{-\frac{1}{2}Ay^{2}}=\frac{i\sqrt AC_i}{C_{i-1}} \Gamma_{i-1}(y)-\frac{\sqrt A C_{i}}{2C_{i+1}} \Gamma_{i+1}(y),
$$
which, combined with (\ref{eq:eigenfunc_cst}), proves (\ref{eq:dGamma}).
\end{proof}

We pursue with some $L^\infty$ and $L^1$ estimates on eigenfunctions, possibly with some polynomial weight. 

\begin{lem}[$L^\infty$ and $L^1$ estimates]
\label{lem:eigenfunc_control} There is $C=C(A)>0$ such that, for all $i\in \N$,
\begin{align}
||\Gamma_{i}||_{L^{1}} & \leq C i^{1/4},\label{eq:eigenfunc_control_L1}\\
||\Gamma_{i}||_{L^\infty} & \leq Ci^{1/4},\label{eq:eigenfunc_control_inf}
\end{align}
together with
\begin{equation}
||\Gamma_{i}^{\prime}||_{L^\infty}  \leq C i^{3/4},\quad  ||\Gamma_{i}^{\prime\prime}||_{L^\infty}  \leq C i^{5/4}, \label{eq:dGamma_control_inf}
\end{equation}
and
\begin{equation}
||y^{2}\Gamma_{i}||_{L^\infty}  \leq C i^{5/4},\quad  ||y^{4}\Gamma_{i}||_{L^\infty}  \leq C i^{9/4}.\label{eq:y2_Gamma_control_inf}
\end{equation}
\end{lem}

\begin{proof} 
The  non so standard $L^{1}$ estimate (\ref{eq:eigenfunc_control_L1}) can be found in \cite[Proposition 2.4]{AlfVer_19}, whereas the $L^{\infty}$ estimate (\ref{eq:eigenfunc_control_inf}) can be found in \cite[Proposition 2.6]{AlfVer_19}. Next, estimate (\ref{eq:dGamma_control_inf})  easily follows from the combination of  (\ref{eq:dGamma}) and  (\ref{eq:eigenfunc_control_inf}), whereas estimate (\ref{eq:y2_Gamma_control_inf}) easily follows from  (\ref{eq:yGamma}) and  (\ref{eq:eigenfunc_control_inf}). Details are omitted. 
\end{proof}

Throughout this paper, we denote $m_i$ the \lq\lq mass'' of the $i$-th eigenfunction, namely
\begin{equation}\label{def:masse}
m_{i}:=\int_{\mathbb{R}}\Gamma_{i}(y)dy.
\end{equation}

\section{Construction of steady states}\label{s:steady}

In this section, we prove Theorem \ref{thm:steady_state_eps} on steady states, Theorem \ref{thm:steady-per} on the periodic case, and Theorem \ref{thm:steady-C0} on the localized case.

\medskip

We look after a steady state solution to (\ref{eq:statio_state}) in the perturbative form  $n^{\varepsilon}(x,y)={}n^{0}(y)+h^{\varepsilon}(x,y)$, where $n^{0}=n^{0}(y)$, given by (\ref{def:n0}), is a steady state when $\ep=0$. From  straightforward computations, we are left to find  $h^{\varepsilon}$ satisfying  $\F(\varepsilon,h^{\varepsilon})=0$, where 
\begin{align*}
\F(\varepsilon,h):={} & h_{xx}+h_{yy}+ n^{0}\left(2A^{2}\varepsilon\theta(x)y-A^{2}\varepsilon^{2}\theta^{2}(x)-\int_{\mathbb{R}}h(x,y^{\prime})dy^{\prime}\right)\\
 & \quad +h\left(1-A^{2}(y-\varepsilon\theta(x))^{2}+\lambda_{0}-\int_{\mathbb{R}}h(x,y^{\prime})dy^{\prime}\right).
\end{align*}
We thus aim at applying the Implicit Function Theorem, namely Theorem \ref{thm:IFT},
to $\F\colon\mathbb{R}\times Y\rightarrow Z$ where
the function spaces $Y,Z$ are to be appropriately chosen.


\subsection{Function spaces\label{subsec:IFT_prelim}}\label{ss:spaces}

Let us fix $\beta>\frac{13}{4}$. Recall that $\Gamma _i=\Gamma _i(y)$ are the eigenfunctions defined in subsection \ref{ss:linear}. We set
\begin{equation}
Y:=\left\{ h\in C^{2}(\mathbb{R}^{2})\left|\begin{array}{c}
\exists C>0,\,\forall |\alpha|\leq 2,\quad\left|D^\alpha h(x,y)\right|\leq\frac{C}{(1+y^{2})^{2}}\quad\text{ on }\mathbb{R}^{2},\vspace{10pt}\\
\exists K>0,\,\forall k\leq2,\,\forall i\in\mathbb{N},\,\quad\underset{x\in\R}{\sup}\left|\int_{\mathbb{R}}D_{x}^{k}h(x,y)\Gamma_{i}(y)dy\right|\leq\frac{K}{(1+i)^{\beta+1-k/2}}
\end{array}\right.\right\} ,\label{eq:Y_space}
\end{equation}
and
\begin{equation}
Z:=\left\{ f\in C(\mathbb{R}^{2})\left|\begin{array}{c}
\exists C>0,\quad\left|f(x,y)\right|\leq\frac{C}{1+y^{2}}\quad\text{ on }\mathbb{R}^{2},\vspace{10pt}\\
\exists K>0,\forall i\in\mathbb{N},\quad\underset{x\in\R}{\sup}\left|\int_{\mathbb{R}}f(x,y)\Gamma_{i}(y)dy\right|\leq\frac{K}{(1+i)^{\beta}}
\end{array}\right.\right\} ,\label{eq:Z_space}
\end{equation}
equipped with the norms
\begin{equation}
||h||_{Y}:=\sum_{|\alpha|\leq2}\sup_{(x,y)\in\mathbb{R}^2}\left|(1+y^{2})^{2}D^{\alpha}h(x,y)\right|+\sum_{k=0}^{2}||D_{x}^{k}h||_{\beta+1-k/2},\label{eq:Y_norm}
\end{equation}
and
\begin{equation}
||f||_{Z}:=\sup_{(x,y)\in\mathbb{R}^2}\left|(1+y^{2})f(x,y)\right|+||f||_{\beta},\label{eq:Z_norm}
\end{equation}
where, for $m\in\mathbb{R}$, we define
\begin{equation}
||w||_{m}:=\sup_{i\in\mathbb{N}}\left[(1+i)^{m}\sup_{x\in\mathbb{R}}\left|\int_{\mathbb{R}}w(x,y)\Gamma_{i}(y)dy\right|\right].\label{eq:X_norm}
\end{equation}


\begin{rem}[Choice of the function spaces]\label{rem:spaces}
Let us comment on the spaces $Y,Z$ and the two controls
appearing in their definition. The crux of the proof is to show that
$\L:= D_{h}\F(0,0)$, given
by (\ref{eq:Lstate}), is bijective from $Y$ to $Z$: for
every fixed $f\in Z$, there is a unique $h\in Y$ such that $\L h=f$. First, thanks to the controls on the $y$-tails, i.e. the first constraint in the definition of $Y$ and $Z$,  $h(x,\cdot),f(x,\cdot)\in L^{2}(\mathbb{R})$
for all $x\in \R$. This allows to decompose $h$ and $f$ along the eigenfunction
basis $(\Gamma_{i})_{i\in\mathbb{N}}$, leading to (\ref{eq:L2_decomp}).
From there we obtain an expression of $h_{i}=h_{i}(x)$ given by (\ref{eq:h_i_expr})---(\ref{eq:h_0_expr}).
Next, the control on $f_{i}(x)=\int_{\mathbb{R}}f(x,y)\Gamma_{i}(y)dy$, i.e. the second constraint in the definition of $Z$, allows to prove the bounds (\ref{eq:h_i_bound})---(\ref{eq:hi_xx_bound}) for $h_{i}$.
This in turn allows to prove the control on the $y$-tails for $h$.
This is done by using (\ref{eq:eigenfunc_control_inf})---(\ref{eq:y2_Gamma_control_inf})
and by taking $\beta>\frac{13}{4}$. 
\end{rem}

In what follows, it is useful to keep in mind the straightforward
estimates
\begin{align}
\forall h\in Y, \forall |\alpha|\leq 2 , \forall (x,y)\in\mathbb{R}^2,\quad & \left|D^{k}h(x,y)\right|\leq\frac{||h||_{Y}}{(1+y^{2})^{2}},\label{eq:bound_Y}\\
\forall f\in Z, \forall (x,y)\in\mathbb{R}^2,\quad & \left|f(x,y)\right|\leq\frac{||f||_{Z}}{1+y^{2}},\label{eq:bound_Z}
\end{align}
and
\begin{equation}
\forall h\in Y, \forall x\in\mathbb{R}, \quad \left|\int_{\mathbb{R}}h(x,y)dy\right|\leq||h||_{Y}\int_{\mathbb{R}}(1+y^{2})^{-2}dy=\frac{\pi}{2}||h||_{Y}.\label{eq:bound_int_Y}
\end{equation}

\begin{lem}[$Y$ and $Z$ are Banach]\label{lem:YZ_Banach}
The spaces $Y$ given by (\ref{eq:Y_space}), and $Z$ given by  (\ref{eq:Z_space}), are Banach spaces when equipped with their respective norm $||.||_{Y}$ given by (\ref{eq:Y_norm}), and $||.||_{Z}$ given by (\ref{eq:Z_norm}).
\end{lem}

\begin{proof} For the sake of completeness, let us give a short proof that $Y$ is Banach, the proof for $Z$ being similar. Let $(h_{n})_{n\in\mathbb{N}}$ be a Cauchy sequence in $Y$. Since the injection $Y\hookrightarrow C^2_b(\R^2)$ is continuous and $C^2_b(\R^2)$ is Banach, there is $h\in C_{b}^{2}(\mathbb{R}^2)$ such that
$h_{n}\rightarrow h$ in the norm $||.||_{C_{b}^{2}(\R^2)}$. 

Let us prove that $h\in Y$. Set
\[
C_{n}:=\sum_{|\alpha|\leq2}\sup_{(x,y)\in\mathbb{R}^2}\left|(1+y^{2})^{2}D^{\alpha}h_{n}(x,y)\right|.
\]
Since $(h_{n})$ is Cauchy, the sequence $(C_{n})_{n\in\mathbb{N}}$
is bounded by some $C\geq 0$. Then, for all $|\alpha|\leq2$ and $(x,y)\in\mathbb{R}^{2}$, there holds $
\left|D^{\alpha}h_{n}(x,y)\right|  \leq \frac{C}{(1+y^{2})^{2}}$, and  $n\rightarrow+\infty$ yields
\[
\left|D^{\alpha}h(x,y)\right|\leq\frac{C}{(1+y^{2})^{2}}.
\]
Similarly, the sequence
\[
K_{n}:=\sum_{k\leq2}||D_{x}^{k}h_{n}||_{\beta+1-k/2}
\]
is bounded by some $K\geq 0$.  Then,  for all  $0\leq k\leq2$, there holds 
\[
\forall i\in\mathbb{N}, \forall x\in\mathbb{R}, \quad \left|\int_{\mathbb{R}}D_{x}^{k}h_{n}(x,y)\Gamma_{i}(y)dy\right|\leq\frac{K}{(1+i)^{\beta+1-k/2}}.
\]
Given that $|D_{x}^{k}h_{n}(x,y)|\leq\frac C{(1+y^{2})^{2}}$,
the dominated convergence theorem allows to let $n\to +\infty$ and obtain that the above estimate also holds for $h$. We conclude that $h\in Y$.

Now, very classical arguments (that we omit) yield that $h_n\to h$ in $Y$.
\end{proof}

We conclude this subsection with a preliminary result,  which in
particular states that each \lq\lq $h$-term'' appearing in $\F(\varepsilon,h)$ has its $Z$-norm controlled by the $Y$-norm of $h$.
For better readability we denote $yh$ and $y^{2}h$ the functions
$(x,y)\mapsto yh(x,y)$ and $(x,y)\mapsto y^{2}h(x,y)$ respectively.

\begin{lem}[Controlling in $Z$ the terms of $\F(\ep,h)$] \label{lem:prelim_controls} There is $C=C(A)>0$ such that, for all $h\in Y$,
\begin{equation}
\label{truc}
\max\left(||h||_{Z}, ||yh||_{Z}, ||y^{2}h||_{Z}, ||h_{xx}||_{Z}, ||h_{yy}||_{Z}\right)\leq C ||h||_{Y}.
\end{equation}
Also,  if $f=f(x,y)\in Z$ and $b=b(x)\in C_{b}(\mathbb{R})$, then  
\begin{equation}\label{truc2}
||bf||_{Z}\leq||b||_{L^\infty}||f||_{Z}.
\end{equation}
\end{lem}

\begin{proof} The proof of assertion (\ref{truc2}) is obvious. As for (\ref{truc}), the estimates for $h$ and $h_{xx}$ follow directly from the definitions of $Y$, $Z$ and their respective norms.  In the sequel, $C$ denotes a positive  constant that may change from line to line, but that always depends only on $A$.

Let us prove  $||h_{yy}||_{Z}\leq C||h||_{Y}$. Since 
$
\left|(1+y^{2})h_{yy}(x,y)\right|\leq\left|(1+y^{2})^{2}h_{yy}(x,y)\right|\leq||h||_{Y}$, it remains to consider the second term appearing in the right hand side of (\ref{eq:Z_norm}), that is $||h_{yy}||_{\beta}$. Integrating by parts and using (\ref{eq:dGamma}) we have
\begin{align*}
\left|\int_{\mathbb{R}}h_{yy}(x,y)\Gamma_{i}(y)dy\right| & =\left|\int_{\mathbb{R}}h(x,y)\Gamma_{i}^{\prime\prime}(y)dy\right|\\
 & \leq q_i^+ q_{i+1}^+\left|\int_{\mathbb{R}}h\Gamma_{i+2}dy\right|+\left(q_{i}^{+}q_{i+1}^{-}+q_{i}^{-}q_{i-1}^{+}\right) \left|\int_{\mathbb{R}}h\Gamma_{i}dy\right|+q_{i}^{-}q_{i-1}^{-}\left|\int_{\mathbb{R}}h\Gamma_{i-2}dy\right|\\
 & \leq  C  \frac {\sqrt{(1+i)(2+i)}}{(1+i)^{\beta+1}} ||h||_{Y},
\end{align*}
from the expressions of $q_i^{\pm}$ and the fact that $h\in Y$ so that $||h||_{\beta+1}\leq ||h||_{Y}$. As a result $||h_{yy}||_{\beta}\leq C||h||_{Y}$ and we are done.

As for the cases of $yh$ and $y^2h$, it suffices to use (\ref{eq:yGamma}) instead of (\ref{eq:dGamma}) and very similar arguments. 
\end{proof}

\subsection{Checking assumptions of Theorem \ref{thm:IFT}}\label{ss:checking}

Equipped with the function spaces $Y$ and $Z$, we thus consider 
\begin{align}
\F(\varepsilon,h):={} & h_{xx}+h_{yy}+ n^{0}\left(2A^{2}\varepsilon\theta(x)y-A^{2}\varepsilon^{2}\theta^{2}(x)-\int_{\mathbb{R}}h(x,y^{\prime})dy^{\prime}\right)\nonumber \\
 & \quad +h\left(1-A^{2}(y-\varepsilon\theta(x))^{2}+\lambda_{0}-\int_{\mathbb{R}}h(x,y^{\prime})dy^{\prime}\right).\label{Fstate}
\end{align}
Clearly $\F(0,0)=0$. We prove below that the assumptions  of Theorem \ref{thm:IFT} hold true.

\begin{proof}[Checking assumptions $\ref{enu:continuite}$ and $\ref{enu:__IFT_diff_exists}$ of Theorem \ref{thm:IFT}] We first check that $\F$ is well defined. Recalling that $n^0(y)=\eta \Gamma _0(y)$ and since $(\Gamma_{i})_{i\in\mathbb{N}}$ is orthonormal in $L^{2}(\mathbb{R})$, it is clear that the conditions in (\ref{eq:Z_space}) are satisfied, so that $n^{0}\in Z$. Similarly and in view of  (\ref{eq:yGamma}), $yn^{0}\in Z$. Next, for fixed $\varepsilon\in\mathbb{R}$ and
$h\in Y$, the function $b(x):=-A^{2}\varepsilon^{2}\theta^{2}(x)-\int_{\mathbb{R}}h(x,y^{\prime})dy^{\prime}$
is continuous and bounded thanks to (\ref{eq:bound_int_Y}), and therefore
$bn^{0}\in Z$ from Lemma \ref{lem:prelim_controls}. In the
same way, setting $\tilde{b}(x):=2A^{2}\varepsilon\theta(x)$, we obtain
$\tilde{b}yn^{0}\in Z$. Finally, the other terms in $\F(\varepsilon,h)$
also belong to $Z$, again by virtue of Lemma \ref{lem:prelim_controls}.

We now compute $D_h\F(0,0)$  the Fréchet derivative of $\F$ along
the second variable at point $(0,0)$. We have $\F(0,h)=\L h+ R (h)$, where 
\begin{equation}
\L h:= h_{xx}+h_{yy}+h\left(1-A^{2}y^{2}+\lambda_{0}\right)-n^{0}\int_{\mathbb{R}}h(x,y^{\prime})dy^{\prime},\label{eq:Lstate}
\end{equation}
and  $R(h)=-h\int_{\mathbb{R}}h(x,y^{\prime})dy^{\prime}$. From  Lemma \ref{lem:prelim_controls} and  (\ref{eq:bound_int_Y}), the remainder $R(h)$ satisfies
$$
||R(h)||_{Z}  \leq||h||_{Z}\left\Vert \int_{\mathbb{R}}h(\cdot,y^{\prime})dy^{\prime}\right\Vert _{L^\infty}\leq  C ||h||_{Y}^{2}.
$$
 On the other
hand, $\L\colon Y\rightarrow Z$ is a linear continuous
operator, which is readily seen by using Lemma \ref{lem:prelim_controls}
and (\ref{eq:bound_int_Y}). Since $\F(0,0)=0$,
we then have $D_{h}\F(0,0)=\L$.

Using similar arguments, one shows that $D_{h}\F$ is well-defined
on a neighborhood of $(0,0)$, as well as the continuity of $\F$
and $D_{h}\F$ at $(0,0)$.
\end{proof}

Now, the main part is to prove the bijectivity of $D_{h}\F(0,0)=\L:Y\to Z$.

\begin{proof}[Checking assumption $\ref{enu:bij}$ of Theorem \ref{thm:IFT}] We proceed by analysis and synthesis. Let $f\in Z$ be given, and assume there exists $h\in Y$ such that $\L h=f$. Thanks to (\ref{eq:bound_Y}) and (\ref{eq:bound_Z}),
 $f(x,\cdot)$ and $ h(x,\cdot)$ are in $L^{2}(\mathbb{R})$
for any $x\in\mathbb{R}$. Since the family of eigenfunctions $(\Gamma_{i})_{i\in\mathbb{N}}$ of Proposition \ref{prop:basis_eigenfunctions} forms a Hilbert basis of
$L^{2}(\mathbb{R})$, we can write
\begin{equation}
h(x,y)=\sum_{i=0}^{+\infty}h_{i}(x)\Gamma_{i}(y),\quad f(x,y)=\sum_{i=0}^{+\infty}f_{i}(x)\Gamma_{i}(y),\label{eq:L2_decomp}
\end{equation}
where, for any $i\in \N$,
\begin{equation*}
h_{i}(x):=\int_{\mathbb{R}}h(x,y)\Gamma_{i}(y)dy,\quad f_{i}(x):=\int_{\mathbb{R}}f(x,y)\Gamma_{i}(y)dy.\label{eq:L2_coord}
\end{equation*}
Notice that, for any $x\in \R$,  the equalities in (\ref{eq:L2_decomp}) correspond, \textit{a
priori}, to a convergence of the series in the Hilbert space $L^{2}(\R)$
norm. However, since $h\in Y$ and $f\in Z$, there
holds
\begin{equation}
||h_{i}||_{L^\infty}  \leq\frac{||h||_{Y}}{(1+i)^{\beta+1}},\label{eq:hi_apriori_bound}
\end{equation}
and
\begin{equation}
||f_{i}||_{L^\infty}  \leq\frac{||f||_{Z}}{(1+i)^{\beta}}.\label{eq:fi_bound}
\end{equation}
Consequently, since $\beta>\frac{13}{4}>\frac 54$ and (\ref{eq:eigenfunc_control_inf})
holds, the convergences in (\ref{eq:L2_decomp}) are also valid pointwise in $\R^{2}$. 
Similarly, thanks to (\ref{eq:eigenfunc_control_L1}),  the equality
\[
\int_{\mathbb{R}}h(x,y)dy=\sum_{i=0}^{+\infty}h_{i}(x)\int_{\mathbb{R}}\Gamma_{i}(y)dy
\]
holds pointwise in $\R$.  Also, thanks to (\ref{eq:bound_Y}) and (\ref{eq:bound_Z}), we obtain
that $h_{i}\in C_{b}^{2}(\mathbb{R})$ and $f_{i}\in C_{b}(\mathbb{R})$,
with
$$
h_{i}^{\prime}(x)  =\int_{\mathbb{R}}h_{x}(x,y)\Gamma_{i}(y)dy,\quad 
h_{i}^{\prime\prime}(x)  =\int_{\mathbb{R}}h_{xx}(x,y)\Gamma_{i}(y)dy.
$$

Now, we project equality $f=\L h$ on each $\Gamma _i$ so that, for all $x\in\mathbb{R}$,
\begin{align*}
f_{i}(x) & =\int_{\mathbb{R}}h_{xx}(x,y)\Gamma_{i}(y)dy+\int_{\mathbb{R}}h_{yy}(x,y)\Gamma_{i}(y)dy+\int_{\mathbb{R}}(1-A^2y^{2}+\lambda_{0})h(x,y)\Gamma_{i}(y)dy\\
 & \quad-\left(\int_{\mathbb{R}}h(x,y')dy'\right)\int_{\mathbb{R}}n^{0}(y)\Gamma_{i}(y)dy\\
 & =h_{i}^{\prime\prime}(x)+\int_{\mathbb{R}}h(x,y)\left[\Gamma_{i}^{\prime\prime}(y)+(1-A^2y^{2}+\lambda_{0})\Gamma_{i}(y)\right]dy-\eta\delta_{i0}\int_{\mathbb{R}}h(x,y')dy'\\
 & =h_{i}^{\prime\prime}(x)-(\lambda_{i}-\lambda_{0})h_{i}(x)-\eta\delta_{i0}\sum_{i=0}^{+\infty}h_{i}(x)\int_{\mathbb{R}}\Gamma_{i}(y)dy,
\end{align*}
where we have  integrated by parts and used (\ref{def:n0}). Therefore, $\L h=f$ is reduced to an infinite system of  linear ordinary differential equations for the $h_i$'s, namely 
\begin{equation}
h_{i}^{\prime\prime}-(\lambda_{i}-\lambda_{0}) h_{i}  =f_{i}(x), \quad    (i\geq1),\label{eq:h_i}
\end{equation}
and
\begin{equation}
h_{0}^{\prime\prime}+\lambda_{0}h_{0}  =f_{0}(x)+\eta\sum_{i=1}^{+\infty}m_{i}h_{i}(x),\label{eq:h_0}
\end{equation}
where we recall the notation (\ref{def:masse}) for the mass $m_i$. Notice that, combining (\ref{eq:hi_apriori_bound}) with
(\ref{eq:eigenfunc_control_L1}), the series appearing in the right-hand
side of (\ref{eq:h_0}) converges to a function in $C_{b}(\mathbb{R})$.

We first deal with the case $i\geq 1$, that is (\ref{eq:h_i}). Since $\lambda_{i}-\lambda_{0}>0$ and $f_{i}\in C_{b}(\mathbb{R})$, there is a unique solution
$h_{i}$ to (\ref{eq:h_i}) which remains in $C_{b}^{2}(\mathbb{R})$, and it is explicitly given by
\begin{equation}
h_{i}(x)=-\rho_{i}\ast f_{i}(x)\; \text{ where }\; \rho_{i}(z):=\frac{1}{2\sqrt{\lambda_{i}-\lambda_{0}}}e^{-\sqrt{\lambda_{i}-\lambda_{0}}|z|},\quad (i\geq1).\label{eq:h_i_expr}
\end{equation}
The functions $h_i$ ($i\geq 1$) now determined, we can deal with the $i=0$ case. Since $\lambda_{0}<0$, there is a unique solution
$h_{0}$ to (\ref{eq:h_0}) which remains in $C_{b}^{2}(\mathbb{R})$,  and it is explicitly given by
\begin{equation}
h_{0}(x)=-\rho_{0}\ast\left(f_{0}+\eta\sum_{i=1}^{+\infty}m_{i}h_{i}\right)(x)\; \text{ where }\; \rho_{0}(z):= \frac{1}{2\sqrt{-\lambda_{0}}}e^{-\sqrt{-\lambda_{0}}|z|}.\label{eq:h_0_expr}
\end{equation}

\medskip 

Conversely, we need to prove that, for $h_i=h_i(x)$ provided by (\ref{eq:h_i_expr}) and then (\ref{eq:h_0_expr}), 
the function
\begin{equation}
h(x,y):=\sum_{i=0}^{+\infty}h_{i}(x)\Gamma_{i}(y),\label{eq:h_expr}
\end{equation}
does belong to $Y$ and that $\L h=f$.

Let us first prove that $h\in C^{2}(\mathbb{R}^{2})$. In the sequel, $C$ denotes a positive constant that may change from line to line, but that  always depends only on $A$ and  $||f||_{Z}$. From (\ref{eq:fi_bound}) and (\ref{eq:eigenval}) we deduce that,  for all
$i\geq1$, 
\begin{align}
||h_{i}||_{L^\infty} & \leq||\rho_{i}||_{L^1}||f_{i}||_{L^\infty}\leq\frac{1}{\lambda_{i}-\lambda_{0}}\times\frac{||f||_{Z}}{(1+i)^{\beta}}\leq\frac{C}{(1+i)^{\beta+1}},\label{eq:h_i_bound}
\end{align}
\begin{align}
||h_{i}^{\prime}||_{L^\infty} & \leq||\rho_{i}^{\prime}||_{L^1}||f_{i}||_{L^\infty}\leq\frac{1}{\sqrt{\lambda_{i}-\lambda_{0}}}\times \frac{||f||_{Z}}{(1+i)^{\beta}}\leq\frac{C}{(1+i)^{\beta+1/2}},\label{eq:hi_x_bound}
\end{align}
and thus, from equation (\ref{eq:h_i}),
\begin{align}
||h_{i}^{\prime\prime}||_{L^\infty} & \leq||f_{i}||_{L^\infty}+(\lambda_{i}-\lambda_{0})||h_{i}||_{L^\infty}\leq\frac{C}{(1+i)^{\beta}}.\label{eq:hi_xx_bound}
\end{align}
Therefore,
with (\ref{eq:eigenfunc_control_inf}), the series in (\ref{eq:h_expr})
is normally convergent, and the equality is valid pointwise. Now,
since $\beta>\frac{13}{4}>\frac 5 4$, combining (\ref{eq:h_i_bound})---(\ref{eq:hi_xx_bound})
and (\ref{eq:eigenfunc_control_inf})---(\ref{eq:dGamma_control_inf}),
we obtain that $h\in C^{2}(\mathbb{R}^{2})$, with the pointwise
expressions
\begin{equation}
D_x^p D _y ^q h(x,y)=\sum_{i=0}^{+\infty}\frac{d^p h_i}{dx^p} (x)\frac{d ^q\Gamma_{i}}{dy^q}(y), \quad (p+q\leq 2).\label{eq:h_xy_decomp}
\end{equation}
Also, recalling definition (\ref{eq:X_norm}), we infer from  (\ref{eq:h_i_bound})---(\ref{eq:hi_xx_bound}) that 
$$
\sum_{k=0}^{2}||D_{x}^{k}h||_{\beta+1-k/2}<+\infty.
$$

In view of (\ref{eq:Y_norm}), we now need to prove that $(x,y)\mapsto(1+y^{2})^{2}D^{\alpha}h(x,y)$
is bounded for any multi-index $|\alpha|\leq2$. 
Using (\ref{eq:h_i_bound}) and (\ref{eq:dGamma_control_inf}),
we find that, for all $(x,y)\in\mathbb{R}^{2}$,
\begin{align*}
\left|(1+y^{2})^{2}h(x,y)\right| & \leq\sum_{i=0}^{+\infty}\left|h_{i}(x)\right|\times \left|(1+2y^{2}+y^{4})\Gamma_{i}(y)\right|\leq C\sum_{i=0}^{+\infty}\frac{i^{9/4}}{(1+i)^{\beta+1}}<+\infty,
\end{align*}
since $\beta >\frac{13}{4}>\frac 9 4$. Analogously, combining  (\ref{eq:h_i_bound})---(\ref{eq:hi_xx_bound}) with  (\ref{eq:dGamma_control_inf}), we can deal with $D^{\alpha}h$ for any other multi-index $\vert \alpha\vert \leq 2$. For instance, notice the so-called \lq\lq worst case'':
\begin{align*}
\left|(1+y^{2})^{2}h_{xx}(x,y)\right|  \leq C \sum_{i=0}^{+\infty}\frac{i^{9/4}}{(1+i)^{\beta}}<+\infty,
\end{align*}
since $\beta>13/4$. 

Eventually, we proved that $||h||_{Y}<+\infty$, therefore $h\in Y$ and the proof of  $\L h=f$ is clear.
\end{proof}

\subsection{Completion of the proof of Theorem \ref{thm:steady_state_eps}}
\label{ss:proof-th2}

\begin{proof}[Proof of Theorem \ref{thm:steady_state_eps}] From the above two subsections, we can apply  Theorem \ref{thm:IFT} to the function
$\F$ around the point $(0,0)$. Hence there are $\ep_0>0$ and $r_1>0$ such that, for any $\vert \ep \vert <\ep _0$, the following holds: there is a unique  $h^{\varepsilon}\in Y$ for which $||h^{\varepsilon}||_{Y}\leq r_{1}$
and $\F(\varepsilon,h^{\varepsilon})=0$.  Recalling  $n^{\varepsilon}(x,y)= n^{0}(y)+h^{\varepsilon}(x,y)$, this transfers into (\ref{objectif1}).

Let us now prove (\ref{terme-un}). Since $\F$ is of the class $C^1$ (the case of the variable $h$ was treated in subsection \ref{ss:checking} and the case of the $\ep$ variable is clear)  we deduce from Theorem \ref{thm:IFT},  $\F (\ep,h^{\ep})=0$ and the chain rule that
$$
D_\varepsilon \F(\varepsilon, h^\ep)+ D_h \F(\varepsilon, h^\ep)\left(\frac{d h^\ep}{d\ep}\right) = 0,
$$
which we evaluate at $\ep=0$ to get
$$
D_\ep\F(0,0)+L\left(\left.\frac{d h^\ep}{d\ep} \right\vert_{\ep=0} \right)=0.
$$
From the expression of $\F=\F(\ep,h)$ we easily compute $D_\ep\F(0,0)
=2A^2 \theta(x)yn^{0}(y)$, so that, since $n^0(y)=\eta \Gamma_0(y)$,
\begin{equation*}
\left.\frac{d h^\ep}{d\ep} \right\vert_{\ep=0}=-2A^{2}\L^{-1}\left(\theta(x)yn^{0}(y)\right)=-2A^{2}\eta \L^{-1}\left(\theta(x)y\Gamma _{0}(y)\right).
\end{equation*}
From (\ref{eq:yGamma}) we know $y\Gamma_0(y)=\frac{1}{\sqrt{2A}}\Gamma_1(y)$ so that
$$
\left.\frac{d h^\ep}{d\ep} \right\vert_{\ep=0}=-\sqrt 2 A^{3/2}\eta \L ^{-1}\left(\theta(x)\Gamma _1(y)\right).
$$
Now, we compute $\L ^{-1}\left(\theta(x)\Gamma _1(y)\right)$ via (\ref{eq:h_i_expr}) and (\ref{eq:h_0_expr}) and reach (recall that $m_1=\int_\R \Gamma_1(y)dy=0$)
\begin{eqnarray*}
\left.\frac{d h^\ep}{d\ep} \right\vert_{\ep=0}&=&-\sqrt 2 A^{3/2}\eta \left[\eta m_1 \left(\rho_0*(\rho _1*\theta)\right)(x)\Gamma _0(y)-(\rho_1*\theta)(x)\Gamma _1(y)\right]\\
&=& \sqrt 2 A^{3/2}\left[y \sqrt{2A}(\rho_1*\theta) (x) \right]n_0(y)\\
&=& 2A^2 (\rho_1*\theta)(x)\, yn_0(y),
\end{eqnarray*}
which can be recast (\ref{terme-un}).

It remains to consider the case when we assume further that $\theta\in C_{b}^{m}(\mathbb{R})$ for some $m\geq 1$ which, in particular, improves the regularity of the solution $n^{\ep}=n^{\ep}(x,y)$. In this case one can actually  redo the proofs above by replacing
the spaces $Y,Z$ in (\ref{eq:Y_space}) and (\ref{eq:Z_space}) with
$Y_{m},Z_{m}$ given by
\begin{equation}
Y_{m}:=\left\{ h\in C^{m+2}(\mathbb{R}^{2})\left|\begin{array}{c}
\exists C>0, \forall|\alpha|\leq m+2,\quad\left|D^{\alpha}h(x,y)\right|\leq\frac{C}{(1+y^{2})^{2}}\quad\text{on }\mathbb{R}^{2},\vspace{10pt}\\
\exists K>0, \forall k\leq m+2,\,\forall i\in\mathbb{N},\vspace{3pt}\\
\underset{x\in\R}{\sup}\left|\int_{\mathbb{R}}D_{x}^{k}h(x,y)\Gamma_{i}(y)dy\right|\leq\frac{K}{(1+i)^{\beta+(m+2-k)/2}}
\end{array}\right.\right\} ,\label{eq:Ym_space}
\end{equation}
and
\[
Z_{m}:=\left\{ f\in C^{m}(\mathbb{R}^{2})\left|\begin{array}{c}
\exists C>0, \forall|\alpha|\leq m,\quad\left|D^{\alpha}f(x,y)\right|\leq\frac{C}{1+y^{2}}\quad\text{on }\mathbb{R}^{2},\vspace{10pt}\\
\exists K>0, \forall k\leq m,\,\forall i\in\mathbb{N},\vspace{3pt}\\
\underset{x\in\R}{\sup}\left|\int_{\mathbb{R}}D_{x}^{k}f(x,y)\Gamma_{i}(y)dy\right|\leq\frac{K}{(1+i)^{\beta+(m-k)/2}}
\end{array}\right.\right\} ,
\]
equipped with their respective norms
\[
||h||_{Y_m}:=\sum_{|\alpha|\leq m+2}\sup_{(x,y)\in\mathbb{R}^{2}}\left|(1+y^{2})^{2}D^{\alpha}h(x,y)\right|+\sum_{k=0}^{m+2}||D_{x}^{k}h||_{\beta+(m+2-k)/2},
\]
and
\[
||f||_{Z_m}:=\sum_{|\alpha|\leq m} \sup_{x,y\in\mathbb{R}}\left|(1+y^{2})D^{\alpha}f(x,y)\right| +\sum_{k=0}^{m}||D_{x}^{k}f||_{\beta+(m-k)/2},
\]
where we recall definition  (\ref{eq:X_norm}). Details are omitted.
\end{proof}

\subsection{Additional properties in the periodic and localized cases}\label{ss:additional}

In this subsection, we start the proof of Theorem \ref{thm:steady-per} and Theorem \ref{thm:steady-C0}, estimates (\ref{positivite-exp-per}) and (\ref{positivite-exp-per-bis}) being postponed to the next subsection.

\begin{proof} [Proof of  the periodicity of the steady states in Theorem \ref{thm:steady-per}] In addition to the conditions of Theorem \ref{thm:steady_state_eps}, let us assume $\theta \in C_{per}^{L}(\R)$ for some $L>0$. Let us recall that, from subsection
\ref{ss:proof-th2}, for any $\vert \ep \vert <\ep _0$, there is a unique  $h^{\varepsilon}\in Y$ for which $||h^{\varepsilon}||_{Y}\leq r_{1}$
and $\F(\varepsilon,h^{\varepsilon})=0$. Defining
\[
\tilde{h}^{\varepsilon}(x,y):=h^{\varepsilon}(x+L,y),
\]
 one readily checks that $\F(\varepsilon,\tilde{h}^{\varepsilon})=0$
and $||\tilde{h}^{\varepsilon}||_{Y}\leq r_{1}$.
Therefore $h^{\varepsilon}\equiv \tilde{h}^{\varepsilon}$, that is $h^{\varepsilon}$
is $L$-periodic in $x$, and so is $n^{\ep}(x,y)=n^{0}(y)+h^{\ep}(x,y)$.
\end{proof}

\begin{proof} [Proof of $n^{\ep} - n_0 \in \tilde{Y}$, where $\tilde{Y}$ is defined by (\ref{eq:Y_C0}), in Theorem \ref{thm:steady-C0}]   In addition to the conditions of Theorem \ref{thm:steady_state_eps}, let us assume $\theta \in C_0(\R)$. Our proof relies on the following technical  lemma, whose proof is postponed. 

\begin{lem}[Function $G$]
\label{lem:theta-G} Let $\theta\in C_{0}(\mathbb{R})$. Then there
is a piecewise constant function $G>0$ such that
\begin{equation}
\begin{cases}
G(x)\geq \max\left(|\theta(x)|,\theta^{2}(x)\right), \qquad \forall x\in\R,\\
G\text{ is even on $\R$, nonincreasing on $[0,+\infty)$,}\\
\lim_{x\rightarrow +\infty} G(x)=0,
\end{cases}\label{eq:___G_hypo}
\end{equation}
together with the following property: there is $\sigma>0$ such that, for all $i\in\mathbb{N}$,
\begin{equation}
(\rho_{i}*G)(x)\leq\frac{\sigma}{1+i}G(x), \quad \forall x\in \R,\label{eq:___G_convo}
\end{equation}
where the $\rho_{i}$'s are given by (\ref{eq:h_i_expr}) and (\ref{eq:h_0_expr}).
\end{lem}

Then, equipped with such a function $G$, we can redo the proof of subsections \ref{subsec:IFT_prelim}
to \ref{ss:proof-th2} by replacing the spaces $Y,Z$ in
(\ref{eq:Y_space}) and (\ref{eq:Z_space}) with

\begin{equation}
\tilde{Y}:=\left\{ h\in C^{2}(\mathbb{R}^{2})\left|\begin{array}{c}
\exists C>0,\,\forall |\alpha|\leq 2, \quad\left|D^\alpha h(x,y)\right|\leq\frac{C G(x)}{(1+y^{2})^{2}}\quad\text{ on }\mathbb{R}^{2},\vspace{10pt}\\
\exists K>0,\,\forall k\leq2,\,\forall i\in\mathbb{N},\quad \left|\int_{\mathbb{R}}D_{x}^{k}h(x,y)\Gamma_{i}(y)dy\right|\leq\frac{KG(x)}{(1+i)^{\beta+1-k/2}}\quad \text{ on }\R
\end{array}\right.\right\},\label{eq:Y_C0}
\end{equation}
and
\[
\tilde{Z}:=\left\{ f\in C(\mathbb{R}^{2})\left|\begin{array}{c}
\exists C>0,\quad\left|f(x,y)\right|\leq\frac{CG(x)}{1+y^{2}}\quad\text{ on }\mathbb{R}^{2},\vspace{10pt}\\
\exists K>0,\forall i\in\mathbb{N},\quad \left|\int_{\mathbb{R}}f(x,y)\Gamma_{i}(y)dy\right|\leq\frac{KG(x)}{(1+i)^{\beta}}\quad \text{ on }\R
\end{array}\right.\right\},
\]
equipped with the norms
\begin{equation}\label{norme-Y-tilde}
||h||_{\tilde{Y}}:=\sum_{|\alpha|\leq2}\sup_{(x,y)\in\mathbb{R}^2}\left|(1+y^{2})^{2}G(x)^{-1}D^{\alpha}h(x,y)\right|+\sum_{k=0}^{2}\left\Vert \frac{1}{G}D_{x}^{k}h \right\Vert_{\beta+1-k/2},
\end{equation}
and
\[
||f||_{\tilde{Z}}:=\sup_{(x,y)\in\mathbb{R}^2} \left|(1+y^{2})G(x)^{-1}f(x,y)\right|+ \left\Vert \frac{1}{G}f \right\Vert_{\beta},
\]
where we recall definition (\ref{eq:X_norm}), and $\beta>\frac{13}{4}$.

Let us make some comments on how the proof is modified. One can readily
check that $\tilde{Y},\tilde{Z}$ are Banach, as in subsection \ref{ss:spaces}. Also, since $G\geq \max(|\theta|,\theta^{2})$, the map $\F:\tilde{Y}\rightarrow\tilde{Z}$
in (\ref{Fstate}) is well-defined, and its continuity, differentiability
are still valid, with $D_h\F(0,0)=\L$ given by (\ref{eq:Lstate}).
To conclude, we need to prove that, for a fixed $f\in\tilde{Z}$, there
exists a unique $h\in\tilde{Y}$ such that $\L h=f$. Following the same procedure as in subsection \ref{ss:checking},
we obtain that the $h_{i}$'s  are necessarily given by (\ref{eq:h_i_expr}) and (\ref{eq:h_0_expr}). We claim (see below) that $h$ defined by (\ref{eq:h_expr}) does belong to $\tilde{Y}$. Then we conclude the proof by applying Theorem \ref{thm:IFT} in the same way as in subsection \ref{ss:proof-th2}.

Let us show that $h$ defined by (\ref{eq:h_expr}) belongs to $\tilde{Y}$.
Notice that since $f\in\tilde{Z}$ and (\ref{eq:___G_convo}) holds,
we obtain, for all $i\geq1$, 
\begin{equation}
|h_{i}(x)|=\left|\rho_{i}*f_{i}(x)\right|\leq\frac{||f||_{\tilde{Z}}}{(1+i)^{\beta}}(\rho_{i}*G)(x)\leq\frac{\sigma||f||_{\tilde{Z}}}{(1+i)^{\beta+1}}G(x),\label{eq:C0_hi_bound}
\end{equation}
and similarly, since $\lambda_{i}-\lambda_{0}=2iA$ for all $i\geq1$,
we have
\begin{align*}
h_{i}^{\prime}(x)=\sqrt{\lambda_{i}-\lambda_{0}}(\rho_{i}*f_{i})(x) & \Rightarrow|h_{i}^{\prime}(x)|\leq\frac{\sqrt{2A}\sigma||f||_{\tilde{Z}}}{(1+i)^{\beta+1/2}}G(x),\\
h_{i}^{\prime\prime}(x)=(\lambda_{i}-\lambda_{0})(\rho_{i}*f_{i})(x) & \Rightarrow|h_{i}^{\prime\prime}(x)|\leq\frac{2A\sigma||f||_{\tilde{Z}}}{(1+i)^{\beta}}G(x).
\end{align*}
The bounds on $h_{0},h_{0}^{\prime},h_{0}^{\prime\prime}$, can then
be deduced. Indeed, from (\ref{eq:C0_hi_bound}) we have
\[
\left|f_{0}(x)+\eta \sum_{i=1}^{+\infty}m_{i}h_{i}(x)\right|\leq\left(1+\eta \sigma\sum_{i=1}^{+\infty}\frac{\vert m_{i}\vert}{(1+i)^{\beta+1}}\right)||f||_{\tilde{Z}}G(x),
\]
where the series converges from  (\ref{eq:eigenfunc_control_L1}) 
and $\beta>\frac{13}4>\frac 1 4$. Combining this with (\ref{eq:h_0_expr}) yields that (\ref{eq:C0_hi_bound}) also holds for $i=0$. For $h_{0}^{\prime},h_{0}^{\prime\prime}$ we proceed as above and thus deduce that $
\sum_{k=0}^{2}||G^{-1}D_{x}^{k}h||_{\beta+1-k/2}<+\infty$. 
It remains to prove the upper bound on $|D^{\alpha}h(x,y)|$ for $|\alpha|\leq2$.
As in subsection \ref{ss:checking},
we have that $h\in C^{2}(\mathbb{R}^{2})$ and (\ref{eq:h_xy_decomp}) holds.
Additionally, combining Lemma  \ref{lem:relations} and Lemma \ref{lem:eigenfunc_control}, for any $p,q\in\mathbb{N}$ such that $p+q\leq2$, there holds
\begin{align*}
\left|(1+y^{2})^{2}D_{x}^{p}D_{y}^{q}h(x,y)\right| & \leq\sum_{i=0}^{+\infty}\vert h_{i}^{(p)}(x)\vert  \times \vert\left(1+y^{2}\right)^{2}\Gamma_{i}^{(q)}(y)\vert\\
 & \leq C G(x)\sum_{i=0}^{+\infty}\frac{1}{(1+i)^{\beta+1-p/2}}\times i^{2}i^{1/4+q/2},
\end{align*}
for some constant $C>0$. The series converges since $\beta>13/4$. Hence $h\in\tilde{Y}$. 
\end{proof}

It remains to prove Lemma \ref{lem:theta-G}.

\begin{proof}[Proof of Lemma \ref{lem:theta-G}]
Set 
\[
0<\alpha_{i}:=\begin{cases}
\sqrt{\lambda_{i}-\lambda_{0}}=\sqrt{2iA} & i\geq1,\\
\sqrt{-\lambda_{0}}=\sqrt{1-A} & i=0,
\end{cases}
\]
and $\tilde{\alpha}:=\frac{1}{2}\min(\alpha_{0},\alpha_{1})>0$. Notice
that $\tilde{\alpha}\leq\frac{1}{2}\alpha_{i}$ for all $i\in\mathbb{N}$. Define 
$$
\tilde G(x):=\sup_{\vert t\vert \geq \vert x\vert} \max \left(\vert \theta (t)\vert,\theta^2(t),e^{-\frac{\tilde \alpha}{2}\vert t\vert}\right),
$$
which clearly satisfies (\ref{eq:___G_hypo}). Then the function 
\[
G(x):=\begin{cases}
\tilde{G}(0) & 0\leq|x|<1,\\
\max\left(\frac{\tilde{G}(0)}{2},\tilde{G}(1)\right) & 1\leq|x|<2,\\
\max\left( \frac{\tilde{G}(0)}{2^{k+2}} , \frac{\tilde{G}(1)}{2^{k+1}} , \frac{\tilde{G}(2)}{2^{k}} , \cdots,  \frac{\tilde{G}(2^{k})}{2},\tilde{G}(2^{k+1})\right) & 2^{k+1}\leq|x|<2^{k+2},\quad k\in\mathbb{N},
\end{cases}
\]
is piecewise constant, satisfies (\ref{eq:___G_hypo}) as well as
\begin{equation}
G(x)e^{\tilde{\alpha}x}\to +\infty, \quad \text{ as } x\to +\infty,\label{eq:___G_HT}
\end{equation}
and
\begin{equation}
G(x/2)\leq2G(x),\quad \forall x\geq0.\label{eq:___G_decay}
\end{equation}

It remains to prove (\ref{eq:___G_convo}). Since both $G$ and $\rho_{i}*G$ are even, it suffices to consider $x\geq 0$. We write
\begin{align*}
I(x) & :=\frac{\alpha_{i}^{2}}{G(x)}(\rho_{i}*G)(x)\\
 & =\frac{\alpha_{i}}{2G(x)}\left(\int_{-\infty}^{x/2}e^{-\alpha_{i}(x-z)}G(z)dz+\int_{x/2}^{x}e^{-\alpha_{i}(x-z)}G(z)dz+\int_{x}^{+\infty}e^{\alpha_{i}(x-z)}G(z)dz\right)\\
 & \eqqcolon I_{-}(x)+I_{0}(x)+I_{+}(x).
\end{align*}
Since $G$ is nonincreasing on $[0,+\infty)$ and satisfies (\ref{eq:___G_decay}),
there holds
\begin{align*}
I_{-}(x) & \leq\frac{\alpha_{i}}{2G(x)e^{\alpha_{i}x}}||G||_{\infty}\int_{-\infty}^{x/2}e^{\alpha_{i}z}dz\leq\frac{||G||_{\infty}}{2G(x)e^{\tilde{\alpha}x}},\\
I_{+}(x) & \leq\frac{\alpha_{i}}{2}e^{\alpha_{i}x}\int_{x}^{+\infty}e^{-\alpha_{i}z}dz=\frac{1}{2},\\
I_{0}(x) & =\frac{\alpha_{i}}{2G(x)e^{\alpha_{i}x}}\int_{x/2}^{x}e^{\alpha_{i}z}G(z)dz\leq\frac{\alpha_{i}}{e^{\alpha_{i}x}}\int_{x/2}^{x}e^{\alpha_{i}z}dz\leq1.
\end{align*}
Since $G$ satisfies (\ref{eq:___G_HT}), $I(x)$ is uniformly bounded
on $[0,+\infty)$ independently of $i\in\mathbb{N}$. Consequently,
since $\alpha_{i}^{2}=2iA$ for $i\geq1$, we see that (\ref{eq:___G_convo})
holds for some $\sigma>0$.
\end{proof}

\subsection{Positivity and control on the $y$-tails in the periodic and localized cases}\label{ss:positivity}

In this subsection, we prove  estimates (\ref{positivite-exp-per}) and (\ref{positivite-exp-per-bis}), thus completing the proof of Theorem \ref{thm:steady-per} and Theorem \ref{thm:steady-C0}.

\begin{proof}[Proof of  (\ref{positivite-exp-per}) and (\ref{positivite-exp-per-bis})] We assume either $\theta\in C_{per}^L(\R)$ for some $L>0$ (periodic case), or $\theta\in C_{0}(\R)$ (localized case). From subsection
\ref{ss:additional}, in the periodic case, $n^{\ep}\in Y$ is $L$-periodic
in $x$, while in the localized case, we have $n^{\varepsilon}-n^0\in\tilde{Y}$
where $\tilde{Y}$ is given by (\ref{eq:Y_C0}). Notice that, in both cases,
$n^{\varepsilon}-n^0\to 0$ as $\varepsilon\to 0$ (in $Y$ or in $\tilde{Y}$ respectively). As a result, by reducing $\ep_0>0$ if necessary, there holds that, for any $\vert \varepsilon\vert <\varepsilon_{0}$,
\begin{align}
\left|n^{\varepsilon}(x,y)\right| & \leq\frac{||n^{\varepsilon}||_{Y}}{(1+y^{2})^{2}} \leq\frac{2||n^0||_{Y}}{(1+y^{2})^{2}},  & & \text{in the periodic case},\label{eq:neps_per_bound}\\
\left|n^{\varepsilon}(x,y)-n^0(y)\right| & \leq\frac{||n^{\varepsilon}-n^0||_{\tilde{Y}}G(x)}{(1+y^{2})^{2}}\leq\frac{G(x)}{(1+y^{2})^{2}}, & & \text{in the localized case}.\label{eq:neps_loc_bound}
\end{align}

\medskip

Assume by contradiction that there is a sequence $\ep_p\to 0$ with $p\geq 1$ such that $n^{\varepsilon_{p}}$ is {\it not} nonnegative  on $\mathbb{R}^{2}$.

{\it Step 1: $n^{\varepsilon_{p}}$ admits a minimum.} Set $m_{p}:=\inf_{(x,y)\in\mathbb{R}^{2}}n^{\varepsilon_{p}}(x,y)<0$, and consider a
sequence $(x_{p}^{k},y_{p}^{k})_{k\in\mathbb{N}}$ such that 
$n^{\varepsilon_{p}}(x_{p}^{k},y_{p}^{k})\to m_{p}$ as $k\to +\infty$. From (\ref{eq:neps_per_bound})---(\ref{eq:neps_loc_bound}), $n^ {\ep_p}(x,y)$ tends to zero as $|y|\to +\infty$ uniformly in $x\in \R$. Thus there
exists $Y_{p}>0$ such that, for all $k$,  $|y_{p}^{k}|\leq Y_{p}$. Notice
that, despite (\ref{eq:neps_per_bound})---(\ref{eq:neps_loc_bound}), $Y_{p}$ depends a priori
on $p$ through the value of $m_{p}$. On the other hand, in the periodic case, we may consider that $x_{p}^{k}\in[0,L]$ while, in the localized case, from (\ref{eq:neps_loc_bound}) we have in the same way $|x_p^k|\leq X_p$. Therefore, assuming  $X_p\geq L$,  we have in both cases
\[
(x_{p}^{k},y_{p}^{k})\in[-X_p,X_p]\times[-Y_{p},Y_{p}],\quad\forall k\in\mathbb{N}.
\]
Hence, up to a
subsequence, $(x_{p}^{k},y_{p}^{k})$ converges to a point $(x_{p},y_{p})\in[-X_p,X_p]\times[-Y_{p},Y_{p}]$,
where $n^{\varepsilon_{p}}$ is thus reaching its minimum.

{\it Step 2: bound on $y_{p}$ that is uniform w.r.t. $p$.} From the steady state equation (\ref{eq:statio_state}) for $n^{\ep_p}$ evaluated at the minimum point $(x_p,y_p)$, we  obtain (recall $m_p<0$)
\begin{eqnarray}
0&=&\frac{1}{m_{p}}\left(\Delta_{x,y}n^{\varepsilon_{p}}(x_{p},y_{p})+n^{\varepsilon_{p}}(x_{p},y_{p})\left(1-A^{2}(y_{p}-\varepsilon_{p}\theta(x_{p}))^{2}-\int_{\mathbb{R}}n^{\varepsilon_{p}}(x_{p},y^{\prime})dy^{\prime}\right)\right)\nonumber\\
&\leq & 1-A^{2}(y_{p}-\varepsilon_{p}\theta(x_{p}))^{2}-\int_{\mathbb{R}}n^{\varepsilon_{p}}(x_p,y^{\prime})dy^{\prime}\nonumber\\
&\leq & 1-A^{2}y_{p}^{2}+2A^{2}\varepsilon_{0}|y_{p}|.||\theta||_{\infty}+A^{2}\varepsilon_{0}^{2}||\theta||_{\infty}^{2}+\begin{cases}
\pi||n^{0}||_{Y} & \text{in the periodic case},\\-\lambda_0 + \frac{\pi}{2}||G||_{\infty} & \text{in the localized case},
\end{cases}\label{eq:__yp_ineq}
\end{eqnarray}
where we used (\ref{eq:neps_per_bound})---(\ref{eq:neps_loc_bound}) in the last inequality. The above enforces the existence
of some $Y>0$ (independent of $p$) such that $|y_{p}|\leq Y$.

{\it Step 3: bound on $x_p$ that is uniform w.r.t. $p$.} In the periodic case, this is obvious since we can assume $x_p^k\in [0,L]$. In the localised case, thanks to (\ref{eq:neps_loc_bound}), we have for all $x\in \R$ and $|y|\leq Y$,
\[
n^{\ep_p}(x,y) \geq n^0(y) - \frac{G(x)}{(1+y^2)^2} ||n^{\ep_p}-n^0||_{\tilde{Y}} \geq n^0(Y) -G(x).
\]
This implies the existence of $X>0$ independent of $p$ such that $n^{\ep_p}(x,y)\geq \frac12 n^0(Y) > 0$ for any $|x|\geq X$ and $|y|\leq Y$. Consequently, for $k$ large enough, we have $|x_p^k|\leq X$. Assuming $X>L$, we thus have in both cases $|x_p| \leq X$.

{\it Step 4: deriving a contradiction.} From the above, we can assert that $(x_{p},y_{p})\in[-X,X]\times[-Y,Y]$
for $p$ large enough. However, let us underline that $n^{0}>0$ on $\mathbb{R}^{2}$
and, in both the periodic and the localized case, 
\[
||n^{\varepsilon}-n^{0}||_{L^{\infty}(\mathbb{R}^{2})}\rightarrow0,\quad\text{as }\varepsilon\rightarrow0.
\]
As a consequence, for $p$ large enough, there
holds $n^{\varepsilon_p}>0$ on $[-X,X]\times[-Y,Y]$, which contradicts $m_{p}=n^{\varepsilon_{p}}(x_{p},y_{p})<0$.

Therefore, by reducing $\varepsilon_{0}>0$ if necessary, we have that, for all $|\varepsilon|\leq\varepsilon_{0}$, the steady state $n^{\ep}$ is nonnegative. Now, as already seen in (\ref{eq:__yp_ineq}), there is $C>0$ such that, for all $x\in \R$, $
\int_{\mathbb{R}}n^{\varepsilon}(x,y)dy\leq C$. We thus deduce from  (\ref{eq:statio_state}) that $
-n^{\ep}_{xx}-n^{\ep}_{yy}-n^{\ep}\left[1-A^{2}(y-\varepsilon\theta(x))^{2}-C\right]\geq0.
$
 The maximum principle then implies
\begin{equation*}\label{n-eps-positif}
\forall|\varepsilon|\leq\varepsilon_{0},\forall (x,y)\in\mathbb{R}^2,\quad n^{\varepsilon}(x,y)
>0.
\end{equation*}

\medskip

Last, we prove the exponential control appearing in (\ref{positivite-exp-per}) and (\ref{positivite-exp-per-bis}). Let $a>0$ be given. Set 
$$
\Omega:=\left\{ (x,y)\in\mathbb{R}^{2}: 1-A^{2}(y-\varepsilon\theta(x))^{2}>-a^2\right\} .
$$
From  (\ref{eq:neps_loc_bound}), there is $N>0$ such that $0<n^{\ep}\leq N$. We now define 
\[
\overline{n}(x,y):= Ne^{a y_{0}}e^{-a|y|},\quad y_{0}:=\varepsilon_{0}||\theta||_{\infty}+\frac{1}{A}\sqrt{1+a^2}>0,
\]
so that $\overline{n}\geq n^{\varepsilon}$ in $\overline{\Omega}$.
It remains to prove that $\overline{n}\ge n^{\varepsilon}$ in $\Omega^{c}$.
Notice that, since $n^{\varepsilon}\geq0$ solves (\ref{eq:statio_state}), there
holds 
\[
\mathcal{E}n^{\varepsilon}:=-n_{xx}^{\varepsilon}-n_{yy}^{\varepsilon}-n^{\varepsilon}\left[1-A^{2}(y-\varepsilon\theta(x))^{2}\right]\leq0.
\]
Meanwhile, in $\Omega^{c}$,
$$
\mathcal{E}\overline{n}  =-a^{2}\overline{n}-\overline{n}\left[1-A^{2}(y-\varepsilon\theta(x))^{2}\right] \geq0.
$$
Due to the maximum principle, we deduce that $n\leq\overline{n}$
on $\Omega^{c}$, and thus on $\mathbb{R}^{2}$. This concludes the proof of (\ref{positivite-exp-per}) and (\ref{positivite-exp-per-bis}).
\end{proof}

\section{Construction of pulsating fronts}\label{s:pulsating}

In this section, we prove Theorem \ref{thm:pulsating} on pulsating
fronts.

\medskip

Let $\varepsilon_{0}>0$ be as in Theorem \ref{thm:steady-per} and, for $\vert\ep\vert <\ep_0$, let $n^{\ep}=n^{\ep}(x,y)$  be the periodic positive steady state provided by Theorem \ref{thm:steady-per}. Let us fix a speed $c_{0}\geq c^{*}=2\sqrt{-\lambda_0}$ and recall that $U=U(z)$ denotes the Fisher-KPP front given by (\ref{eq:FKPP_front})
and traveling at speed $c_0$. We look after a pulsating front
solution to (\ref{eq-r}) in the perturbative form
\[
u^{\varepsilon}(z,x,y)=U(z)n^{\varepsilon}(x,y)+v_{\varepsilon}(z,x,y),\quad c_{\varepsilon}=c_{0}+s_{\varepsilon},
\]
where we understand $z=x-c_{\ep}t$, meaning that the front spreads
at the perturbed speed $c_{\varepsilon}=c_{0}+s_{\varepsilon}$.
Plugging this into (\ref{eq-r}), using the steady state equation (\ref{eq:statio_state}) for $n^\ep(x,y)$ and the front equation (\ref{eq:FKPP_front})  for $U(z)$,
we are left to find $(s_{\varepsilon},v_{\varepsilon})$ satisfying
$\Ff(\varepsilon,s_{\varepsilon},v_{\varepsilon})=0$ where
\begin{align}
\Ff(\varepsilon,s,v):= & \;v_{zz}+2v_{xz}+v_{xx}+v_{yy}+(c_{0}+s)v_{z}
+sU^{\prime}(z)n^{\varepsilon}(x,y)+2U^{\prime}(z)n_{x}^{\varepsilon}(x,y)\nonumber\\
 & +v\left(1-A^{2}(y-\varepsilon\theta(x))^{2}-U(z)\int_{\mathbb{R}}n^{\varepsilon}(x,y^{\prime})dy^{\prime}-\int_{\mathbb{R}}v(z,x,y^{\prime})dy^{\prime}\right)\nonumber\\
 & -U(z)n^{\varepsilon}(x,y)\int_{\mathbb{R}}v(z,x,y^{\prime})dy^{\prime}+U(z)(1-U(z))n^{\varepsilon}(x,y)\left(\lambda_{0}+\int_{\mathbb{R}}n^{\varepsilon}(x,y')dy'\right).\label{op-deg}
\end{align}
However, since the elliptic operator appearing in the right-hand side above is degenerate in the $(z,x)$ variables, we need to consider the regularization
\begin{equation}\label{op-reg}
\Fmu(\varepsilon,s,v):=\Ff(\varepsilon,s,v)+\mu v_{xx}, \quad 0<\mu \ll 1.
\end{equation}
To prove Theorem \ref{thm:pulsating}, the very crude strategy is as follows. We first apply the Implicit Function Theorem, namely
Theorem \ref{thm:IFT}, to $\Fmu\colon\mathbb{R}\times\mathbb{R}\times\mathcal{Y}_{\mu}\to\mathcal{Z}$
where the function spaces $\mathcal{Y}_{\mu}$ and $\mathcal{Z}$
are appropriately chosen. 
This will provide a couple $( s_{\ep,\mu},v_{\ep,\mu} )\in \R\times \mathcal{Y}_\mu$ for any $\mu>0$ small enough. Then, we shall obtain $s_\ep,v_\ep$ by passing to the limit $\mu\to 0$. See Remark in subsection \ref{subsec:Spaces_Yrond_Zrond} for more details on the key ideas of the proof.

\medskip

By assumption, see Theorem \ref{thm:pulsating}, there are $\gamma>3$, $k\geq 0$ and $0\leq \delta <1$ with $k+\delta>\gamma+\frac{1}{2}$ such that $\theta$ belongs to $C^{k,\delta}(\mathbb{R})\cap C_{per}^{L}(\mathbb{R})$, and so does  $\theta^{2}$. In particular, the Fourier coefficients
of $\theta$ and $\theta^{2}$ decay at least at speed $|m|^{-(k+\delta)}$
as $|m|\to\infty$, that is
\begin{equation}
\exists K_{\theta}>0,\,\forall m\in\mathbb{Z},\; \max\left(\left|\theta_{m}\right|,\left|(\theta^{2})_{m}\right|\right) \leq\frac{K_{\theta}}{(1+|m|)^{k+\delta}},\label{eq:thetam_decay}
\end{equation}
where we denote
\[
\theta_{m}:=\frac{1}{L}\int_{0}^{L}\theta(x)e^{-\frac{2i\pi mx}{L}}dx,\qquad(\theta^{2})_{m}:=\frac{1}{L}\int_{0}^{L}\theta^{2}(x)e^{-\frac{2i\pi mx}{L}}dx.
\]

\subsection{Function spaces\label{subsec:Spaces_Yrond_Zrond}}

We first present a few notations that will be used below. For any function $f=f(z,x,y)\in C_{b}(\mathbb{R}^{3})$
such that $f(z,x,\cdot)\in L^{2}(\mathbb{R})$ and $f(z,x+L,y)=f(z,x,y)$
for all $z,x,y$, we denote
\begin{equation}
f_{j}(z,x):=\int_{\mathbb{R}}f(z,x,y)\Gamma_{j}(y)dy,\label{eq:fj_nota}
\end{equation}
that is $f_{j}$ denotes the $j$-th coordinate of $f$ along the
basis of eigenfunctions $(\Gamma_{j}=\Gamma_{j}(y))_{j\in\mathbb{N}}$.
We also define
\begin{equation}
f_{j}^{n}(z):=\frac{1}{L}\int_{0}^{L}f_{j}(z,x)e^{-\frac{2i\pi n}{L}x}dx=\frac{1}{L}\int_{0}^{L}f_{j}(z,x)e_{-n}(x)dx,\label{eq:fnj_nota}
\end{equation}
\begin{equation}
e_{n}(x):= e^{\frac{2i\pi n}{L}x}=e^{i\sigma nx},\qquad\sigma:=\frac{2\pi}{L},\,n\in\mathbb{Z},\label{eq:en_nota}
\end{equation}
that is $f_{j}^{n}(z)$ denotes the $n$-th Fourier coefficient 
of $x\mapsto f_{j}(z,x)$.

\medskip

Now, for a $\kappa\in\left(0,-\frac{1}{2}c_{0}+\frac{1}{2}\sqrt{c_{0}^{2}-4\lambda_{0}}\right)$ to be precised later, we define
\begin{equation}
\mathcal{Y}_{\mu}:=\left\{ v\in C^{2}(\mathbb{R}^{3})\left|\begin{array}{c}
v(z,x+L,y)=v(z,x,y)\quad\text{on }\mathbb{R}^{3},\vspace{7pt}\\
\exists C>0,\,\forall|\alpha|\leq2,\quad\left|D^{\alpha}v(z,x,y)\right|\leq\frac{Ce^{-\kappa|z|}}{(1+y^{2})^{2}}\quad\text{on }\mathbb{R}^{3},\vspace{7pt}\\
\exists K>0,\,\forall n\in\mathbb{Z},\,\forall j\in\mathbb{N},\,\forall k\leq2,\quad\text{there holds}\vspace{3pt}\\
|(v_{j}^{n})^{(k)}(z)|\leq\frac{Ke^{-\kappa|z|}}{(1+j)^{\beta}(1+|n|)^{\gamma}}\times\frac{1+|n|^{k}+j^{k/2}}{1+\mu n^{2}+j+|n|}\quad\text{on }\mathbb{R}\vspace{7pt}
\end{array}\right.\right\} ,\label{eq:Yrond}
\end{equation}
\begin{equation}
\mathcal{Z}:=\left\{ f\in C(\mathbb{R}^{3})\left|\begin{array}{c}
f(z,x+L,y)=f(z,x,y)\quad\text{on }\mathbb{R}^{3},\vspace{7pt}\\
\exists C>0,\quad\left|f(z,x,y)\right|\leq\frac{Ce^{-\kappa|z|}}{1+y^{2}}\quad\text{on }\mathbb{R}^{3},\vspace{7pt}\\
\exists K>0,\,\forall n\in\mathbb{Z},\,\forall j\in\mathbb{N},\quad\text{there holds}\vspace{3pt}\\
|f_{j}^{n}(z)|\leq\frac{Ke^{-\kappa|z|}}{(1+j)^{\beta}(1+|n|)^{\gamma}}\quad\text{on }\mathbb{R}
\end{array}\right.\right\} .\label{eq:Zrond}
\end{equation}
We equip the space $\mathcal{Y}_{\mu}$ with the norm
\begin{equation}
||v||_{\mathcal{Y_{\mu}}}:=\sum_{|\alpha|\leq2}\sup_{(z,x,y)\in\mathbb{R}^3}\left[\left|(1+y^{2})^{2}D^{\alpha}v(z,x,y)\right|e^{\kappa|z|}\right]+||v||_{\beta,\gamma,\mu},\label{eq:Yrond_norm}
\end{equation}
where
\[
||v||_{\beta,\gamma,\mu}:=\sum_{k=0}^{2}\sup_{n\in\mathbb{Z},j\in\mathbb{N}}\left[(1+j)^{\beta}(1+|n|)^{\gamma}\frac{1+\mu n^{2}+j+|n|}{1+|n|^{k}+j^{k/2}}\sup_{z\in\mathbb{R}}\left|(v_{j}^{n})^{(k)}(z)e^{\kappa|z|}\right|\right].
\]
We equip the space $\mathcal Z$ with the norm
\begin{equation}
||f||_{\mathcal{Z}}=\sup_{(z,x,y)\in\mathbb{R}^3}\left[\left|(1+y^{2})f(z,x,y)\right|e^{\kappa|z|}\right]+||f||{}_{\beta,\gamma},\label{eq:Zrond_norm}
\end{equation}
where
\[
||f||_{\beta,\gamma}:=\sup_{n\in\mathbb{Z},j\in\mathbb{N}}\left[(1+j)^{\beta}(1+|n|)^{\gamma}\sup_{z\in\mathbb{R}}\left|f_{j}^{n}(z)e^{\kappa|z|}\right|\right].
\]

\begin{rem}[Choice of the function spaces and overview of the proof of Theorem \ref{thm:pulsating}]\label{rem:spaces-2}
Let us comment on the spaces $\mathcal{Y}_{\mu},\mathcal{Z}$
and the two controls appearing in their definition. As in the stationary
case, i.e. Section \ref{s:steady}, the crux of the proof is to show
that $\Lmu:= D_{(s,v)}\Fmu(0,0,0)$, given by (\ref{eq:Lpar}), is
bijective from  $\R \times \mathcal{S}_\mu$ to $\mathcal{Z}$, where $\mathcal{S}_{\mu}\subset\mathcal{Y}_{\mu}$ is to be determined, that is for every fixed $f\in\mathcal{Z}$, there
is a unique $s_{\mu}\in\mathbb{R}$ and a unique $v_{\mu}\in\mathcal{S}_{\mu}$
such that $\Lmu(s_{\mu},v_{\mu})=f$. Using the controls on the y-tails
provided by the first constraint in the definition of $\mathcal{Y}_{\mu}$
and $\mathcal{Z}$, and then the $L$-periodicity in $x$, we decompose
successively $v_{\mu}$ and $f$ along the eigenfunction bases $(\Gamma_{j})_{j\in\mathbb{N}}$
and $(e_{n})_{n\in\mathbb{Z}}$ respectively, leading to (\ref{eq:vnj}),
where we denoted $v=v_{\mu}$ to ease readability. Next, the control
on $f_{j}^{n}$, i.e. the second constraint in the definition of $\mathcal{Z}$,
allows to prove the bound (\ref{eq:vnj_toProve}). 

However, the operators $\mathcal{L}_{n,j,\mu}$ defined by (\ref{eq:vnj}), (\ref{eq:Lnj}) and (\ref{eq:E_kappa}) might not be injective. Thus, in order to ensure
the uniqueness, we require that $(v_{\mu})_{j}^{n}$ belongs to a
subspace $\mathcal{S}_{n,j,\mu}$ of the departure space of $\mathcal{L}_{n,j,\mu}$.
These additional conditions lead to $v_{\mu}\in\mathcal{S}_{\mu}$,
after reconstruction of $v_{\mu}$ according to (\ref{eq:v_reconstr}).
To show that $v_{\mu}\in C_{b}^{2}(\mathbb{R}^{3})$ and $v_{\mu}$
satisfies the first control in $\mathcal{Y}_{\mu}$, we require $\beta>\frac{17}{4}$
and $\gamma>2$. This allows to apply Theorem \ref{thm:IFT} and deduce
the existence of $s_{\varepsilon,\mu},v_{\varepsilon,\mu}$ such that
$\Fmu(\varepsilon,s_{\varepsilon,\mu},v_{\varepsilon,\mu})=0$ for
any $|\varepsilon|\leq\overline{\varepsilon}_{0}$.

The next step is to let $\mu\to 0$ in $\Fmu(\varepsilon,s_{\varepsilon,\mu},v_{\varepsilon,\mu})=0$. 
However, to do so, we first require that $\overline{\varepsilon}_{0}$
may be fixed independently on $\mu$. This is actually true from the crucial observation that, despite $\Fmu$ is unbounded with respect to $\mu$, both $(\Lmu)^{-1}$ and $\Fmu-\Lmu$, see (\ref{def:Gmu}), are bounded (see subsection \ref{subsec:Constr_veps_mu}).

Last, to ensure that a subsequence of $(v_{\varepsilon,\mu})_{\mu}$
converges as $\mu\to0$, we need to redo the above proof by replacing
$\mathcal{Y}_{\mu}$ and $\mathcal{Z}$ with (\ref{eq:Yhat}) and
(\ref{eq:Zhat}) respectively. This allows to obtain $C_{b}^{3}$
regularity for $v_{\varepsilon,\mu}$, at the cost of the assumptions
$\beta>\frac{19}{4}$ and $\gamma>3$. Then, the $L$-periodicity
on $x$ and the controls on the $y$- and $z$-tails in the definition
of (\ref{eq:Yhat}) allow in some sense to compactify the domain of
definition of $v_{\varepsilon,\mu}$, so that we can adapt the proof
of the Arzelà-Ascoli theorem and conclude, see subsection \ref{subsec:ascoli}.
\end{rem}

In what follows, we will repeatedly use the following straightforward estimates:
\begin{alignat}{2}
 & \forall v\in\mathcal{Y}_\mu,\,\forall|\alpha|\leq2,\,\forall(z,x,y)\in\mathbb{R}^{3}, & \quad & |D^{\alpha}v(z,x,y)|\leq ||v||_{\mathcal{\mathcal{Y}_{\mu}}} \frac{e^{-\kappa|z|}}{(1+y^{2})^{2}},\label{eq:v_Ybound}\\
 & \forall v\in\mathcal{Y}_\mu,\,\forall k\leq2,\,\forall(n,j)\in\mathbb{Z}\times\mathbb{N},\,\forall z\in\mathbb{R}, & \quad & |(v_{j}^{n})^{(k)}(z)|\leq ||v||_{\mathcal{Y}_{\mu}}\frac{e^{-\kappa|z|}\left(1+|n|^{k}+j^{k/2}\right)}{(1+j)^{\beta}(1+|n|)^{\gamma}\left(1+\mu n^{2}+j+|n|\right)},\label{eq:vnj_Ybound}\\
 & \forall f\in\mathcal{Z},\,\forall(z,x,y)\in\mathbb{R}^{3}, & \quad & |f(z,x,y)|\leq ||f||_{\mathcal{Z}} \frac{e^{-\kappa|z|}}{1+y^{2}},\label{eq:f_Zbound}\\
 & \forall f\in\mathcal{Z},\,\forall(n,j)\in\mathbb{Z}\times\mathbb{N},\,\forall z\in\mathbb{R}, & \quad & |f_{j}^{n}(z)|\leq ||f||_{\mathcal{Z}}\frac{e^{-\kappa|z|}}{(1+j)^{\beta}(1+|n|)^{\gamma}},\label{eq:fnj_Zbound}\\
 & \forall v\in\mathcal{Y}_\mu,\,\forall(z,x)\in\mathbb{R}^{2}, & \quad & \left|\int_{\mathbb{R}}v(z,x,y)dy\right|\leq\frac{\pi}{2}||v||_{\mathcal{Y}_{\mu}}e^{-\kappa|z|}.\label{eq:intv_Ybound}
\end{alignat}
Also, we claim that there exists $K_{A}>0$, that depends
only on $A$, such that
\begin{equation}
\forall v\in\mathcal{Y}_\mu,\,\forall z\in\mathbb{R},\,\forall n\in\mathbb{Z},\qquad\left|\frac{1}{L}\int_{0}^{L}\int_{\mathbb{R}}v(z,x,y)dy e_{-n}(x)dx\right| \leq K_{A}||v||_{\mathcal{Y}_{\mu}} \frac{e^{-\kappa|z|}}{(1+|n|)^{\gamma+1}}.\label{eq:v_thetacond}
\end{equation}
Indeed, from (\ref{eq:vnj_Ybound}), we have (as usual the constant $C>0$ is independent of $z$, $x$, $y$, $j$ and $n$ but may change from line to line)
\[
\left|v_{j}^{n}(z)\right|\leq C ||v||_{\mathcal{Y}_{\mu}} \frac{e^{-\kappa|z|}}{(1+j)^{\beta}(1+|n|)^{\gamma+1}}.
\]
Let us recall that $\beta>\frac{19}{4}$
and $\gamma>3$. Thus we obtain
\[
\vert v_j(z,x)\vert \leq\sum_{n=-\infty}^{\infty}|v_{j}^{n}(z)|\leq  C ||v||_{\mathcal{Y}_{\mu}} \frac{e^{-\kappa|z|}}{(1+j)^{\beta}}.
\]
Since (\ref{eq:eigenfunc_control_L1}) holds, we deduce 
\[
\int_{\mathbb{R}}v(z,x,y)dy=\int_{\mathbb{R}}\left(\sum_{j\in\mathbb{N}}v_{j}(z,x)\Gamma_{j}(y)\right)dy=\sum_{j\in\mathbb{N}}m_{j}v_{j}(z,x),
\]
where we recall the notation $m_j:= \int_\R \Gamma_j(y)dy$. This in turn leads to
\[
\left|\frac{1}{L}\int_{0}^{L}\left(\int_{\mathbb{R}}v(z,x,y)dy\right)e_{-n}(x)dx\right|=\left|\sum_{j\in\mathbb{N}}m_{j}v_{j}^{n}(z)\right|\leq C ||v||_{\mathcal{Y}_{\mu}} \frac{e^{-\kappa|z|}}{(1+|n|)^{\gamma+1}}\sum_{j\in\mathbb{N}}\frac{|m_{j}|}{(1+j)^{\beta}},
\]
which proves the claim (\ref{eq:v_thetacond}) using again (\ref{eq:eigenfunc_control_L1}).

\begin{lem}[$\mathcal{Y}_\mu$ and $\mathcal{Z}$ are Banach]
\label{lem:YZrondBanach}For all $0<\mu<1$, the spaces $\mathcal{Y}_{\mu}$
given by (\ref{eq:Yrond}), and $\mathcal{Z}$ given by (\ref{eq:Zrond}),
are Banach spaces when equipped with their respective norm $||.||_{\mathcal{Y}_{\mu}}$
given by (\ref{eq:Yrond_norm}), and $||.||_{\mathcal{Z}}$ given
by (\ref{eq:Zrond_norm}).
\end{lem}

\begin{proof}  Let us fix $0<\mu<1$. For the sake of completeness, we give a short
proof that $\mathcal{Y}_{\mu}$ is Banach, the proof for $\mathcal{Z}$
being similar. Let $(v_{m})_{m\in\mathbb{N}}$ be a Cauchy sequence
in $\mathcal{Y}_{\mu}$. Since the embedding $\mathcal{Y}_{\mu}\hookrightarrow C_{b}^{2}(\mathbb{R}^{3})$
is continuous and $C_{b}^{2}(\mathbb{R}^{3})$ is Banach, there is
$v\in C_{b}^{2}(\mathbb{R}^{3})$ such that $||v_{m}-v||_{C_{b}^{2}(\mathbb{R}^{3})}\to0$.

Let us prove that $v\in\mathcal{Y}_{\mu}$. The $L$-periodicity in
$x$ of $v$ is obvious. Following the same arguments as in the proof
of Lemma \ref{lem:YZ_Banach}, there exists $C>0$ such that
\begin{equation}
|D^{\alpha}v(z,x,y)|\leq\frac{Ce^{-\kappa|z|}}{(1+y^{2})^{2}},\label{eq:___vC}
\end{equation}
for all $|\alpha|\leq2$ and $(z,x,y)\in\mathbb{R}^{3}$. Next, similarly
to the proof of Lemma \ref{lem:YZ_Banach}, the sequence $K_{m}:=\sum_{k=0}^{2}||v_{m}||_{\gamma,\beta,\mu}$ is bounded for all $m\in\mathbb{N}$ by some $K>0$. Since (\ref{eq:___vC})
holds, we deduce by the dominated convergence theorem that for any
$k\leq2$,
\begin{align*}
\left|(v_{j}^{n})^{(k)}(z)\right| & =\lim_{m\to\infty}\left|\frac{1}{L}\int_{0}^{L}\left(\int_{\mathbb{R}}D_{z}^{k}v_{m}(z,x,y)\Gamma_{j}(y)dy\right)e_{-n}(x)dx\right|\\
& \leq\frac{Ke^{-\kappa|z|}}{(1+j)^{\beta}(1+|n|)^{\gamma}}\times\frac{1+|n|^{k}+j^{k/2}}{1+\mu n^{2}+j+|n|}.
\end{align*}
Therefore $v\in\mathcal{Y}_{\mu}$. From there, classical arguments
(that we omit) yield that $||v_{m}-v||_{\mathcal{Y}_\mu}\to0$.
\end{proof}

We now state some preliminary results. For better readability, we
denote $yv$ the function $(z,x,y)\mapsto yv(z,x,y)$, and similarly for
$y^{2}v$.

\begin{lem}[Controlling in $\mathcal{Z}$ the $v$ terms of $\Fmu(\varepsilon,s,v)$]
\label{lem:puls_prelim_control}There exists $C>0$ such that, for
any $\mu\in(0,1)$ and $v=v(z,x,y)\in\mathcal{Y}_{\mu}$,
\begin{equation}
\max\left(||v||_{\mathcal{Z}},||D_z v||_{\mathcal{Z}},||yv||_{\mathcal{Z}},||y^{2}v||_{\mathcal{Z}}\right)\leq C||v||_{\mathcal{Y}_{\mu}},\label{eq:Z_control_noMu}
\end{equation}
\begin{equation}
\max_{|\alpha|\leq2}||D^{\alpha}v||_{\mathcal{Z}}\leq\mu^{-1}C||v||_{\mathcal{Y}_{\mu}}.\label{eq:Z_control}
\end{equation}
Also, set $\rho\geq0$ and assume $b=b(z,x)\in C_{b}(\mathbb{R}^{2})$
satisfies
\begin{equation}
\begin{cases}
b(z,x+L)=b(z,x) & \forall z,x\in\mathbb{R},\vspace{3pt}\\
\left|b_{m}(z)\right|:=\left|\frac{1}{L}\int_{0}^{L}b(z,x)e_{-m}(x)dx\right|\leq\frac{K_{b}}{(1+|m|)^{\gamma+\rho}} & \forall m\in\mathbb{Z},\,\forall z\in\mathbb{R},
\end{cases}\label{eq:bm_cond}
\end{equation}
for some $K_{b}>0$. Then there are $C_{\rho},C_{\rho}^{\prime}>0$
such that, for any $\mu\in(0,1)$ and $v=v(z,x,y)\in\mathcal{Y}_{\mu}$,
\begin{equation}
||bv||_{\mathcal{Z}}\leq\begin{cases}
(\mu^{-1}C_{\rho}K_{b}+||b||_{L^{\infty}(\mathbb{R}^{2})})||v||_{\mathcal{Y}_{\mu}} & \text{if }\rho=0,\vspace{5pt}\\
(C_{\rho}K_{b}+||b||_{L^{\infty}(\mathbb{R}^{2})})||v||_{\mathcal{Y}_{\mu}}  & \text{if }\rho>0,
\end{cases}\label{eq:bv_bound}
\end{equation}
and
\begin{equation}
||byv||_{\mathcal{Z}}\leq(C_{\rho}^{\prime}K_{b}+||b||_{L^{\infty}(\mathbb{R}^{2})})||v||_{\mathcal{Y}_{\mu}} \qquad\text{if }\rho>1/2.\label{eq:ybv_bound}
\end{equation}
\end{lem}

\begin{proof}
Fix $\mu\in(0,1)$. By definition of $\mathcal{Y}_{\mu}$, for any
function $w\in\{D^{\alpha}v,yv,y^{2}v,bv,byv\}$, it is clear that
$w$ is $L$-periodic in $x$ and satisfies, thanks to (\ref{eq:v_Ybound}),
\[
|w(z,x,y)|\leq\frac{C||v||_{\mathcal{Y}_{\mu}}e^{-\kappa|z|}}{1+y^{2}},\qquad\forall (z,x,y)\in\mathbb{R}^3,
\]
with $C=||b||_{L^{\infty}(\mathbb{R}^{2})}$ if $w\in\{bv,byv\}$,
and $C=1$ otherwise. Thus in order to prove (\ref{eq:Z_control_noMu}) and (\ref{eq:Z_control}), it is enough to control $||w||_{\beta,\gamma}$ for each $w\in\{D^{\alpha}v,yv,y^{2}v,bv,byv\}$.

If $w=yv$, by virtue of (\ref{eq:yGamma}), we have
\begin{align*}
(yv)_{j}^{n}(z) & =\frac{1}{L}\int_{0}^{L}\left(\int_{\mathbb{R}}v(z,x,y)y\Gamma_{j}(y)dy\right)e_{-n}(x)dx=\begin{cases}
p_{j}^{+}v_{j+1}^{n}(z)+p_{j}^{-}v_{j-1}^{n}(z) & j\geq1,\\
p_{0}^{+}v_{1}^{n}(z) & j=0.
\end{cases}
\end{align*}
From (\ref{eq:vnj_Ybound}), we thus obtain $||yv||_{\beta,\gamma}\leq C ||v||_{\mathcal{Y}_{\mu}}$ for some $C>0$. One can readily check that the same is true for $w=y^{2}v$.

Now, set $|\alpha|\leq2$ and $w=D^{\alpha}v$. If $w=v$ or $w=D_z v$, then from (\ref{eq:vnj_Ybound}),
we deduce $||w||_{\beta,\gamma}\leq ||v||_{\beta,\gamma,\mu} \leq ||v||_{\mathcal{Y}_\mu}$. If $w=D_{z}^{2}v$, then (\ref{eq:vnj_Ybound})
yields $||D_{z}^{2}v||_{\beta,\gamma}\leq\mu^{-1}||v||_{\mathcal{Y}_{\mu}}$
since $0<\mu<1$. Now, consider $w=D_{x}^{k}v$ with $k\in\{1,2\}$.
Then by integration by parts there holds
\begin{align*}
(D_{x}^{k}v)_{j}^{n}(z) & =\frac{1}{L}\int_{0}^{L}D_{x}^{k}\left(\int_{\mathbb{R}}v(z,x,y)\Gamma_{j}(y)dy\right)e_{-n}(x)dx=\left(in\sigma\right)^{k}v_{j}^{n}(z),
\end{align*}
which, thanks to (\ref{eq:vnj_Ybound}) implies $||D_{x}v||_{\beta,\gamma}\leq\sigma||v||_{\mathcal{Y}_{\mu}}$
and $||D_{x}^{2}v||_{\beta,\gamma}\leq\mu^{-1}\sigma||v||_{\mathcal{Y}_{\mu}}$.
As for $w=D_{y}^{k}v$ with $k\leq2$, the proof is similar to $w=y^{k}v$
by using (\ref{eq:dGamma}) instead of (\ref{eq:yGamma}). Therefore (\ref{eq:Z_control}) holds
for $D^{\alpha}\in\{D_{z}^{k},D_{x}^{k},D_{y}^{k}\}$ with $0\leq k\leq2$.
The proof for the cross derivatives $D^{\alpha}\in\{D_{xy},D_{xz},D_{yz}\}$
results from a combination of the above arguments. Therefore we proved (\ref{eq:Z_control_noMu})---(\ref{eq:Z_control}).

Next, let us consider $w=bv$. Since (\ref{eq:bm_cond}) holds with
$\gamma+\rho\geq\gamma>3>1$, the Fourier series of $b(z,\cdot)$
converges uniformly on $\mathbb{R}$ and we have pointwise
\[
b(z,x)=\sum_{m=-\infty}^{\infty}b_{m}(z)e_{m}(x).
\]
This leads to
\begin{align*}
(bv)_{j}^{n}(z) & =\frac{1}{L}\int_{0}^{L}b(z,x)\left(\int_{\mathbb{R}}v(z,x,y)\Gamma_{j}(y)dy\right)e_{-n}(x)dx\\
 & =\frac{1}{L}\int_{0}^{L} \left( \sum_{m=-\infty}^{\infty}b_{m}(z)e_{m}(x)\right) v_j(z,x) e_{-n}(x)dx\\
 & =\sum_{m=-\infty}^{\infty}b_{m}(z)v_{j}^{n-m}(z)=\sum_{m=-\infty}^{\infty}b_{n-m}(z)v_{j}^{m}(z).
\end{align*}
Let us first assume that $\rho=0$. The controls (\ref{eq:vnj_Ybound})
and (\ref{eq:bm_cond}) then yield
\begin{align*}
\left|(bv)_{j}^{n}(z)\right| & \leq\sum_{m=-\infty}^{\infty}|b_{n-m}(z)|\times\left|v_{j}^{m}(z)\right|\\
 & \leq K_{b}||v||_{\mathcal{Y}_{\mu}}\sum_{m=-\infty}^{\infty}\frac{1}{(1+|n-m|)^{\gamma}}\times\frac{e^{-\kappa|z|}}{(1+j)^{\beta}(1+|m|)^{\gamma}}\times\frac{1}{1+\mu m^{2}+j+|m|}.
\end{align*}
From there, one can readily check, by studying all possible cases on the signs of $n$, $m$ and $n-m$, that 
\[
\frac{1}{(1+|n-m|)(1+|m|)}\leq\frac{1}{1+|n|},\qquad\forall m,n\in\mathbb{Z}.
\]
Therefore
\begin{align*}
\left|(bv)_{j}^{n}(z)\right| & \leq\frac{K_{b}||v||_{\mathcal{Y}_{\mu}}e^{-\kappa|z|}}{(1+j)^{\beta}(1+|n|)^{\gamma}}\sum_{m=-\infty}^{\infty}\frac{1}{1+\mu m^{2}+j+|m|}\\
 & \leq\frac{\mu^{-1}K_{b}||v||_{\mathcal{Y}_{\mu}}e^{-\kappa|z|}}{(1+j)^{\beta}(1+|n|)^{\gamma}}\sum_{m=-\infty}^{\infty}\frac{1}{1+m^{2}},
\end{align*}
which gives (\ref{eq:bv_bound}) for $\rho=0$. Meanwhile, if $\rho>0$, similar calculations yield
\begin{align*}
\left|(bv)_{j}^{n}(z)\right| & \leq\frac{K_{b}||v||_{\mathcal{Y}_{\mu}}e^{-\kappa|z|}}{(1+j)^{\beta}(1+|n|)^{\gamma}}\sum_{m=-\infty}^{\infty}\frac{1}{(1+|n-m|)^{\rho}}\times\frac{1}{1+\mu m^{2}+j+|m|}\\
 & \leq\frac{K_{b}||v||_{\mathcal{Y}_{\mu}}e^{-\kappa|z|}}{(1+j)^{\beta}(1+|n|)^{\gamma}}\sum_{m=-\infty}^{\infty}\frac{1}{(1+|n-m|)^{\rho}(1+|m|)}.
\end{align*}
Since $\rho>0$, H\"{o}lder inequality shows that the sum of the infinite series above is bounded by some  $C_\rho>0$ which is independent of $n\in\Z$. This gives (\ref{eq:bv_bound}) for $\rho >0$.

Finally, let us prove (\ref{eq:ybv_bound}). With (\ref{eq:yGamma}),
we obtain that, for some $C_{A,\beta}>0$,
\begin{align*}
\left|(byv)_{j}^{n}(z)\right| & \leq\frac{K_{b}||v||_{\mathcal{Y}_{\mu}}e^{-\kappa|z|}}{(1+j)^{\beta}(1+|n|)^{\gamma}}\sum_{m=-\infty}^{\infty}\frac{1}{(1+|n-m|)^{\rho}}\times\frac{C_{A,\beta}\sqrt{j}}{1+j+|m|}\\
 & \leq\frac{K_{b}||v||_{\mathcal{Y}_{\mu}}e^{-\kappa|z|}}{(1+j)^{\beta}(1+|n|)^{\gamma}}\sum_{m=-\infty}^{\infty}\frac{1}{(1+|n-m|)^{\rho}}\times\frac{ C_{A,\beta}}{\sqrt 2 (1+|m|)^{\frac 12}},
\end{align*}
as easily seen by studying the maximum of $j\in[0,+\infty)\mapsto \frac{\sqrt j}{1+j+\vert m\vert}$. Since $\rho>\frac 12$, H\"{o}lder inequality shows that the sum of the infinite series above is bounded by some  $C_\rho '>0$ which is independent of $n\in\Z$. This gives (\ref{eq:ybv_bound}).
\end{proof}

To conclude this subsection, we prove that, by taking $|\varepsilon|$
possibly smaller, we obtain some estimates on the steady state $n^{\varepsilon}=n^\ep(x,y)$ in
the $\mathcal{Y}_{\mu},\mathcal{Z}$ norms. For better readability,
we denote $e^{-\kappa|z|}h$ the function $(z,x,y)\mapsto e^{-\kappa|z|}h(x,y)$
and similarly for $e^{-\kappa|z|}h_{x}$.

\begin{lem}[The steady state $n^{\varepsilon}$ when $\theta\in C_{per}^{L}(\mathbb{R})$
further satisfies (\ref{eq:thetam_decay})]
\label{lem:neps_Yast}Fix $\beta>\frac{17}{4}$ and $\gamma>2$. Let the conditions of Theorem
\ref{thm:steady-per} hold.
Assume further that $\theta\in C_{per}^{L}(\mathbb{R})$ satisfies (\ref{eq:thetam_decay}). Then,
there is $\varepsilon_{0}^{*}>0$ such that, for any $|\varepsilon|\leq\varepsilon_{0}^{*}$,
\[
\text{there is a unique }n^{\varepsilon}\in Y^{*}\text{ such that }n^{\varepsilon}\text{ solves }(\ref{eq:statio_state}),
\]
where the function space $Y^{*}$ is given by (\ref{eq:Yast}). Additionally, we have
$||n^{\varepsilon}-n^{0}||_{Y^{*}}\to 0$, as $\ep\to 0$, 
where $||\cdot||_{Y^{*}}$ is the norm given by (\ref{eq:Yast_norm}). Finally, there are $K_\sigma>0$ (depending only on $\sigma=\frac{2\pi}L$), $K_\kappa>0$ (depending only on $\kappa$) and $K_A>0$ (depending only on $A$) such that, for any $h\in Y^{*}$,
\begin{align}
||e^{-\kappa|z|}h||_{\mathcal{Z}} & \leq||h||_{Y^{*}},\label{eq:hZ_Yast}\\
||e^{-\kappa|z|}h_{x}||_{\mathcal{Z}} & \leq K_\sigma ||h||_{Y^{*}},\label{eq:hxZ_Yast}\\
||e^{-\kappa|z|}h||_{\mathcal{Y}_{\mu}} & \leq K_\kappa ||h||_{Y^{*}},\qquad \forall 0<\mu<1,\label{eq:hY_Yast}\\
\left|\int_{\mathbb{R}}h(x,y)dy\right| & \leq\frac{\pi}{2}||h||_{Y^{*}} \qquad\quad \forall x\in \R, \label{eq:inth_bound}\\
\left|\frac{1}{L}\int_{0}^{L}\left(\int_{\mathbb{R}}h(x,y)dy\right)e_{-n}(x)dx\right| & \leq\frac{K_{A}||h||_{Y^{*}}}{(1+|n|)^{\gamma+2}} \qquad\forall n\in\mathbb{Z}.\label{eq:h_thetacond}
\end{align}
\end{lem}

\begin{proof}
In the context of this proof, for any function $h=h(x,y)\in C_{b}(\mathbb{R}^{2})$
such that $h(x,\cdot)\in L^{2}(\mathbb{R})$ and $h(x+L,y)=h(x,y)$
for all $(x,y)\in\mathbb{R}^{2}$, we denote
\[
h_{j}(x):=\int_{\mathbb{R}}h(x,y)\Gamma_{j}(y)dy, \qquad h_{j}^{n}:=\frac{1}{L}\int_{0}^{L}h_{j}(x)e_{-n}(x)dx,
\]
that is $h_{j}^{n}$ is the $n$-th Fourier coefficient of $h_{j}$,
which is the $j$-th coordinate of $h$ along the basis of eigenfunctions
$(\Gamma_{j})_{j\in\mathbb{N}}$. We now define
\begin{equation}
Y^{*}:=\left\{ h\in C^{2}(\mathbb{R}^{2})\left|\begin{array}{c}
h(x+L,y)=h(x,y),\qquad\forall x,y\in\mathbb{R},\vspace{7pt}\\
\exists C>0,\,\forall |\alpha|\leq 2, \quad \left|D^\alpha h(x,y)\right| \leq\frac{C}{(1+y^{2})^{2}}\quad\text{ on }\mathbb{R}^{2},\vspace{7pt}\\
\exists K>0,\,\forall n\in\mathbb{Z},\,\forall j\in\mathbb{N},\quad\left|h_{j}^{n}\right|\leq\frac{K}{(1+j)^{\beta}(1+|n|)^{\gamma}} \times \frac{1}{1+j+n^2}
\end{array}\right.\right\} ,\label{eq:Yast}
\end{equation}
\[
Z^{*}:=\left\{ f\in C(\mathbb{R}^{2})\left|\begin{array}{c}
f(x+L,y)=f(x,y),\qquad\forall x,y\in\mathbb{R},\vspace{7pt}\\
\exists C>0,\quad\left|f(x,y)\right|\leq\frac{C}{1+y^{2}}\quad\text{ on }\mathbb{R}^{2},\vspace{7pt}\\
\exists K>0,\,\forall n\in\mathbb{Z},\,\forall j\in\mathbb{N},\quad\left|f_{j}^{n}\right|\leq\frac{K}{(1+j)^{\beta}(1+|n|)^{\gamma}}
\end{array}\right.\right\} ,
\]
\begin{equation}
||h||_{Y^{*}}:=\sum_{|\alpha|\leq2}\sup_{(x,y)\in\mathbb{R}^{2}}\left|(1+y^{2})^{2}D^{\alpha}h(x,y)\right|+\sup_{n\in\mathbb{Z},j\in\mathbb{N}}\left[(1+j)^{\beta}(1+|n|)^{\gamma}(1+j+n^2)\left|h_{j}^{n}\right|\right],\label{eq:Yast_norm}
\end{equation}
\[
||f||_{Z^{*}}:=\sup_{(x,y)\in\mathbb{R}^{2}}\left|(1+y^{2})f(x,y)\right|+\sup_{n\in\mathbb{Z},j\in\mathbb{N}}\left[(1+j)^{\beta}(1+|n|)^{\gamma}\left|f_{j}^{n}\right|\right].
\]
The proof of Lemma \ref{lem:neps_Yast} relies on applying the Implicit
Function Theorem, namely Theorem \ref{thm:IFT}, to the function $\F=\F(\varepsilon,h)$
defined by (\ref{Fstate}). Firstly, adapting the proof of Lemmas
\ref{lem:YZ_Banach} and \ref{lem:prelim_controls}, one can readily
check that $Y^{*},Z^{*}$ are Banach spaces, that
$\F:\R \times Y^{*}\to Z^{*}$ is well-defined, and  that conditions
$\ref{enu:continuite}$---$\ref{enu:__IFT_diff_exists}$ of Theorem \ref{thm:IFT}
are satisfied, with $D_{h}\F(0,0)=\L$ given by (\ref{eq:Lstate}).
It remains to prove that $\L:Y^{*}\to Z^{*}$ is bijective. Following
the same procedure as in subsection \ref{ss:checking}, we have that
$h_{j},f_{j}$ satisfy (\ref{eq:h_i})---(\ref{eq:h_0}). We now use the Fourier coefficients: for $n\in\mathbb{Z}$, we multiply equations (\ref{eq:h_i})---(\ref{eq:h_0}) by $\frac{1}{L}e_{-n}(x)$
and integrate over $x\in [0,L]$. We obtain 
\begin{equation}
\begin{cases}
\left(-n^{2}\sigma^{2}-(\lambda_{j}-\lambda_{0})\right)h_{j}^{n}=f_{j}^{n}, & j\geq1,\vspace{5pt}\\
\left(-n^{2}\sigma^{2}+\lambda_{0}\right)h_{0}^{n}=f_{0}^{n}+\eta \sum_{\ell =1}^{+\infty} m_\ell  h_\ell ^n, & j=0.
\end{cases}\label{eq:___hjn}
\end{equation}
For $j\geq 1$, since $0<\lambda_j-\lambda_0$, we see that, for any $n\in \Z$,  there is a unique $h_{j}^{n}\in\mathbb{C}$ solving the first equation in (\ref{eq:___hjn}). Since
$f\in Z^*$, we have 
\[
|f_{j}^{n}|\leq\frac{||f||_{Z^{*}}}{(1+j)^{\beta}(1+|n|)^{\gamma}},
\]
which leads to 
\[
|h_{j}^{n}|=\frac{\vert f_j^n\vert}{n^2\sigma^2+2jA}\leq \frac{K||f||_{Z^{*}}}{(1+j)^{\beta}(1+|n|)^{\gamma}}\times \frac{1}{1+j+n^2},
\]
for some $K=K(A,L)>0$. From there, in view of (\ref{eq:eigenfunc_control_L1}) and $\beta>\frac{17}{4}>\frac{5}{4}$, the right-hand side of the second equation of (\ref{eq:___hjn}) is well-defined, and bounded by $M||f||_{Z^\ast} (1+|n|)^{-\gamma}$ for some $M>0$ independent of $n$. Therefore we obtain
\[
|h_0^n| \leq \frac{K||f||_{Z^\ast}}{(1+|n|)^\gamma(1+n^2)},
\]
by taking $K$ possibly larger. It remains to reconstruct $h$ and prove that it belongs to $Y^\ast$. Since $\gamma>2>1$, we
have for any $0\leq k\leq2$
\[
h_{j}^{(k)}(x)=\sum_{n\in\mathbb{Z}}h_{j}^{n}(i\sigma n)^{k}e_{n}(x),
\]
from which we deduce
\[
||h_{j}^{(k)}||_{\infty}\leq\frac{K||f||_{Z^{*}}}{(1+j)^{\beta}}\sum_{n\in\mathbb{Z}}\frac{(\sigma|n|)^{k}}{(1+|n|)^{\gamma+2}}\leq\frac{C||f||_{Z^{*}}}{(1+j)^{\beta}},
\]
for some $C=C(A,L,\gamma)>0$. In other words, we obtain the estimates
playing the roles of  (\ref{eq:h_i_bound})---(\ref{eq:hi_xx_bound}). Then, like the rest of the proof in subsection \ref{ss:checking},
we prove that $h\in Y^{*}$ since $\beta>\frac{17}{4}$. Thus $\L$ is bijective. Finally, we apply Theorem \ref{thm:IFT},
which leads to the existence of $\varepsilon_{0}^{*}>0$ such that,
for any $|\varepsilon|\leq\varepsilon_{0}^{*}$, there exists a unique
function $h^{\varepsilon}\in Y^{*}$ such that $\F(\varepsilon,h^{\varepsilon})=0$,
with $||h^{\varepsilon}||_{Y^{*}}\to0$ as $\varepsilon\to0$. Since
$n^{0}\in Y^{*}$, we deduce that $n^{\varepsilon}(x,y)=n^{0}(y)+h^{\varepsilon}(x,y)\in Y^{*}$
and $n^{\varepsilon}$ solves (\ref{eq:statio_state}).

To conclude, (\ref{eq:hZ_Yast})---(\ref{eq:hY_Yast}) simply follow
from the definitions of $\mathcal{Z},\mathcal{Y}_{\mu},Y^{*}$, given that $(h_{x})_{j}^{n}=(in\sigma)h_{j}^{n}$. Meanwhile,
(\ref{eq:inth_bound})---(\ref{eq:h_thetacond})  is proved in the
same way as (\ref{eq:intv_Ybound})---(\ref{eq:v_thetacond}).
\end{proof}

\subsection{Checking assumptions $\ref{enu:continuite}$ and $\ref{enu:__IFT_diff_exists}$
of Theorem \ref{thm:IFT}\label{subsec:Front_IFT_i_ii}}

For the rest of this section, we assume that $|\varepsilon|\leq\varepsilon_{0}^{*}$,
where $\varepsilon_{0}^{*}$ is obtained from Lemma \ref{lem:neps_Yast}.
We also recall that  $\mu\in(0,1)$. Equipped
with the above spaces $\mathcal{Y}_{\mu}$ and $\mathcal{Z}$, we
thus consider
\begin{align*}
\Fmu(\varepsilon,s,v)= & \;v_{zz}+2v_{xz}+(1+\mu)v_{xx}+v_{yy}+(c_{0}+s)v_{z}
+sU^{\prime}(z)n^{\varepsilon}(x,y)+2U^{\prime}(z)n_{x}^{\varepsilon}(x,y)\\
 & +v\left(1-A^{2}(y-\varepsilon\theta(x))^{2}-U(z)\int_{\mathbb{R}}n^{\varepsilon}(x,y^{\prime})dy^{\prime}-\int_{\mathbb{R}}v(z,x,y^{\prime})dy^{\prime}\right)\\
 & -U(z)n^{\varepsilon}(x,y)\int_{\mathbb{R}}v(z,x,y^{\prime})dy^{\prime}+U(z)(1-U(z))n^{\varepsilon}(x,y)\left(\lambda_{0}+\int_{\mathbb{R}}n^{\varepsilon}(x,y')dy'\right).
\end{align*}
Recall that $n^{\varepsilon}(x,y)=n^{0}(y)$  when $\ep=0$ and $\int _\R n^0(y')dy'=-\lambda_0$. Consequently $\Fmu(0,0,0)=0$. 

\begin{proof}[Checking assumptions \ref{enu:continuite} and \ref{enu:__IFT_diff_exists}
of Theorem \ref{thm:IFT}]
Fix $\mu\in(0,1)$. We first prove that $\Fmu\colon\mathcal{Y}_{\mu}\to\mathcal{Z}$
is well-defined and continuous at $(0,0,0)$. Since all terms of $\Fmu(\varepsilon,s,v)$
are obviously $L$-periodic in $x$, and since $\Fmu(0,0,0)=0$, it
suffices to prove that each term of $\Fmu(\varepsilon,s,v)$ tends
to zero in the norm $||\cdot||_{\mathcal{Z}}$ as $|\varepsilon|+|s|+||v||_{\mathcal{Y}_{\mu}}\to0$.
Firstly, Lemma \ref{lem:puls_prelim_control} and the fact that $\theta$
satisfies (\ref{eq:thetam_decay}) imply 
\[
\exists C>0,\,\forall\mu\in(0,1),\,\forall v\in\mathcal{Y}_{\mu},\qquad||w||_{\mathcal{Z}}\leq\mu^{-1}C||v||_{\mathcal{Y}_{\mu}}\xrightarrow[v\to0]{}0,
\]
for any $w\in\{D^{\alpha}v,y^{2}v,y\theta v,\theta^{2}v\}$ and $|\alpha|\leq2$.
Next, let us recall that $v$ satisfies (\ref{eq:intv_Ybound})---(\ref{eq:v_thetacond}),
and from Lemma (\ref{lem:neps_Yast}), $n^{\varepsilon}$ satisfies
(\ref{eq:inth_bound})---(\ref{eq:h_thetacond}). As a result, since
$|U(z)|\leq1$ the functions $U(z)\int_{\mathbb{R}}n^{\varepsilon}(x,y')dy'$
and $\int_{\mathbb{R}}v(z,x,y')dy'$ are uniformly bounded and satisfy
(\ref{eq:bm_cond}) with $\rho=1$ and $K_{b}=K_{A}$. From Lemma
\ref{lem:puls_prelim_control}, we thus deduce
\[
\left\Vert v(z,x,y)U(z)\int_{\mathbb{R}}n^{\varepsilon}(x,y^{\prime})dy^{\prime}\right\Vert _{\mathcal{Z}}\leq\left(C_{1}K_{A}+\frac{\pi}{2}\right)||n^{\varepsilon}||_{Y^{*}}||v||_{\mathcal{Y}_{\mu}}\xrightarrow[v\to0]{}0,
\]
\begin{equation}
\left\Vert v(z,x,y)\int_{\mathbb{R}}v(z,x,y^{\prime})dy^{\prime}\right\Vert _{\mathcal{Z}}\leq\left(C_{1}K_{A}+\frac{\pi}{2}\right)||v||_{\mathcal{Y}_{\mu}}^{2}\xrightarrow[v\to0]{}0.\label{eq:___v_intv}
\end{equation}
We now look at the term $U(z)n^{\varepsilon}(x,y)\int_{\mathbb{R}}v(z,x,y')dy^{\prime}$.
Since $|U(z)|\leq1$, $n^{\varepsilon}$ satisfies (\ref{eq:hY_Yast}),
and $v$ satisfies (\ref{eq:v_thetacond}), we have
\[
U(z)n^{\varepsilon}(x,y)\int_{\mathbb{R}}v(z,x,y^{\prime})dy^{\prime}=\underbrace{n^{\varepsilon}(x,y)e^{-\kappa|z|}}_{\in\mathcal{Y}_{\mu}}\times\underbrace{U(z)\int_{\mathbb{R}}v(z,x,y^{\prime})e^{\kappa|z|}dy^{\prime}.}_{\text{satisfies (\ref{eq:bm_cond})}\text{ with }\rho=1,\,K_{b}=K_{A}}
\]
Thus, thanks to (\ref{eq:intv_Ybound}) and (\ref{eq:bv_bound}),
we have
\begin{align*}
\left\Vert U(z)n^{\varepsilon}(x,y)\int_{\mathbb{R}}v(z,x,y^{\prime})dy^{\prime}\right\Vert _{\mathcal{Z}} & \leq\left(C_{1}K_{A}+\frac{\pi}{2}\right)||v||_{\mathcal{Y}_{\mu}}||n^{\varepsilon}e^{-\kappa|z|}||_{\mathcal{Y}_{\mu}}\\
 & \leq K_\kappa \left(C_{1}K_{A}+\frac{\pi}{2}\right)||v||_{\mathcal{Y}_{\mu}}||n^{\varepsilon}||_{Y^{*}}\xrightarrow[v\to0]{}0.
\end{align*}
Next, it is well-known that, since $\kappa<-\frac{1}{2}c_{0}+\frac{1}{2}\sqrt{c_{0}^{2}-4\lambda_{0}}$,
\[
\exists C_{U}>0,\,\forall z\in\mathbb{R},\quad U(1-U)(z),|U^{\prime}(z)|\leq C_{U}e^{-\kappa|z|}.
\]
Therefore, from (\ref{eq:hZ_Yast})---(\ref{eq:hY_Yast}), there
holds
\[
||sU^{\prime}(z)n^{\varepsilon}(x,y)||_{\mathcal{Z}}\leq C_{U}\,|s|\,||e^{-\kappa|z|}n^{\varepsilon}(x,y)||_{\mathcal{Z}}\leq C_{U}\,|s|\,||n^{\varepsilon}||_{Y^{*}}\xrightarrow[s\to0]{}0,
\]
\[
||U^{\prime}(z)n_{x}^{\varepsilon}(x,y)||_{\mathcal{Z}}=||U^{\prime}(z)(n^{\varepsilon}-n^{0})_{x}(x,y) ||_{\mathcal{Z}}\leq C_{U}K_\sigma ||n^{\varepsilon}-n^{0}||_{Y^{*}}\xrightarrow[\varepsilon\to0]{}0.
\]
Finally, setting 
\[
b_{\varepsilon}(z,x):= U(z)(1-U(z))e^{\kappa|z|}\left(\lambda_{0}+\int_{\mathbb{R}}n^{\varepsilon}(x,y')dy^{\prime}\right)=U(z)(1-U(z))e^{\kappa|z|}\int_{\mathbb{R}}(n^{\varepsilon}(x,y')-n^{0}(y'))dy^{\prime},
\]
we have, since (\ref{eq:inth_bound})---(\ref{eq:h_thetacond}) holds,
that $||b_{\varepsilon}||_{L^{\infty}(\mathbb{R}^{2})}\leq C_{U}\frac{\pi}{2}||n^{\varepsilon}-n^{0}||_{Y^{*}}$
and satisfies (\ref{eq:bm_cond}) with $\rho=2$ and $K_{b_{\varepsilon}}=K_{A}C_{U}||n^{\varepsilon}-n^{0}||_{Y^{*}}$.
Therefore from (\ref{eq:bv_bound}) we deduce
\[
||b_{\varepsilon}(z,x)e^{-\kappa|z|}n^{\varepsilon}(x,y)||_{\mathcal{Y}_{\mu}}\leq\left(C_{2}K_{b_{\varepsilon}}+||b_{\varepsilon}||_{\infty}\right)||n^{\varepsilon}||_{Y^{*}}\underset{\varepsilon\to0}{=}||n^{\varepsilon}||_{Y^{*}}\times o(1).
\]
Therefore $\Fmu$ is well-defined and continuous at $(0,0)$.

We now compute $D_{(s,v)}\Fmu(0,0,0)$, that is the Fréchet derivative
of $\Fmu$ along the $(s,v)$ variables at point $(0,0,0)$. We have
$\Fmu(0,s,v)=\Lmu(s,v)+\mathcal{R}(s,v)$ where $\mathcal{R}(s,v)=sv_{z}-v\int_{\mathbb{R}}v(z,x,y^{\prime})dy^{\prime}$ and
\begin{equation}
\begin{aligned}
\Lmu(s,v) & = v_{zz}+2v_{xz}+(1+\mu)v_{xx}+v_{yy}+c_{0}v_{z}+sU^{\prime}(z)n^{0}(y)\\
& \quad +v\left(1-A^{2}y^{2}+\lambda_{0}U(z)\right)-U(z)n^{0}(y)\int_{\mathbb{R}}v(z,x,y^{\prime})dy^{\prime}.\label{eq:Lpar}
\end{aligned}
\end{equation}
We readily check from (\ref{eq:Z_control_noMu}) and (\ref{eq:___v_intv})
that $\mathcal{R}(s,v)=o\left(|s|+||v||_{\mathcal{Y}_{\mu}}\right)$.
The continuity of $\Lmu$ is a consequence of the controls obtained
above. Consequently, $D_{(s,v)}\Fmu(0,0,0)=\Lmu$. It remains to prove
the continuity of $D_{(s,v)}\Fmu$ around $(0,0,0)$. This results
from similar arguments as above. Details are omitted.
\end{proof}

\subsection{Bijectivity of $\protect\Lmu$\label{subsec:Front_bijectivity}}

In this subsection we prove that, if $\mu>0$ is small enough, $\Lmu$
is bijective from $\mathbb{R}\times\mathcal{S}_{\mu}$ to $\mathcal{Z}$,
where $\mathcal{S}_{\mu}$ is a subset of $\mathcal{Y}_{\mu}$ that
will be determined later. We proceed by analysis and synthesis. Fix
$f\in\mathcal{Z}$, and assume there exist $(s,v)\in\mathbb{R}\times\mathcal{Y}_{\mu}$
such that $\Lmu(s,v)=f$. Naturally, $s$ and $v$ depend {\it a priori}
on $\mu$, but to ease the readability we shall omit this dependence in the notations.

\subsubsection{Decoupling in $x$ and $y$}

Thanks to (\ref{eq:v_Ybound}) and (\ref{eq:f_Zbound}), we have $v(z,x,\cdot),f(z,x,\cdot)\in L^{2}(\mathbb{R})$
for all $z,x\in\mathbb{R}$. Since the family of eigenfunctions $(\Gamma_{j})_{j\in\mathbb{N}}$
of Proposition \ref{prop:basis_eigenfunctions} forms a Hilbert basis
of $L^{2}(\mathbb{R})$, we can write
\begin{equation}
v(z,x,y)=\sum_{j=0}^{\infty}v_{j}(z,x)\Gamma_{j}(y),\qquad f(z,x,y)=\sum_{j=0}^{\infty}f_{j}(z,x)\Gamma_{j}(y),\label{eq:pulsL2_decomp}
\end{equation}
where we used the notation (\ref{eq:fj_nota}) for $v_{j}$ and $f_{j}$.
Since $(v,f)\in\mathcal{Y}_\mu\times\mathcal{Z}$, all functions $v_{j},f_{j}$
are $L$-periodic in $x$, we may compute their Fourier coefficients
in $x$:
\begin{equation}
v_{j}(z,x)=\sum_{n\in\mathbb{Z}}v_{j}^{n}(z)e_{n}(x),\qquad f_{j}(z,x)=\sum_{n\in\mathbb{Z}}f_{j}^{n}(z)e_{n}(x),\label{eq:pulsFourier_decomp}
\end{equation}
where we used the notation (\ref{eq:fnj_nota})---(\ref{eq:en_nota})
for $v_{j}^{n}$, $f_{j}^{n}$, and $e_{n}$. Note that the equalities
(\ref{eq:pulsL2_decomp})---(\ref{eq:pulsFourier_decomp}) correspond,
\textit{a priori}, to a convergence of the series in the $L^{2}(\mathbb{R})$
and $L^{2}(0,L)$ norms respectively. However, since $\gamma>3>1$, we deduce from (\ref{eq:vnj_Ybound})
and (\ref{eq:fnj_Zbound}) that equalities in (\ref{eq:pulsFourier_decomp})
hold pointwise. Additionally, $v_{j}\in C_{b}^{2}(\mathbb{R}^{2})$
and $f_{j}\in C_{b}(\mathbb{R})$ with
\[
||f_{j}||_{L^{\infty}(\mathbb{R}^{2})}\leq\frac{||f||_{\mathcal{Z}}}{(1+j)^{\beta}},
\]
and the pointwise equality
\[
D_{z}^{p}D_{x}^{q}v_{j}(z,x)=\sum_{n\in\mathbb{Z}} (v_{j}^{n})^{(p)}(z)(i\sigma n)^{q}e_{n}(x),\qquad(p+q\leq2),
\]
which leads to 
\[
||D_{z}^{p}D_{x}^{q}v_{j}||_{L^{\infty}(\mathbb{R}^{2})}\leq\frac{C||v||_{\mathcal{Y}_\mu}}{(1+j)^{\beta}},\qquad(p+q\leq2),
\]
for some $C=C(\gamma)$. Next, since $\beta>\frac{19}{4}>\frac{5}{4}$
and (\ref{eq:eigenfunc_control_inf}) holds, the series in (\ref{eq:pulsL2_decomp})
are also normally convergent, which leads to pointwise equalities
in (\ref{eq:pulsL2_decomp}). Additionally, since $\beta>\frac{19}{4}>\frac{9}{4}$, with (\ref{eq:dGamma_control_inf})
we have the following pointwise equality:
\[
D_{z}^{p}D_{x}^{q}D_{y}^{r}v(z,x,y)=\sum_{j\in\mathbb{N}}D_{z}^{p}D_{x}^{q}v_{j}(z,x)\Gamma_{j}^{(r)}(y),\qquad(p+q+r\leq2),
\]
and since (\ref{eq:eigenfunc_control_L1}) holds, we also have
\[
\int_{\mathbb{R}}v(z,x,y)dy=\sum_{j=0}^{\infty}v_{j}(z,x)\int_{\mathbb{R}}\Gamma_{j}(y)dy=\sum_{j=0}^{\infty}m_{j}v_{j}(z,x),
\]
where we recall the notation $m_{j}:=\int_{\mathbb{R}}\Gamma_{j}(y)dy$.

\medskip

Let us recall that $n^{0}$ is given by (\ref{def:n0}) and that $\Gamma_{j}^{\prime\prime}+(1-A^{2}y^{2})\Gamma_{j}=-\lambda_{j}\Gamma_{j}$
from Proposition \ref{prop:basis_eigenfunctions}. Consequently, when
projecting the equation $\Lmu(s,v)=f$ along $\Gamma_{j}$, we obtain 
\begin{equation}
(v_{j})_{zz}+2(v_{j})_{xz}+(1+\mu)(v_{j})_{xx}+c_{0}(v_{0})_{z}-\left(\lambda_{j}-\lambda_{0}U(z)\right)v_{j}=f_{j},\qquad j\geq1,\label{eq:front_v_j}
\end{equation}
\begin{equation}
(v_{0})_{zz}+2(v_{0})_{xz}+(1+\mu)(v_{0})_{xx}+c_{0}(v_{0})_{z}-\lambda_{0}\left(1-2U(z)\right)v_{0}=f_{0}-\eta U^{\prime}(z)s+\eta U(z)\sum_{\ell=1}^{\infty}m_{\ell}v_{\ell}.\label{eq:front_v_0}
\end{equation}
Then, multiplying (\ref{eq:front_v_j}) and (\ref{eq:front_v_0})
by $\frac{1}{L}e_{-n}(x)$ and integrating over $x\in [0,L]$, we obtain
\begin{equation}
\mathcal{E}_{n,j,\mu}[v_{j}^{n}]:=(v_{j}^{n})^{\prime\prime}+\left(2in\sigma+c_{0}\right)(v_{j}^{n})^{\prime}-\left(\lambda_{j}-(1+\delta_{0j})\lambda_{0}U(z)+(1+\mu)n^{2}\sigma^{2}\right)v_{j}^{n}=\begin{cases}
f_{j}^{n}(z) & j\geq1,\vspace{3pt}\\
\widetilde{f_{0}^{n}}(z) & j=0,
\end{cases}\label{eq:vnj}
\end{equation}
where we recall $\sigma:=\frac{2\pi}{L}>0$ and denote
\begin{equation}
\widetilde{f_{0}^{n}}(z):= f_{0}^{n}(z)-\eta U^{\prime}(z)s\delta_{n0}+\eta U(z)\sum_{\ell=1}^{\infty}m_{\ell}v_{\ell}^{n}(z).\label{eq:f0n_tilde}
\end{equation}
Finally, we define the operator 
\begin{alignat}{2}
\mathcal{L}_{n,j,\mu}\colon & E_{\kappa}^{2} &  & \to E_{\kappa}^{0}\label{eq:Lnj}\\
 & u &  & \mapsto\mathcal{E}_{n,j,\mu}[u],\nonumber 
\end{alignat}
where for any $k\in\mathbb{N}$ we set
\begin{equation}
E_{\kappa}^{k}:=\left\{ g\in C^{k}(\mathbb{R},\mathbb{C}): ||g||_{\kappa,k}<\infty\right\} ,\qquad ||g||_{\kappa,k}:=\sum_{r=0}^{k} ||g^{(r)}(z)e^{\kappa|z|}||_{L^{\infty}}.\label{eq:E_kappa}
\end{equation}

\medskip

The proof for the rest of subsection \ref{subsec:Front_bijectivity}
is organized as follows. 
\begin{itemize}
\item In subsection \ref{sssec:homo_pb}, we construct a fundamental system
of solutions of the homogenous equations associated to (\ref{eq:vnj}).
\item Then, in subsection \ref{subsec:kappa_Lredef}, we fix the value of
$\kappa$ and investigate the injectivity of the linear operators
$\mathcal{L}_{n,j,\mu}$. To ensure that each $\mathcal{L}_{n,j,\mu}$
is injective, we may redefine some of them on a smaller space $\mathcal{S}_{n,j,\mu}\subset E_{\kappa}^{2}$.
\item Next, in subsection \ref{sssec:solve_j_geq1}, for any $j\geq1$,
we construct explicitly the solution of (\ref{eq:vnj}), which proves
the surjectivity of $\mathcal{L}_{n,j,\mu}$. We also prove that,
for any $n\in\mathbb{Z}$, $j\geq1$ and $\mu>0$ small enough, $v_{j}^{n}$
satisfies (\ref{eq:vnj_toProve}).
\item Afterwards, in subsection \ref{sssec:solve_j_0}, we prove that $\widetilde{f_{0}^{n}}$
satisfies a bound of the type (\ref{eq:fnj_Zbound}). The construction
of $v_{0}^{n}$ then follows in the same way, except for the case
$n=0$, where we shall also prove the existence and uniqueness of
$s$.
\item Finally, in subsection \ref{sssec:Reconstruct_v}, we prove the existence
and uniqueness of $s\in\mathbb{R}$ and $v\in\mathcal{S}_{\mu}\subset\mathcal{Y}_{\mu}$
such that $\Lmu(s,v)=f$, where $\mathcal{S}_{\mu}$ is constructed
from the spaces $\mathcal{S}_{n,j,\mu}$.
\end{itemize}

\subsubsection{Fundamental system of solutions for the homogeneous problem\label{sssec:homo_pb}}

We consider the homogeneous equation associated to (\ref{eq:vnj}),
that is
\begin{equation}
k^{\prime\prime}+\left(2in\sigma+c_{0}\right)k^{\prime}-\left(\lambda_{j}-(1+\delta_{0j})\lambda_{0}U(z)+(1+\mu)n^{2}\sigma^{2}\right)k=0.\label{eq:homo}
\end{equation}
Although we assumed $0<\mu<1$, we also need to consider
solutions of (\ref{eq:homo}) for $\mu=0$. For that reason we shall
assume in this subsection that $0\leq\mu<1$ unless otherwise stated.

To construct a fundamental system of solutions of (\ref{eq:homo}),
we first take the limit $z\rightarrow\pm\infty$ in the coefficients
of (\ref{eq:homo}), and thus consider
\begin{equation}
k^{\prime\prime}+\left(2in\sigma+c_{0}\right)k^{\prime}-\left(\lambda_{j}-(1+\delta_{0j})\lambda_{0}+(1+\mu)n^{2}\sigma^{2}\right)k=0,\label{eq:homo_-}
\end{equation}
and
\begin{equation}
k^{\prime\prime}+\left(2in\sigma+c_{0}\right)k^{\prime}-\left(\lambda_{j}+(1+\mu)n^{2}\sigma^{2}\right)k=0.\label{eq:homo_+}
\end{equation}
A fundamental system of solutions of (\ref{eq:homo_-}) is given by
$z \mapsto e^{a_{n,j,\mu}^{\pm}z}$ with\footnote{In what follows, for any $z=re^{i\theta}\in\mathbb{C}$ with $r\geq0$
and $\theta\in]-\pi,\pi]$, we denote $\sqrt{z}:=\sqrt{r}e^{i\theta/2}$.
In particular, $\text{Re }\sqrt{z}>0$ if $z\in\mathbb{C}\backslash\mathbb{R}_{-}$.}
\begin{equation}
a_{n,j,\mu}^{\pm}=\frac{1}{2}\left(-2in\sigma-c_{0}\pm\sqrt{4\mu n^{2}\sigma^{2}+c_{0}^{2}+4in\sigma c_{0}+4\left[(1-\delta_{0j})\lambda_{j}-\lambda_{0}\right]}\right).\label{eq:a_nj}
\end{equation}
Similarly, a system for (\ref{eq:homo_+}) is given by $z \mapsto e^{b_{n,j,\mu}^{\pm}z}$
with 
\begin{equation}
b_{n,j,\mu}^{\pm}=\frac{1}{2}\left(-2in\sigma-c_{0}\pm\sqrt{4\mu n^{2}\sigma^{2}+c_{0}^{2}+4in\sigma c_{0}+4\lambda_{j}}\right).\label{eq:b_nj}
\end{equation}
Note that, for all $(n,j)\in\mathbb{Z}\times\mathbb{N}$ and $0\leq\mu<1$,
one can straightforwardly check that
\begin{equation}
\text{Re } a_{n,j,\mu}^{-}<0<\text{Re }a_{n,j,\mu}^{+},\qquad\text{Re }b_{n,j,\mu}^{-}<0,\label{eq:sign_ab}
\end{equation}
\begin{equation}
\text{sign}\left(\text{Re }(b_{n,j,\mu}^{+})\right)=\text{sign}\left(\lambda_{j}+(1+\mu)n^{2}\sigma^{2}\right),\label{eq:sign_b+}
\end{equation}
\begin{equation}
\Re\left(a^{+}_{n,j,\mu}-a^{-}_{n,j,\mu}\right)>0,\qquad\Re\left(b^{+}_{n,j,\mu}-b^{-}_{n,j,\mu}\right)\begin{cases}
>0 & \text{if }(n,j,c_{0})\neq(0,0,c^{*}),\\
=0 & \text{otherwise},
\end{cases}\label{eq:Re_a-a_sign}
\end{equation}
with the convention $\text{sign}(0)=0$ and where we recall $c_0\geq c^*:=2\sqrt{-\lambda_0}$. We have the following estimates. 

\begin{lem}[Estimates related to $a_{n,j,\mu}^{\pm},b_{n,j,\mu}^{\pm}$]
\label{lem:ab_estim}There exist $\underline{C},\overline{C}>0$
such that for any $(n,j)\in\mathbb{Z}\times\mathbb{N}$ and $0\leq\mu<1$,
there holds
\begin{equation}
|a_{n,j,\mu}^{\pm}|,\,|b_{n,j,\mu}^{\pm}|\leq\overline{C}\left(1+|n|+\sqrt{j}\right),\label{eq:a_bound}
\end{equation}
\begin{equation}
\left|a_{n,j,\mu}^{+}-a_{n,j,\mu}^{-}\right|,\,\left|b_{n,j,\mu}^{+}-b_{n,j,\mu}^{-}\right|\,\geq\underline{C}\sqrt{\mu n^{2}+j+|n|},\label{eq:a-a_bound}
\end{equation}
\begin{equation}
\left|\Re a_{n,j,\mu}^{\pm}\right|,\,\left|\Re b_{n,j,\mu}^{\pm}\right|\geq\underline{C}\sqrt{\mu n^{2}+j+|n|}-c_{0},\label{eq:Re_a_bound}
\end{equation}
\begin{equation}
\Re\left(a_{n,j,\mu}^{+}-a_{n,j,\mu}^{-}\right),\,\Re\left(b_{n,j,\mu}^{+}-b_{n,j,\mu}^{-}\right)\geq\underline{C}\sqrt{j+|n|},\label{eq:Re_a-a_noMu}
\end{equation}
\begin{equation}
|a_{n,j,\mu}^{\pm}-b_{n,j,\mu}^{\pm}|\leq\overline{C}.\label{eq:a-b_bound}
\end{equation}
\end{lem}

\begin{proof}
The proofs of estimates (\ref{eq:a_bound})---(\ref{eq:Re_a-a_noMu})
are straightforward and omitted. As for (\ref{eq:a-b_bound}), notice
that
\[
-2\left(a_{n,j,\mu}^{-}-b_{n,j,\mu}^{-}\right)=2(a_{n,j,\mu}^{+}-b_{n,j,\mu}^{+})=\sqrt{Z(n,j,\mu)-4\lambda_{0}(1+\delta_{0j})}-\sqrt{Z(n,j,\mu)},
\]
where $Z(n,j,\mu)$ belongs to the half-plane $H_{+}:=\left\{ z\in\mathbb{C}: \Re z\geq0\right\} $
for all $n,j,\mu$. Therefore, setting $\Lambda:=-4\lambda_{0}(1+\delta_{0j})\in \{-4\lambda _0,-8\lambda_0\}>0$,
it is enough to prove that the function $h\colon Z\in H_{+}\mapsto\sqrt{Z+\Lambda}-\sqrt{Z}$
is uniformly bounded, which is rather clear.
\end{proof}

The construction of solutions for (\ref{eq:homo}) follows from the following when $(n,j)\neq (0,0)$.

\begin{lem}[Fundamental system of (\ref{eq:homo}) for $(n,j)\neq(0,0)$]
\label{lem:homo_nj}There exists $\mu_{max}>0$ such that the following
results hold. Fix any $(n,j)\in\mathbb{Z}\times\mathbb{N}$ with\textbf{
$(n,j)\neq(0,0)$} and $0\leq\mu<\mu_{max}$. There exists a fundamental
system of solutions $(\varphi_{-},\varphi_{+})$ of (\ref{eq:homo})
such that
\begin{equation}
\varphi_{-}(z)=\begin{cases}
P_{-}(z)e^{a_{n,j,\mu}^{-}z} & z\leq0,\\
Q_{-}(z)e^{b_{n,j,\mu}^{-}z} & z\geq0,
\end{cases}\qquad\varphi_{+}(z)=\begin{cases}
P_{+}(z)e^{a_{n,j,\mu}^{+}z} & z\leq0,\\
Q_{+}(z)e^{b_{n,j,\mu}^{+}z} & z\geq0,
\end{cases}\label{eq:phi_pm}
\end{equation}
with $P_{\pm}\in C_{b}^{2}(\mathbb{R}_{-})$, $Q_{\pm}\in C_{b}^{2}(\mathbb{R}_{+})$
and $a_{n,j,\mu}^{\pm},b_{n,j,\mu}^{\pm}$ given by (\ref{eq:a_nj})---(\ref{eq:b_nj}).
Also, $\liminf_{z\to-\infty}|P_{-}(z)|>0$ and $\liminf_{z\to+\infty}|Q_{+}(z)|>0$.

Additionally, there exists $R_{max}>0$ such that 
\begin{equation}
\sup_{(n,j)\neq(0,0)}\sup_{0\leq\mu<\mu_{max}}\sup_{R\in\{P_{\pm},Q_{\pm}\}}\left(||R||_{\infty}+||R^{\prime}||_{\infty}\right)\leq R_{max},\label{eq:PQ_bounds}
\end{equation}
where by convention the sup norm is taken over the domain of definition 
of $R$. 

Next, there exists $W_{0}>0$ such that for all $(n,j)\neq(0,0)$
and $0\leq\mu<\mu_{max}$, the Wronskian of $(\varphi_-,\varphi_+)$ at $z=0$ satisfies
\begin{equation}
\left|W_{\varphi}\right|:=\left|\left[\varphi_{-}^{\prime}\varphi_{+}-\varphi_{+}^{\prime}\varphi_{-}\right](0)\right|\geq W_{0},\label{eq:W0}
\end{equation}
Also, there exist $C_{W},N_{0},J_{0}>0$ such that, if $|n|\geq N_{0}$
or $j\geq J_{0}$, we have for any such $n,j$
\begin{equation}
\frac{1}{|W_{\varphi}|}\leq\frac{C_{W}}{\sqrt{1+\mu n^{2}+j+|n|}},\qquad\forall\mu\in[0,\mu_{max}).\label{eq:W0_bignj}
\end{equation}

Finally, there exist $\zeta_{1},\zeta_{2}>0$ such that for all $(n,j)\neq0$
and $0\leq\mu<\mu_{max}$,
\begin{equation}
\int_{-\infty}^{0}|\varphi_{+}(z)|^{2}dz\geq\zeta_{1}e^{-\zeta_{2}\Re a^{+}}.\label{eq:int_phi}
\end{equation}
\end{lem}

The proof of Lemma \ref{lem:homo_nj}, lengthy and technical,
is postponed to Appendix \ref{sec:Lemmaphi}. The case $n=j=0$ is simpler and reads as follows.

\begin{lem}[Fundamental system of (\ref{eq:homo}) for $(n,j)=(0,0)$]
\label{lem:homo_00}For all $0\leq \mu<1$, a fundamental system of solutions
of (\ref{eq:homo}) when $n=j=0$ is given by $(U^{\prime},\Upsilon)$,
where $U$ solves (\ref{eq:FKPP_front}) and 
\begin{equation}
\Upsilon(z):= U^{\prime}(z)\int_{z}^{+\infty}\frac{1}{U^{\prime}(\omega)^{2}}e^{-c_{0}\omega}d\omega.\label{eq:upsilon}
\end{equation}
Additionally, we have 
\begin{equation}
U^{\prime}(z)\approx_{-\infty}e^{a_{0,0,0}^{+}z},\qquad U^{\prime}(z)\approx_{+\infty}\begin{cases}
e^{b_{0,0,0}^{+}z} & \text{if }c_{0}>c^{*},\\
ze^{b_{0,0,0}^{+}z} & \text{if }c_{0}=c^{*},
\end{cases}\label{eq:Up_equiv}
\end{equation}
\begin{equation}
\Upsilon(z)\approx_{-\infty}e^{a_{0,0,0}^{-}z}\qquad\Upsilon(z)\approx_{+\infty}\begin{cases}
e^{b_{0,0,0}^{-}z} & \text{if }c_{0}>c^{*},\\
\frac{1}{z}e^{b_{0,0,0}^{-}z} & \text{if }c_{0}=c^{*},
\end{cases}\label{eq:upsilon_equiv}
\end{equation}
where $A(z)\approx_{\pm\infty}B(z)$ with $B(z)>0$ means $0<\liminf_{\pm\infty}\frac{\left|A(z)\right|}{B(z)}<\limsup_{\pm\infty}\frac{\left|A(z)\right|}{B(z)}<+\infty$.
\end{lem}

\begin{proof}
Estimates (\ref{eq:Up_equiv}) are classical results for the critical $(c_0=c^*)$ and supercritical ($c_0>c^*$) Fisher-KPP traveling waves. When $n=j=0$,
(\ref{eq:homo}) amounts to
\begin{equation}
k^{\prime\prime}+c_{0}k^{\prime}-\lambda_{0}\left(1-2U(z)\right)k=0.\label{eq:L00_homo}
\end{equation}
Note that $\mu$ does not play any role here. First, we see that $U^{\prime}$
solves (\ref{eq:L00_homo}) since $U$ solves (\ref{eq:FKPP_front}). In this case another solution (non-proportional to $U^\prime$) of (\ref{eq:L00_homo}) can be sought in the form of $\Upsilon(z) = g(z)U^{\prime}(z)$. Using this, some straightforward computations yield that (\ref{eq:upsilon}) is another solution. Then (\ref{eq:upsilon_equiv})
follows straightforwardly from (\ref{eq:upsilon}) and (\ref{eq:Up_equiv}). 
\end{proof}

\subsubsection{Fixing the values $\kappa$ and $\mu_{max}$, redefinitions of $\mathcal{L}_{n,j,\mu}$\label{subsec:kappa_Lredef}}

Here, we shall fix the value of $\kappa$ with the following Lemma.

\begin{lem}[Choice of $\kappa$]
\label{lem:kappa_choice}If $\mu_{max}>0$ is small enough, there
exists $\kappa>0$ such that for all $(n,j)\in\mathbb{Z}\times\mathbb{N}$
and $0\leq\mu<\mu_{max}$, we have
\begin{equation}
\begin{cases}
\left|\Re a_{n,j,\mu}^{\pm}\right|,\left|\Re b_{n,j,\mu}^{-}\right|\geq2\kappa,\vspace{0.3cm}\\
\left|\Re b_{n,j,\mu}^{+}\right|\geq2\kappa, & \text{if }\Re b_{n,j,\mu}^{+}<0.
\end{cases}\label{eq:kappa}
\end{equation}
Also, there exist $C_{\kappa},\overline{N},\overline{J}\geq0$ such
that if $|n|\geq\overline{N}$ or $j\geq\overline{J}$, then for all
$0\leq\mu<\mu_{max}$, we have
\begin{equation}
0<\frac{1}{|\Re a_{n,j,\mu}^{\pm}|-\kappa},\frac{1}{|\Re b_{n,j,\mu}^{\pm}|-\kappa}\leq\frac{C_{\kappa}}{\sqrt{1+\mu n^{2}+j+|n|}}.\label{eq:C_kappa}
\end{equation}
\end{lem}

\begin{proof}
Let us first prove (\ref{eq:kappa}). We define the following sets:
\begin{align*}
\mathcal{I} & :=\left\{(n,j)\in\mathbb{Z}\times\mathbb{N}: \lambda_{j}+n^{2}\sigma^{2}<0\right\} =\left\{ (n,j)\in\mathbb{Z}\times\mathbb{N}: \Re b_{n,j,0}^{+}<0\right\},\\
\mathcal{R}_{\mu} & :=\left(\cup_{(n,j)\in\mathbb{Z}\times\mathbb{N}}\left\{ \left|\Re a_{n,j,\mu}^{+}\right|,\left|\Re a_{n,j,\mu}^{-}\right|,\left|\Re b_{n,j,\mu}^{-}\right|\right\} \right)\bigcup\left(\cup_{(n,j)\in\mathcal{I}}\left\{ \left|\Re b_{n,j,\mu}^{+}\right|\right\} \right),
\end{align*}
for all $0\leq\mu<\mu_{max}$. Because of (\ref{eq:sign_ab}) and
the definition of $\mathcal{I}$, there holds $0\notin\mathcal{R}_{0}$.
Note that, due to (\ref{eq:Re_a_bound}), the sets $\mathcal{I}$
and $\mathcal{R}_{0}\cap[0,a_{0,0,0}^{+}]$ are finite. Therefore
we have
\[
m:=\min\left(\mathcal{R}_{0}\cap[0,a_{0,0,0}^{+}]\right)=\min\mathcal{R}_{0}>0.
\]
Now, because of (\ref{eq:Re_a_bound}), there exist $N,J\geq0$ such
that for all $0\leq\mu<\mu_{max}$
\[
\mathcal{R}_{\mu}\cap[0,a_{0,0,\mu}^{+}]\subset\cup_{|n|\leq N,j\leq J}\left\{ \left|\Re a_{n,j,\mu}^{+}\right|,\left|\Re a_{n,j,\mu}^{-}\right|,\left|\Re b_{n,j,\mu}^{-}\right|,\left|\Re b_{n,j,\mu}^{+}\right|\right\} .
\]
From (\ref{eq:a_nj})---(\ref{eq:b_nj}) we easily obtain that 
\[
\sup_{|n|\leq N,\,j\leq J}\left|\Re a_{n,j,\mu}^{\pm}-\Re a_{n,j,0}^{\pm}\right|\xrightarrow[\mu\to0]{}0,\qquad\sup_{|n|\leq N,\,j\leq J}\left|\Re b_{n,j,\mu}^{\pm}-\Re b_{n,j,0}^{\pm}\right|\xrightarrow[\mu\to0]{}0.
\]
As a result, taking $\mu_{max}$ small enough, we have
\begin{equation*}
\inf_{0\leq\mu<\mu_{max}}\min\left(\mathcal{R}_{\mu}\cap[0,a_{0,0,\mu}^{+}]\right)=\inf_{0\leq\mu<\mu_{max}}\min\mathcal{R}_{\mu}\geq\frac{m}{2}>0.
\end{equation*}
Consequently, (\ref{eq:kappa}) holds with $\kappa=\frac{m}{4}>0$.

Finally, it remains to prove (\ref{eq:C_kappa}). Let $\underline{C}>0$
being given by Lemma \ref{lem:ab_estim}. There exist $\overline{N},\overline{J}\geq0$
such that, if $|n|\geq\overline{N}$ or $j\ge\overline{J}$, we have
for all $0\leq\mu<\mu_{max}$
\[
\sqrt{\mu n^{2}+j+|n|}\geq\frac{2}{\underline{C}}(c_{0}+\kappa)+1.
\]
We then deduce that
\[
\frac{1}{\underline{C}\sqrt{\mu n^{2}+j+|n|}-c_{0}-\kappa}\leq\frac{C_{\kappa}}{1+\sqrt{\mu n^{2}+j+|n|}}\leq\frac{C_{\kappa}}{\sqrt{1+\mu n^{2}+j+|n|}},
\]
for $C_{\kappa}=2/\underline{C}>0$. This yields (\ref{eq:C_kappa})
thanks to (\ref{eq:Re_a_bound}).
\end{proof}

Let us recall that $\mathcal{L}_{n,j,\mu}$ and $E_{\kappa}^{k}$
are defined by (\ref{eq:Lnj}) and (\ref{eq:E_kappa}) respectively.
We equip $E_{\kappa}^{2}$ with the Hermitian inner product $\langle g_{1},g_{2}\rangle:=\int_{\mathbb{R}}g_{1}(z)\overline{g_{2}(z)}dz$.

\begin{lem}[Injectivity of $\mathcal{L}_{n,j,\mu}$ after redefinitions]
\label{lem:redef}Let $\mu_{max}>0$ small enough so that both Lemmas
\ref{lem:homo_nj} and \ref{lem:kappa_choice} hold. Let $(n,j)\neq(0,0)$,
$\mu\in(0,\mu_{max})$ and $\kappa>0$ given by Lemma \ref{lem:kappa_choice}. 

If $\Re b_{n,j,\mu}^{+}\geq0$, then $\mathcal{L}_{n,j,\mu}$ is injective.

If $\Re b_{n,j,\mu}^{+}<0$, then we set $\mathcal{S}_{n,j,\mu}:=\{\varphi_{+}\}^{\bot}\subset E_{\kappa}^{2}$,
and we redefine $\mathcal{L}_{n,j,\mu}\colon\mathcal{S}_{n,j,\mu}\to E_{\kappa}^{0}$
as an injective operator.

Finally, we set $\mathcal{S}_{0,0}:=\left\{ U^{\prime}\right\} ^{\bot}\subset E_{\kappa}^{2}$
, and we redefine $\mathcal{L}_{0,0,\mu}\colon\mathcal{S}_{0,0}\to E_{\kappa}^{0}$
as an injective operator.
\end{lem}

\begin{proof}
Let $n,j,\mu$ satisfy the above conditions. Let us recall that from
Lemma \ref{lem:homo_nj}, the solutions of $\mathcal{E}_{n,j,\mu}[u]=0$
are exactly $C_{-}\varphi_{-}+C_{+}\varphi_{+}$ with $C_{\pm}\in\mathbb{C}$.
Note that $\varphi_{-}\notin E_{\kappa}^{2}$ since $|\varphi_{-}(-\infty)|=+\infty$
from (\ref{eq:sign_ab}) and (\ref{eq:phi_pm}).

 If $\Re b_{n,j,\mu}^{+}\geq0$,
then from (\ref{eq:phi_pm}) we also have $\varphi_{+}\notin C_{0}(\mathbb{R},\mathbb{C})$,
so that $\varphi_{+}\notin E_{\kappa}^{2}$, which implies $\ker\mathcal{L}_{n,j,\mu}$
is trivial. 

If $\Re b_{n,j,\mu}^{+}<0$, then from (\ref{eq:phi_pm}) and (\ref{eq:kappa}),
we have $\varphi_{+}\in E_{\kappa}^{2}$. Therefore $\ker\mathcal{L}_{n,j,\mu}=\text{span}(\varphi_{+})$.
Setting $\mathcal{S}_{n,j,\mu}:=\{\varphi_{+}\}^{\bot}$, we
have that $\mathcal{L}_{n,j,\mu}\colon\mathcal{S}_{n,j,\mu}\to E_{\kappa}^{0}$
is injective. 

The last assertion for $\mathcal{L}_{0,0,\mu}$ is proved similarly, using (\ref{eq:Up_equiv})---(\ref{eq:upsilon_equiv})
and (\ref{eq:kappa}).
\end{proof}

\subsubsection{Solving (\ref{eq:vnj}) when $j\protect\geq1$\label{sssec:solve_j_geq1}}

For the rest of this section, we fix $\mu_{max},\kappa>0$ small enough
such that Lemmas \ref{lem:homo_nj} and \ref{lem:kappa_choice} are
valid. From Lemma \ref{lem:homo_nj}, we are equipped with $(\varphi_{-},\varphi_{+})$
given by (\ref{eq:phi_pm}), which is a fundamental system of solutions
of (\ref{eq:homo}). Let us mention that by construction $\kappa<a_{0,0,0}^{+}=-\frac{1}{2}c_{0}+\frac{1}{2}\sqrt{c_{0}^{2}-4\lambda_{0}}$,
which is consistent with our assumption at the beginning of subsection
\ref{subsec:Spaces_Yrond_Zrond}.

\medskip

In this subsection, we prove that, for each $n\in\mathbb{Z}$, $j\geq1$,
and $0<\mu<\mu_{max}$ there exists a unique $v_{j}^{n}$ such that
$\mathcal{L}_{n,j,\mu}v_{j}^{n}=f_{j}^{n}$. Additionally, we shall
prove the existence of $K>0$ independent of $n,j,\mu,f$ such that
\begin{equation}
\left|(v_{j}^{n})^{(k)}(z)\right|\leq K||f||_{\mathcal{Z}} \frac{e^{-\kappa|z|}}{(1+j)^{\beta}(1+|n|)^{\gamma}}\times\frac{1+|n|^{k}+j^{k/2}}{1+\mu n^{2}+j+|n|},\quad k\in\{0,1,2\},\,\forall z\in\mathbb{R}.\label{eq:vnj_toProve}
\end{equation}
Let us recall that $f$ satisfies (\ref{eq:fnj_Zbound}), and thus
$f_{j}^{n}\in E_{\kappa}^{0}$. In what follows, we denote $a^{\pm}=a_{n,j,\mu}^{\pm}$
and $b^{\pm}=b_{n,j,\mu}^{\pm}$ when there is no confusion. We shall
split the proof in two subcases, depending on the sign of $\Re b^{+}$.

\paragraph{Indexes $(n,j,\mu)$ such that $j\protect\geq1$ and $\protect\Re b_{n,j,\mu}^{+}\protect\geq0$.} For such $n,j,\mu$, we have the injectivity of $\mathcal{L}_{n,j,\mu}\colon E_{\kappa}^{2}\to E_{\kappa}^{0}$
from Lemma \ref{lem:redef}, so there is at most one solution $v_{j}^{n}\in E_{\kappa}^{2}$
of (\ref{eq:vnj}). To prove its existence, we construct explicitly
a solution with the variation of the constant, that is
\[
v_{j}^{n}(z)=\varphi_{-}(z)\int_{-\infty}^{z}\frac{1}{W(\omega)}\varphi_{+}(\omega)f_{j}^{n}(\omega)d\omega+\varphi_{+}(z)\int_{z}^{+\infty}\frac{1}{W(\omega)}\varphi_{-}(\omega)f_{j}^{n}(\omega)d\omega,
\]
where we denote the Wronskian $W(\omega):=\left[\varphi_{-}^{\prime}\varphi_{+}-\varphi_{+}^{\prime}\varphi_{-}\right](\omega)\neq0$.
Also, notice that since $(\varphi_{-},\varphi_{+})$ solve (\ref{eq:vnj}),
there holds
\begin{align*}
W(\omega) & =W(0)e^{-(c_{0}+2in\sigma)\omega}=W(0)e^{(a^{+}+a^{-})\omega}=W(0)e^{(b^{+}+b^{-})\omega}.
\end{align*}
To prove that $v_{j}^{n}$ satisfies $\mathcal{L}_{n,j,\mu}v_{j}^{n}=f_{j}^{n}$,
it suffices to prove that $v_{j}^{n}\in E_{\kappa}^{2}$. It is in
particular enough to prove that $v_{j}^{n}$ satisfies (\ref{eq:vnj_toProve}).

Let us first prove that (\ref{eq:vnj_toProve}) holds for $k=0$.
For all $z\geq0$, there holds
\begin{align*}
v_{j}^{n}(z) & =Q_{-}(z)e^{b^{-}z}\left(\int_{-\infty}^{0}\frac{P_{+}(\omega)}{W_{\varphi}}e^{-a^{-}\omega}f_{j}^{n}(\omega)+\int_{0}^{z}\frac{Q_{+}(\omega)}{W_{\varphi}}e^{-b^{-}\omega}f_{j}^{n}(\omega)d\omega\right)\\
 & \quad+Q_{+}(z)e^{b^{+}z}\left(\int_{z}^{+\infty}\frac{Q_{-}(\omega)}{W_{\varphi}}e^{-b^{+}\omega}f_{j}^{n}(\omega)d\omega\right),
\end{align*}
where we recall that $W(0)=W_{\varphi}$ satisfies (\ref{eq:W0}).
Combining (\ref{eq:fnj_Zbound}), (\ref{eq:PQ_bounds}) and (\ref{eq:kappa}),
we obtain
\begin{align*}
|v_{j}^{n}(z)| & \leq\frac{R_{max}^{2}||f||_{\mathcal{Z}}}{|W_{\varphi}|\,(1+j)^{\beta}(1+|n|)^{\gamma}}\times\\
 & \quad\left[e^{\Re b^{-}z}\left(\int_{-\infty}^{0}e^{-\Re a^{-}\omega}e^{\kappa\omega}d\omega+\int_{0}^{z}e^{-\Re b^{-}\omega}e^{-\kappa\omega}d\omega\right)+e^{\Re b^{+}z}\int_{z}^{+\infty}e^{-\Re b^{+}\omega}e^{-\kappa\omega}d\omega\right]\\
 & \leq\frac{R_{max}^{2}||f||_{\mathcal{Z}}}{|W_{\varphi}|\,(1+j)^{\beta}(1+|n|)^{\gamma}}\times\left(\frac{e^{\Re b^{-}z}}{-\Re a^{-}+\kappa}+\frac{e^{-\kappa z}-e^{\Re b^{-}z}}{-\Re b^{-}-\kappa}+\frac{e^{-\kappa z}}{\Re b^{+}+\kappa}\right)\\
 & \leq\frac{R_{max}^{2}||f||_{\mathcal{Z}}}{|W_{\varphi}|\,(1+j)^{\beta}(1+|n|)^{\gamma}}\left(\frac{1}{-\Re a^{-}+\kappa}+\frac{1}{-\Re b^{-}-\kappa}+\frac{1}{\Re b^{+}+\kappa}\right)e^{-\kappa z}.
\end{align*}
Let $\overline{N}_{0}=\max(N_{0},\overline{N})$ and $\overline{J}_{0}=\max(J_{0},\overline{J})$,
where $N_{0},J_{0}$ are given by Lemma \ref{lem:homo_nj} and $\overline{N},\overline{J}$
are given by Lemma \ref{lem:kappa_choice}. If $|n|\geq\overline{N}_{0}$
or $j\geq\overline{J}_{0}$, then (\ref{eq:W0_bignj}) and (\ref{eq:C_kappa})
hold. Therefore
\begin{equation}
|v_{j}^{n}(z)|\leq\frac{3C_{W}C_{\kappa}R_{max}^{2}||f||_{\mathcal{Z}}e^{-\kappa z}}{(1+j)^{\beta}(1+|n|)^{\gamma}}\times\frac{1}{1+\mu n^{2}+j+|n|},\qquad\text{if }|n|\geq\overline{N}_{0}\text{ or }j\geq\overline{J}_{0}.\label{eq:vnj_bound_bignj}
\end{equation}
Meanwhile, if $|n|\leq\overline{N}_{0}$ and $j\leq\overline{J}_{0}$,
we have from (\ref{eq:W0}) and (\ref{eq:kappa}) that 
\begin{equation}
|v_{j}^{n}(z)|\leq\frac{3R_{max}^{2}||f||_{\mathcal{Z}}e^{-\kappa z}}{\kappa W_{0}(1+j)^{\beta}(1+|n|)^{\gamma}},\qquad\text{if }|n|\leq\overline{N}_{0}\text{ and }j\leq\overline{J}_{0}.\label{eq:vnj_bound_smallnj}
\end{equation}
Note that $\overline{N}_{0},\overline{J}_{0}$ do not depend on $\mu\in(0,\mu_{max})$.
Therefore, combining (\ref{eq:vnj_bound_bignj})---(\ref{eq:vnj_bound_smallnj}),
there exists $K>0$ independent of $n,j,\mu,f$ such that (\ref{eq:vnj_toProve})
holds for $k=0$, $z\geq0$. The proof is similar for $z\leq0$.

Let us now prove that (\ref{eq:vnj_toProve}) is valid for $k=1$.
Note that 
\[
(v_{j}^{n})^{\prime}(z)=\varphi_{-}^{\prime}(z)\int_{-\infty}^{z}\frac{1}{W(\omega)}\varphi_{+}(\omega)f_{j}^{n}(\omega)d\omega+\varphi_{+}^{\prime}(z)\int_{z}^{+\infty}\frac{1}{W(\omega)}\varphi_{-}(\omega)f_{j}^{n}(\omega)d\omega.
\]
Then similar calculations and arguments yield that if $|n|\geq\overline{N}_{0}$
or $j\geq\overline{J}_{0}$, then for any $z\in\mathbb{R}$
\begin{align*}
\left|(v_{j}^{n})^{\prime}(z)\right| & \leq\frac{3C_{W}C_{\kappa}R_{max}^{2}||f||_{\mathcal{Z}}e^{-\kappa z}}{(1+j)^{\beta}(1+|n|)^{\gamma}}\times\frac{1+\max\left(|a^{+}|,|a^{-}|,|b^{+}|,|b^{-}|\right)}{\left(1+\sqrt{\mu n^{2}+j+|n|}\right)^{2}},
\end{align*}
thus, using (\ref{eq:a_bound}), we obtain
\begin{align*}
|(v_{j}^{n})^\prime (z)| & \leq\frac{3C_{W}C_{\kappa}(1+\overline{C})R_{max}^{2}||f||_{\mathcal{Z}}e^{-\kappa z}}{(1+j)^{\beta}(1+|n|)^{\gamma}}\times\frac{1+|n|+\sqrt{j}}{1+\mu n^{2}+j+|n|}.
\end{align*}
Meanwhile, if $|n|\leq\overline{N}_{0}$ and $j\leq\overline{J}_{0}$,
then in the same fashion, for any $z\in\mathbb{R}$, there holds
\begin{align*}
|(v_{j}^{n})^\prime (z)| & \leq\frac{3R_{max}^{2}||f||_{\mathcal{Z}}e^{-\kappa z}}{\kappa W_{0}(1+j)^{\beta}(1+|n|)^{\gamma}}\times\left[1+\max\left(|a^{+}|,|a^{-}|,|b^{+}|,|b^{-}|\right)\right]\\
 & \leq\frac{3R_{max}^{2}||f||_{\mathcal{Z}}e^{-\kappa z}}{\kappa W_{0}(1+j)^{\beta}(1+|n|)^{\gamma}}\left(1+\max_{|n|\leq\overline{N}_{0},j\leq\overline{J}_{0},0\leq\mu\leq\mu_{max}}\max\left(|a_{n,j,\mu}^{+}|,|a_{n,j,\mu}^{-}|,|b_{n,j,\mu}^{+}|,|b_{n,j,\mu}^{-}|\right)\right).
\end{align*}
Likewise, taking $K>0$ possibly even larger, $v_{j}^{n}$ satisfies
(\ref{eq:vnj_toProve}) for $k=1$. Finally, since (\ref{eq:vnj_toProve})
is proved for $k\in\{0,1\}$, the proof for $k=2$ is a direct consequence,
with a possibly larger $K>0$, since $v_{j}^{n}$ solves (\ref{eq:vnj})
and (\ref{eq:fnj_Zbound}) holds. Therefore, assuming $j\geq1$ and
$\Re b_{n,j,\mu}^{+}\geq0$, (\ref{eq:vnj_toProve}) holds with $K>0$
that does not depend on $n,j,\mu,f$.

\paragraph{Indexes $(n,j,\mu)$ such that $j\protect\geq1$ and $\protect\Re b_{n,j,\mu}^{+}<0$.} For such $n,j,\mu$, from Lemma \ref{lem:redef}, the operator $\mathcal{L}_{n,j,\mu}\colon\mathcal{S}_{n,j,\mu}\to E_{\kappa}^{0}$
is injective with $\mathcal{S}_{n,j,\mu}=\{\varphi_{+}\}^{\bot}$.
We define the family
\[
\chi_{\xi}(z):=\xi\varphi_{+}(z)+\varphi_{-}(z)\int_{-\infty}^{z}\frac{1}{W(\omega)}\varphi_{+}(\omega)f_{j}^{n}(\omega)d\omega+\varphi_{+}(z)\int_{0}^{z}\frac{1}{W(\omega)}\varphi_{-}(\omega)f_{j}^{n}(\omega)d\omega,\qquad\xi\in\mathbb{C}.
\]
Owing to the variation of the constant, we see that $\chi_{\xi}$ solves
(\ref{eq:vnj}). Using (\ref{eq:fnj_Zbound}), (\ref{eq:PQ_bounds})---(\ref{eq:W0})
and (\ref{eq:kappa}) as we did above, one can readily check that
$\chi_{\xi}\in E_{\kappa}^{2}$ for all $\xi\in\mathbb{C}$. Thus
there is a unique $\xi_{0}=-\frac{\langle\chi_{0},\varphi_{+}\rangle}{\langle\varphi_{+},\varphi_{+}\rangle}\in\mathbb{C}$
such that $\chi_{\xi_{0}}\in\mathcal{S}_{n,j,\mu}$. Therefore, the
equation $\mathcal{L}_{n,j,\mu}v_{j}^{n}=f_{j}^{n}$ admits a unique
solution, given by $v_{j}^{n}=\chi_{\xi_{0}}$.

It remains to prove that $v_{j}^{n}$ satisfies (\ref{eq:vnj_toProve}).
First, notice that $\Re b_{n,j,\mu}^{+}<0$ implies, from (\ref{eq:sign_b+}),
that $(n,j)$ belongs to a finite set $S\subset\mathbb{Z}\times\mathbb{N}_{+}$, independently of $\mu\in (0,\mu_{max})$.
Fix now $(n,j)\in S$. It can be readily checked that there exists
$C_{n,j}>0$ independent of $f$ such that 
\[
||\varphi_{+}||_{\kappa,2}\leq C_{n,j},\quad||\chi_{0}||_{\kappa,2}\leq C_{n,j}||f||_{\mathcal{Z}},\quad\forall\mu\in(0,\mu_{max}).
\]
We claim that there exists $C_{n,j}^{\prime}>0$ independent of $f$
such that $|\xi_{0}|\leq C_{n,j}^{\prime}||f||_{\mathcal{Z}}$ for
all $\mu\in(0,\mu_{max})$. On the one hand, by the Cauchy-Schwarz
inequality
\[
\left|\langle\chi_{0},\varphi_{+}\rangle\right|\leq\sqrt{\int_{\mathbb{R}}|\chi_{0}(z)|^{2}dz\int_{\mathbb{R}}|\varphi_{+}(z)|^{2}dz}\leq\frac{1}{\kappa}||\chi_{0}||_{\kappa,0}||\varphi_{+}||_{\kappa,0}\leq\frac{1}{\kappa}C_{n,j}^{2}||f||_{\mathcal{Z}}.
\]
On the other hand, we have from (\ref{eq:int_phi})
\begin{align*}
\left|\langle\varphi_{+},\varphi_{+}\rangle\right| & \geq\zeta_{1}e^{-\zeta_{2}\Re a^{+}}\geq\zeta_{1}e^{-\zeta_{2}M},\qquad M:=\max_{(n,j)\in S}\max_{0\leq\mu\leq1}\Re a_{n,j,\mu}^{+}>0.
\end{align*}
Therefore we deduce that such $C_{n,j}^{\prime}$ exists. Thus, we
have $||v_{j}^{n}||_{\kappa,2}\leq\left(1+C_{n,j}^{\prime}\right)C_{n,j}||f||_{\mathcal{Z}}$
for all $(n,j)\in S$ and $0<\mu<\mu_{max}$. Since the set $S$ is
finite, taking $K>0$ possibly even larger, independently of $n,j,\mu,f$,
we deduce that $v_{j}^{n}$ satisfies (\ref{eq:vnj_toProve}).

\subsubsection{Solving (\ref{eq:vnj}) when $j=0$\label{sssec:solve_j_0}}

From subsection \ref{sssec:solve_j_geq1}, we are now equipped with
$v_{j}^{n}=\mathcal{L}_{n,j,\mu}^{-1}(f_{j}^{n})$ for every $n\in\mathbb{Z}$,
$j\geq1$ and $0<\mu<\mu_{max}$. Also, there exists $K>0$ independent
of $n,j,\mu,f$ such that those $v_{j}^{n}$ satisfy (\ref{eq:vnj_toProve}).
Therefore, since (\ref{eq:eigenfunc_control_L1}) holds and $\beta>\frac{19}{4}>\frac{5}{4}$,
we have
\[
\left|\sum_{\ell=1}^{\infty}m_{\ell}v_{\ell}^{n}(z)\right|\leq K||f||_{\mathcal{Z}} \frac{e^{-\kappa|z|}}{(1+|n|)^{\gamma}}\times\Sigma,\qquad\Sigma:=\sum_{\ell=1}^{\infty}\frac{|m_{\ell}|}{(1+j)^{\beta}}<\infty.
\]
Let us recall that $\widetilde{f_{0}^{n}}$ is defined by (\ref{eq:f0n_tilde}). Since (\ref{eq:fnj_Zbound}) holds, we deduce that
\begin{equation}
\exists C>0,\,\forall n\neq0,\,\forall\mu\in(0,\mu_{max})\quad\left|\widetilde{f_{0}^{n}}(z)\right|\leq C||f||_{\mathcal{Z}} \frac{e^{-\kappa|z|}}{(1+|n|)^{\gamma}},\label{eq:___f0n_bound}
\end{equation}
As a consequence, for any $n\neq0$, we prove, in the same manner
as in subsection \ref{sssec:solve_j_geq1}, that $v_{0}^{n}$ satisfies
(\ref{eq:vnj_toProve}) if $K>0$ is large enough, independently of
$n,\mu,f$.

\medskip

The case $n=j=0$ is particular since this is the only equation where
$s$, i.e. our perturbed speed, appears. Given that $\mathcal{L}_{0,0,\mu}$
does not depend on $\mu$, we denote it $\mathcal{L}_{0,0}$ from
now on. Let us recall that from Lemma \ref{lem:redef}, $\mathcal{L}_{0,0}\colon\mathcal{S}_{0,0}\to E_{\kappa}^{0}$
is injective. Repeating the same arguments as above, we prove that
$\mathcal{L}_{0,0}$ is surjective, thus bijective. Now, we set 
\begin{equation}
h(z):=\mathcal{L}_{0,0}^{-1}(U^{\prime}),\qquad\widehat{\mathcal{S}}_{0,0}:=\left\{ U^{\prime}\right\} ^{\bot}\cap\left\{ h\right\} ^{\bot}\subset E_{\kappa}^{2},\label{eq:h_S00}
\end{equation}
and we define the following operator as a restriction of $\mathcal{L}_{0,0}$: 
\[
\widehat{\mathcal{L}}_{0,0}\colon\mathcal{\widehat{\mathcal{S}}}_{0,0}\subset E_{\kappa}^{2}\rightarrow E_{\kappa}^{0}.
\]
It is clear that $\widehat{\mathcal{L}}_{0,0}$ is not bijective since
$\widehat{\mathcal{S}}_{0,0}\subsetneq\mathcal{S}_{0,0}$. However,
we shall prove that the linear operator
\begin{alignat*}{2}
\mathcal{M}\colon & \mathbb{C}\times\widehat{\mathcal{S}}_{0,0} &  & \to E_{\kappa}^{0}\\
 & (s,v) &  & \mapsto\eta U^{\prime}(z) s+\widehat{\mathcal{L}}_{0,0}v
\end{alignat*}
is bijective. Assume that $\mathcal{M}(s,v)=0$. Then 
\[
\widehat{\mathcal{L}}_{0,0}v=\mathcal{L}_{0,0}v=-\eta U^{\prime}(z) s,
\]
which implies $v=-\eta sh$. Since $\widehat{\mathcal{S}}_{0,0}\subset\{h\}^{\bot}$,
we deduce that $s=0$, thus $v=0$. Therefore $\mathcal{M}$ is injective.
Let us now prove that $\mathcal{M}$ is surjective. For any $f\in E_{\kappa}^{0}$,
we set
\[
s=\frac{\langle\mathcal{L}_{0,0}^{-1}f,h\rangle}{\eta\langle h,h\rangle},\qquad v=\mathcal{L}_{0,0}^{-1}(f-\eta sU^{\prime})=\mathcal{L}_{0,0}^{-1}f-\eta sh.
\]
By definition of $\mathcal{L}_{0,0}$, we have $v\in\{U^{\prime}\}^{\bot}$,
thus $v\in\widehat{\mathcal{S}}_{0,0}$ by our choice of $s$. Finally,
$\widehat{\mathcal{L}}_{0,0}v=\mathcal{L}_{0,0}v=f-\eta sU^{\prime}$,
so that we indeed have $\mathcal{M}(s,v)=f$. Hence $\mathcal{M}$
is bijective.

\medskip

To conclude, we return to (\ref{eq:vnj}) for $n=j=0$. Note that,
with $f\in\mathcal{Z}$ being given, the functions $(v_{\ell}^{0})_{\ell\ge1}$
are uniquely determined from subsection \ref{sssec:solve_j_geq1}.
From now on, we rewrite $\mathcal{L}_{0,\ell}:=\mathcal{L}_{0,\ell,\mu}$
since those operators do not, in fact, depend on $\mu$. Therefore
we may recast (\ref{eq:vnj}) as
\begin{equation}
\eta U^{\prime}(z)s+\mathcal{E}_{0,0,0}[v_{0}^{0}]=f_{0}^{0}(z)+\eta U(z) \sum_{\ell=1}^{\infty}m_{\ell}\mathcal{L}_{0,\ell}^{-1}(f_{\ell}^{0})\eqqcolon\Phi_{f}(z).\label{eq:v00_recast}
\end{equation}
From the bijectivity of $\mathcal{M}$, there thus exists a unique
couple $(s,v_{0}^{0})\in\mathbb{C}\times\widehat{\mathcal{S}}_{0,0}$
solving (\ref{eq:v00_recast}). We claim that
\begin{equation}
s\in\mathbb{R},\qquad|s|\leq K_{0}||f||_{\mathcal{Z}},\qquad|v_{0}^{0}(z)|\leq K_{0}||f||_{\mathcal{Z}}e^{-\kappa|z|},\quad\forall z\in\mathbb{R},\label{eq:s_v00}
\end{equation}
for some $K_{0}>0$ independent of $f,\mu$. On the one hand, from
(\ref{eq:fj_nota})---(\ref{eq:en_nota}), we see that $f_{\ell}^{0}$
is real-valued for all $\ell\in\mathbb{N}$. On the other
hand, note that for any $\ell\in\mathbb{N}$, $\mathcal{L}_{0,\ell}$
has real coefficients. By uniqueness of the solution of $\mathcal{L}_{0,\ell}v_{\ell}^{0}=f_{\ell}^{0}$
for all $\ell\geq1$, the functions $v_{\ell}^{0}$ are also real-valued.
Therefore $\Phi_{f}(z)\in\mathbb{R}$ and does not depend on $\mu$.
Also, repeating the same arguments that we used to obtain (\ref{eq:___f0n_bound}),
we have $||\Phi_{f}||_{\kappa,0}\leq C_{\Phi}||f||_{\mathcal{Z}}$
for some $C_{\Phi}>0$ independent of $f,\mu$. Now, inverting $\mathcal{M}$,
we obtain
\begin{equation}
s=\frac{\int_{\mathbb{R}}\left[\mathcal{L}_{0,0}^{-1}\Phi_{f}\right](z)h(z)dz}{\eta\int_{\mathbb{R}}|h(z)|^{2}dz}.\label{eq:s}
\end{equation}
Similarly as above, $h=\mathcal{L}_{0,0}^{-1}(U^{\prime})$ is real-valued
and does not depend on $\mu$. Therefore $s\in\mathbb{R}$ does not
depend on $\mu$, and there holds
\begin{equation}
|s|\leq\frac{C_{\Phi}||\mathcal{L}_{0,0}^{-1}||\,||h||_{\kappa,0}}{\kappa\eta\int_{\mathbb{R}}|h(z)|^{2}dz}||f||_{\mathcal{Z}}.\label{eq:s_bound}
\end{equation}
Therefore, $s$ satisfies (\ref{eq:s_v00}) for $K_{0}$ large enough.
One can readily check that the same is true for $v_{0}^{0}=\mathcal{L}_{0,0}^{-1}(\Phi_{f})-\eta sh$.

\medskip

Combining the results of subsections \ref{sssec:solve_j_geq1} and
\ref{sssec:solve_j_0}, we have thus proved the following.

\begin{prop}[Results of subsections \ref{sssec:solve_j_geq1} and \ref{sssec:solve_j_0}]
\label{prop:summary}There exist $\mu_{max},\kappa>0$ so that for
any fixed $\mu\in(0,\mu_{max})$, the following results hold: there
exists a finite set $I_{\mu}\subset\mathbb{Z}\times\mathbb{N}$, and
a family of subsets $(\mathcal{S}_{n,j,\mu})_{(n,j)\in I_{\mu}}$
of $E_{\kappa}^{2}$, such that there exist
a unique $s\in\mathbb{R}$ and, for any $(n,j)\in\mathbb{Z}\times\mathbb{N}$,
a unique 
\[
v_{j}^{n}\begin{cases}
\in\mathcal{S}_{n,j,\mu} & \text{if }(n,j)\in I_{\mu},\\
\in E_{\kappa}^{2} & \text{otherwise},
\end{cases}
\]
such that (\ref{eq:vnj}) holds. Additionally, there exists $K>0$
independent of $n,j,\mu,f$ such that (\ref{eq:vnj_toProve}) holds.

Finally, with $h,\Phi_{f}$ being defined by (\ref{eq:h_S00})---(\ref{eq:v00_recast}),
the real $s$ is given by (\ref{eq:s}), satisfies (\ref{eq:s_bound}),
and does not depend on $\mu$.
\end{prop}

\subsubsection{Reconstruction of $v=v_{\mu}$ so that $\protect\Lmu(s,v_{\mu})=f$\label{sssec:Reconstruct_v}}

\paragraph{The set $\mathcal{S}_{\mu}$.}

Let us fix $0<\mu<\mu_{max}$. Let us recall that $e_{n}$ is defined
by (\ref{eq:en_nota}). We set 
\[
\mathcal{S}_{\mu}:=\bigcap_{(n,j)\in I_{\mu}}\left\{ (z,x,y)\mapsto V(z)e_{n}(x)\Gamma_{j}(y): V\in\mathcal{S}_{n,j,\mu}^{\bot}\right\} ^{\bot}\subset\mathcal{Y}_{\mu},
\]
the second orthogonal being taken according to the following hermitian
product on $\mathcal{Y}_{\mu}$:
\[
\langle u,v\rangle_{\mathcal{Y}_{\mu}}=\int_{\mathbb{R}}\int_{0}^{L}\int_{\mathbb{R}}u(z,x,y)\overline{v}(z,x,y)dydxdz.
\]
Since $I_\mu$ is finite, it is clear that $\mathcal{S}_{\mu}$ is non-empty.
Furthermore, $\mathcal{S}_{\mu}$ is closed for the topology associated to
$\langle\cdot,\cdot\rangle$, and also for the topology of $\mathcal{Y}_{\mu}$, by virtue of the dominated convergence theorem. 
 Therefore $\mathcal{S}_{\mu}$ is a Banach space when equipped with
$||\cdot||_{\mathcal{Y}_{\mu}}$ defined by (\ref{eq:Yrond_norm}).
Note also that $\Fmu$ redefined as a function of $\mathbb{R}\times\mathbb{R}\times\mathcal{S}_{\mu}$
to $\mathcal{Z}$ still satisfies conditions $(i)$---$(ii)$ of Theorem \ref{thm:IFT}, since we only restrict the departure space.

\paragraph{Bijectivity of $\protect\Lmu$.}

Let us prove that $\Lmu\colon\mathbb{R}\times\mathcal{S}_{\mu}\to\mathcal{Z}$
given by (\ref{eq:Lpar}) is bijective. From Proposition \ref{prop:summary}
and (\ref{eq:pulsL2_decomp})---(\ref{eq:pulsFourier_decomp}), we
already have that $\Lmu$ is injective. Let us now prove that $\Lmu$
is surjective. Since most of the arguments were already used in subsection
\ref{ss:checking}, we only give a short proof. We are equipped with
$s$ and $v_{j}^{n}=v_{j}^{n}(z)$ provided by Proposition \ref{prop:summary}.
Notice that (\ref{eq:vnj_toProve}) implies 
\begin{equation}
\left|(v_{j}^{n})^{(k)}(z)\right|\leq K||f||_{\mathcal{Z}} \frac{1}{(1+j)^{\beta}(1+|n|)^{\gamma+1-k}},\qquad\forall\mu\in(0,\mu_{max}),\,\forall k\le2.\label{eq:vnj_noMu}
\end{equation}
Now, we define 
\begin{equation}
v_{\mu}(z,x,y):=\sum_{j=0}^{+\infty}\left(\sum_{n=-\infty}^{+\infty}v_{j}^{n}(z)e_{n}(x)\right)\Gamma_{j}(y).\label{eq:v_reconstr}
\end{equation}
Because (\ref{eq:eigenfunc_control_inf})---(\ref{eq:dGamma_control_inf}) hold and $v_{j}^{n}$ satisfies
(\ref{eq:vnj_noMu}) with $\beta>\frac{19}{4}>\frac{9}{4}$ and $\gamma>3>2$,
the function $v_{\mu}$ is well-defined, $L$-periodic in $x$ and
belongs to $C^{2}(\mathbb{R}^{3})$ with
\[
D_{z}^{p}D_{x}^{q}D_{y}^{r}v_{\mu}(z,x,y)=\sum_{j=0}^{+\infty}\left(\sum_{n=-\infty}^{+\infty}(v_{j}^{n})^{(p)}(z)(in\sigma)^{q}e_{n}(x)\right)\Gamma_{j}^{(r)}(y),\qquad p+q+r\leq2.
\]
Similarly, using (\ref{eq:y2_Gamma_control_inf}), since $\beta>\frac{19}{4}>\frac{17}{4}$
and $\gamma>3>2$, and because $K$ does not depend on $\mu,f$, there
holds 
\begin{equation}
\exists C>0,\,\forall\mu\in(0,\mu_{max}),\quad\left|D^{\alpha}v_{\mu}(z,x,y)\right|\leq C||f||_{\mathcal{Z}} \frac{e^{-\kappa|z|}}{(1+y^{2})^{2}},\label{eq:Dalpha_v_bound}
\end{equation}
for any $|\alpha|\leq2$ and $(z,x,y)\in\mathbb{R}^{3}$. Thus $v_{\mu}\in\mathcal{Y}_{\mu}$. By construction \textbf{$v_{\mu}\in\mathcal{S}_{\mu}$} and satisfies
$\Lmu(s,v_{\mu})=f$. Therefore $\Lmu$ is bijective.

\paragraph{Boundedness of $||(\protect\Lmu)^{-1}||$ w.r.t. $\mu$.}

From Proposition \ref{prop:summary} and (\ref{eq:Dalpha_v_bound}), we see
that
\[
||v_{\mu}||_{\mathcal{Y}_{\mu}}\leq(C+K)||f||_{\mathcal{Z}},\quad\forall\mu\in(0,\mu_{max}),
\]
where $C,K>0$ do not depend on $f$. Meanwhile, $s$ satisfies a
similar estimate in (\ref{eq:s_bound}) and does not depend on $\mu$.
As a consequence,
\begin{equation}
\exists C_{\mathcal{L}}>0,\,\forall\mu\in(0,\mu_{max}),\quad||(\Lmu)^{-1}||\leq C_{\mathcal{L}}.\label{eq:L_inv_bound}
\end{equation}

\subsection{Construction of $(s_{\varepsilon,\mu},v_{\varepsilon,\mu})$ solving
$\protect\Fmu(\varepsilon,s,v)=0$\label{subsec:Constr_veps_mu}}

Let us fix $\mu\in(0,\mu_{max})$, and recall that $\ep_0^\ast>0$ has been fixed by Lemma \ref{lem:neps_Yast}. From subsections \ref{subsec:Spaces_Yrond_Zrond},
\ref{subsec:Front_IFT_i_ii} and \ref{sssec:Reconstruct_v}, we can
apply Theorem \ref{thm:IFT} to the function $\Fmu$ at the point
$(0,0,0)$. Hence there are $0<\overline{\varepsilon}_{0}\leq\varepsilon_{0}^{*}$
and $r>0$ that depend \textit{a priori} on $\mu$, such that, for
any $|\varepsilon|<\overline{\varepsilon}_{0}$, the following holds:
there is a unique $s_{\varepsilon,\mu}\in\mathbb{R}$ and $v_{\varepsilon,\mu}\in\mathcal{S}_{\mu}\subset\mathcal{Y}_{\mu}$
for which $|s_{\varepsilon,\mu}|+||v_{\varepsilon,\mu}||_{\mathcal{Y}_{\mu}}\leq r$
and $\Fmu(\varepsilon,s_{\varepsilon,\mu},v_{\varepsilon,\mu})=0$.

\medskip

We shall now prove that $\overline{\varepsilon}_{0},r$ can be selected
independently of $\mu$, which is crucial for letting $\mu \to 0$ in the next subsection. To do so we have to redo the proof of Theorem
\ref{thm:IFT} in a more accurate way than depicted in \cite{Zei_86}, which warrants to be detailed
here. Set 
\begin{alignat*}{2}
T_{\varepsilon,\mu}\colon & \mathbb{R}\times\mathcal{S}_{\mu} \; & \to &  \; \mathbb{R}\times\mathcal{S}_{\mu}\\
 \; & (s,v) & \mapsto &\;  (s,v)-(\Lmu)^{-1}\left(\Fmu(\varepsilon,s,v)\right).
\end{alignat*}
It is clear that $\Fmu(\varepsilon,s,v)=0$ if and only if $(s,v)$
is a fixed point of $T_{\varepsilon,\mu}$. Now, notice that $\Fmu(\varepsilon,s,v)=\Lmu(s,v)+\mathcal{G}_{\mu}(\varepsilon,s,v)$,
where
\begin{align}
\mathcal{G}_{\mu}(\varepsilon,s,v) & = sv_{z}+v\left(2A^{2}\varepsilon y\theta(x)-A^{2}\varepsilon^{2}\theta(x)^{2}-U(z)\int_{\mathbb{R}}(n^{\varepsilon}-n^{0})(x,y^{\prime})dy^{\prime}-\int_{\mathbb{R}}v(z,x,y^{\prime})dy^{\prime}\right)\nonumber \\
 & \quad -U(z)(n^{\varepsilon}-n^{0})\int_{\mathbb{R}}v(z,x,y^{\prime})dy^{\prime}+sU^{\prime}(z)(n^{\varepsilon}-n^{0})+2U^{\prime}(z)n_{x}^{\varepsilon}\nonumber \\
 & \quad +U(z)(1-U(z))n^{\varepsilon}\int_{\mathbb{R}}(n^{\varepsilon}-n^{0})(x,y^{\prime})dy^{\prime}.\label{def:Gmu}
\end{align}
Therefore $T_{\varepsilon,\mu}(s,v)=-(\Lmu)^{-1}\left(\mathcal{G}_{\mu}(\varepsilon,s,v)\right)$.
Note that 
\begin{align*}
\left[D_{(s,v)}\mathcal{G}_{\mu}(\varepsilon,s,v)\right](\tau,w) & = sw_{z}+\tau v_{z}+w\left(2A^{2}\varepsilon y\theta(x)-A^{2}\varepsilon^{2}\theta(x)^{2}-U(z)\int_{\mathbb{R}}(n^{\varepsilon}-n^{0})(x,y^{\prime})dy^{\prime}\right)\\
 & \quad -w\int_{\mathbb{R}}v(z,x,y^{\prime})dy^{\prime}-v\int_{\mathbb{R}}w(z,x,y^{\prime})dy^{\prime}\\
 & \quad -U(z)(n^{\varepsilon}-n^{0})\int_{\mathbb{R}}w(z,x,y^{\prime})dy^{\prime}+\tau U^{\prime}(z)(n^{\varepsilon}-n^{0}).
\end{align*}
Let us recall that $\theta$ satisfies (\ref{eq:thetam_decay}) with
$k+\delta>\gamma+\frac{1}{2}$. In particular, $\theta$ satisfies
(\ref{eq:bm_cond}) with $\rho=k+\delta-\gamma>1/2$ and $K_{b}=K_{\theta}$.
Repeating the same arguments as in subsection  \ref{subsec:Front_IFT_i_ii},
we have
\[
||sw_{z}||_{\mathcal{Z}}\leq C|s|\,||w||_{\mathcal{Y}_{\mu}},\qquad||\tau v_{z}||_{\mathcal{Z}}\leq C|\tau|\,||v||_{\mathcal{Y}_{\mu}},
\]
\[
||y\theta w||_{\mathcal{Z}}\leq\left(C_{\rho}^{\prime}K_{\theta}+||\theta||_{\infty}\right)||w||_{\mathcal{Y}_{\mu}},\qquad||\theta^{2}w||\leq\left(C_{\rho}K_{\theta}+||\theta||_{\infty}\right)||w||_{\mathcal{Y}_{\mu}},
\]
\[
\left\Vert Uw\int_{\mathbb{R}}(n^{\varepsilon}-n^{0})(x,y^{\prime})dy^{\prime}\right\Vert _{\mathcal{Z}}\leq\left(C_{1}K_{A}+\frac{\pi}{2}\right)||n^{\varepsilon}-n^{0}||_{Y^{*}}||w||_{\mathcal{Y}_{\mu}},
\]
\[
\left\Vert w\int_{\mathbb{R}}v(z,x,y^{\prime})dy^{\prime}\right\Vert _{\mathcal{Z}},\left\Vert v\int_{\mathbb{R}}w(z,x,y^{\prime})dy^{\prime}\right\Vert _{\mathcal{Z}}\leq\left(C_{1}K_{A}+\frac{\pi}{2}\right)||v||_{\mathcal{Y}_{\mu}}||w||_{\mathcal{Y}_{\mu}},
\]
\[
\left\Vert U(n^{\varepsilon}-n^{0})\int_{\mathbb{R}}w(z,x,y^{\prime})dy^{\prime}\right\Vert _{\mathcal{Z}}\leq K_\kappa \left(C_{1}K_{A}+\frac{\pi}{2}\right)||w||_{\mathcal{Y}_{\mu}}||n^{\varepsilon}-n^{0}||_{\mathcal{Y}_{\mu}},
\]
\[
\left\Vert \tau U^{\prime}(n^{\varepsilon}-n^{0})\right\Vert _{\mathcal{Z}}\leq C_{U}|\tau|\,||n^{\varepsilon}-n^{0}||_{Y^{*}}.
\]
Consequently, we have 
\[
\left\Vert D_{(s,v)}\mathcal{G}_{\mu}(\varepsilon,s,v)\right\Vert \leq\text{C}\left(||n^{\varepsilon}-n^{0}||_{Y^{*}}+|\varepsilon|+|\varepsilon|^{2}+|s|+||v||_{\mathcal{Y}_{\mu}}\right),
\]
where, crucially, $\text{C}>0$ does not depend on $\mu$. Fix now any $\ell>0$. Since $||n^{\varepsilon}-n^{0}||_{Y^{*}}\xrightarrow[\varepsilon\to0]{}0$
from Lemma \ref{lem:neps_Yast}, we may select $0<r<\ell$ small enough
such that for all $\mu\in(0,\mu_{max})$, we have 
\begin{equation}
\forall\mu\in(0,\mu_{max}),\quad\min\left(|\varepsilon|,\,||(s,v)||_{\mathbb{R}\times\mathcal{S}_{\mu}}\right)\leq r \Longrightarrow \left\Vert D_{(s,v)}\mathcal{G}_{\mu}(\varepsilon,s,v)\right\Vert \leq\frac{1}{2C_{\mathcal{L}}}.\label{eq:___DGmu}
\end{equation}
Then, using (\ref{eq:hxZ_Yast}), we have
\begin{align*}
||\mathcal{G}_{\mu}(\varepsilon,0,0)||_{\mathcal{Z}}= & \left\Vert 2U^{\prime}n_{x}^{\varepsilon}+U(1-U)n^{\varepsilon}\int_{\mathbb{R}}(n^{\varepsilon}-n^{0})(x,y^{\prime})dy^{\prime}\right\Vert _{\mathcal{Z}}\\
\leq & \, 2C_{U}K_\sigma ||n^{\varepsilon}-n^{0}||_{Y^{*}}+\left(C_{1}K_{A}+\frac{\pi}{2}\right)C_{U}||n^{\varepsilon}-n^{0}||_{Y^{*}},
\end{align*}
and thus we may select $0<\overline{\varepsilon}_{0}<\min(r,\varepsilon_{0}^{*})$
small enough such that
\begin{equation}
\forall\mu\in(0,\mu_{max}),\quad\forall|\varepsilon|\leq\overline{\varepsilon}_{0},\qquad||\mathcal{G}_{\mu}(\varepsilon,0,0)||_{\mathcal{Z}}\leq\frac{1}{2C_{\mathcal{L}}}r.\label{eq:___Gmu_eps00}
\end{equation}
Let $B_{r}=\left\{ (s,v)\in\mathbb{R}\times\mathcal{S}_{\mu}: ||(s,v)||_{\mathbb{R}\times\mathcal{S}_{\mu}}\leq r\right\} $
be a closed subset of the Banach space $\mathbb{R}\times\mathcal{S}_{\mu}$.
Note that $\mathcal{G}_{\mu},D_{(s,v)}\mathcal{G}_{\mu}$ are continuous
at $(0,0)$. Then from Taylor's theorem, (\ref{eq:L_inv_bound}) and
(\ref{eq:___DGmu}), we have for any $|\varepsilon|<\overline{\varepsilon}_{0}< r$
and $(s,v),(s^{\prime},v^{\prime})\in B_{r}$,
\begin{align*}
\left\Vert T_{\varepsilon,\mu}(s,v)-T_{\varepsilon,\mu}(s^{\prime},v^{\prime})\right\Vert _{\mathbb{R}\times\mathcal{S}_{\mu}} & \leq||(\Lmu)^{-1}||\,\left\Vert \mathcal{G}_{\mu}(\varepsilon,s,v)-\mathcal{G}_{\mu}(\varepsilon,s^{\prime},v^{\prime})\right\Vert _{\mathbb{R}\times\mathcal{S}_{\mu}}\\
 & \leq C_{\mathcal{L}}\sup_{0<\omega<1}\left\Vert D_{(s,v)}\mathcal{G}_{\mu}(\varepsilon,s+\omega(s^{\prime}-s),v+\omega(v^{\prime}-v))\right\Vert \\
 & \qquad \times ||(s-s^{\prime},v-v^{\prime})||_{\mathbb{R}\times\mathcal{S}_{\mu}}\\
 & \leq\frac{1}{2}||(s-s^{\prime},v-v^{\prime})||_{\mathbb{R}\times\mathcal{S}_{\mu}}.
\end{align*}
Repeating this argument, along with (\ref{eq:___Gmu_eps00}), yields
\begin{align*}
\left\Vert T_{\varepsilon,\mu}(s,v)\right\Vert _{\mathbb{R}\times\mathcal{S}_{\mu}} & \leq||(\Lmu)^{-1}||\,\left\Vert \mathcal{G}_{\mu}(\varepsilon,s,v)-\mathcal{G}_{\mu}(\varepsilon,0,0)\right\Vert _{\mathbb{R}\times\mathcal{S}_{\mu}}+||(\Lmu)^{-1}||\,\left\Vert \mathcal{G}_{\mu}(\varepsilon,0,0)\right\Vert _{\mathbb{R}\times\mathcal{S}_{\mu}}\\
 & \leq\frac{1}{2}||(s,v)||_{\mathbb{R}\times\mathcal{S}_{\mu}}+\frac{1}{2}r\leq r.
\end{align*}
Consequently $T_{\varepsilon,\mu}$ maps $B_{r}$ into itself and
is contractive, thus by the fixed-point theorem it admits a unique
fixed point in $B_{r}$. In conclusion, for any $\ell>0$, we can select $0<\overline{\varepsilon}_{0}<r<\ell$ such that
for each $|\varepsilon|<\overline{\varepsilon}_{0}$ and $\mu\in(0,\mu_{max})$,
there exists a unique $(s_{\varepsilon,\mu},v_{\varepsilon,\mu})\in B_{r}$
satisfying $\Fmu(\varepsilon,s_{\varepsilon,\mu},v_{\varepsilon,\mu})=0$. Since $\ell>0$ was taken arbitrarily, we also have
\begin{equation}
\sup_{0<\mu<\mu_{max}} ||(s_{\ep,\mu},v_{\ep,\mu})||_{\mathbb{R}\times\mathcal{S}_{\mu}} \xrightarrow[\ep \to 0]{} 0.\label{eq:ep_to_0}
\end{equation}

\subsection{Letting the parameter $\mu$ tend to zero}\label{subsec:ascoli}

\global\long\def\Ymup{\widehat{\mathcal{Y}}_{\mu}}
\global\long\def\Zmup{\widehat{\mathcal{Z}}}

Note that so far we only used the fact that $\beta>\frac{19}{4}>\frac{17}{4}$ and $\gamma>3>2$ at most. Since by assumption $\beta>\frac{19}{4}$ and $\gamma>3$,
we may redo the above proof by replacing $\mathcal{Y}_{\mu},\mathcal{Z}$ in (\ref{eq:Yrond}) and (\ref{eq:Zrond})
with
\begin{equation}
\Ymup:=\left\{ v\in C^{3}(\mathbb{R}^{3})\left|\begin{array}{c}
v(z,x+L,y)=v(z,x,y),\quad\text{on }\mathbb{R}^{3},\vspace{7pt}\\
\exists C>0,\,\forall|\alpha|\leq3,\quad\left|D^{\alpha}v(z,x,y)\right|\leq\frac{Ce^{-\kappa|z|}}{(1+y^{2})^{2}}\quad\text{on }\mathbb{R}^{3},\vspace{7pt}\\
\exists K>0,\,\forall n\in\mathbb{Z},\,\forall j\in\mathbb{N},\,\forall z\in\mathbb{R},\quad\text{there holds}\vspace{3pt}\\
|(v_{j}^{n})^{(k)}(z)|\leq\frac{Ke^{-\kappa|z|}}{(1+j)^{\beta}(1+|n|)^{\gamma}}\times\frac{1+|n|^{k}+j^{k/2}}{1+\mu n^{2}+j+|n|},\quad k\leq3\vspace{7pt}
\end{array}\right.\right\},\label{eq:Yhat}
\end{equation}
and
\begin{equation}
\Zmup:=\left\{ f\in C^{1}(\mathbb{R}^{3})\left|\begin{array}{c}
f(z,x+L,y)=f(z,x,y),\quad\text{on }\mathbb{R}^{3},\vspace{7pt}\\
\exists C>0,\,\forall|\alpha|\leq1,\quad\left|D^{\alpha}f(z,x,y)\right|\leq\frac{Ce^{-\kappa|z|}}{1+y^{2}}\quad\text{on }\mathbb{R}^{3},\vspace{7pt}\\
\exists K>0,\,\forall n\in\mathbb{Z},\,\forall j\in\mathbb{N},\,\forall z\in\mathbb{R},\quad\text{there holds}\vspace{3pt}\\
|(f_{j}^{n})^{(k)}(z)|\leq\frac{Ke^{-\kappa|z|}}{(1+j)^{\beta}(1+|n|)^{\gamma}}\times\left(1+|n|^{k}+j^{k/2}\right),\quad k\leq1
\end{array}\right.\right\},\label{eq:Zhat}
\end{equation}
equipped with the respective norms
\begin{align*}
||v||_{\Ymup}= & \sum_{|\alpha|\leq3}\left[\sup_{(z,x,y)\in\mathbb{R}^3}\left|(1+y^{2})^{2}D^{\alpha}v(z,x,y)\right|e^{\kappa|z|}\right]\\
 & +\sum_{k=0}^{3}\sup_{n\in\mathbb{Z},j\in\mathbb{N}}\left[(1+j)^{\beta}(1+|n|)^{\gamma}\frac{1+\mu n^{2}+j+|n|}{1+|n|^{k}+j^{k/2}}\sup_{z\in\mathbb{R}}\left|(w_{j}^{n})^{(k)}(z)e^{\kappa|z|}\right|\right],
\end{align*}
\begin{align*}
||f||_{\Zmup}= & \sum_{|\alpha|\leq1}\left[\sup_{(z,x,y)\in\mathbb{R}^3}\left|(1+y^{2})D^{\alpha}f(z,x,y)\right|e^{\kappa|z|}\right]\\
 & +\sum_{k\in\{0,1\}}\sup_{n\in\mathbb{Z},j\in\mathbb{N}}\left[\frac{(1+j)^{\beta}(1+|n|)^{\gamma}}{1+|n|^{k}+j^{k/2}}\sup_{z\in\mathbb{R}}\left|(f_{j}^{n})^{(k)}(z)e^{\kappa|z|}\right|\right].
\end{align*}
The proof in itself requires only slightly more precision, for example
the Young inequality $\sqrt{j}|n|\leq\frac{2}{3}j^{3/2}+\frac{1}{3}|n|^{3}$,
or the proof of (\ref{eq:bv_bound})---(\ref{eq:ybv_bound}) which
requires to split the summation over $m\in\mathbb{Z}$ into
$m\leq0$, $m\geq n$ and $0\leq m\leq n$ (assuming $n\geq0$). Details are omitted.

\medskip

Let us fix $|\varepsilon|<\overline{\varepsilon}_{0}$. From subsections
\ref{subsec:Spaces_Yrond_Zrond} to \ref{subsec:Constr_veps_mu},
for any $\mu\in(0,\mu_{max})$, we are thus equipped with $(s_{\varepsilon,\mu},v_{\varepsilon,\mu})\in\mathbb{R}\times\Ymup$,
with $|s_{\varepsilon,\mu}|,\,||v_{\varepsilon,\mu}||_{\mathcal{Y}_{\mu}}\leq r$
where $r$ does not depend on $\mu$. Therefore there exists a sequence $(\mu_{m})_{m\in\mathbb{N}}$ in $(0,\mu_{max})$ that tends to zero such
that $s_{\varepsilon,\mu_{m}}\xrightarrow[m\to+\infty]{} s_{\varepsilon}$ for some  $s_{\varepsilon}\in[-r,r]$. On the other hand, if we define for any $k\in\mathbb{N}$,
\[
C_{w,L}^{k}(\mathbb{R}^{3}):=\left\{ g\in C_{b}^{k}(\mathbb{R}^{3}): g(z,x+L,y)=g(z,x,y)\text{ on }\mathbb{R}^{3},\quad||g||_{w,k}<\infty\right\} ,
\] 
\[
w(z,y):=(1+y^{2})^{2}e^{\kappa|z|},\quad||g||_{w,k}:=\sum_{|\alpha|=0}^{k}\sup_{(z,x,y)\in\mathbb{R}^3}\left|w(z,y)D^{\alpha}g(z,x,y)\right|,
\]
we see that $||v_{\varepsilon,\mu_{m}}||_{w,3}\leq||v_{\varepsilon,\mu_{m}}||_{\Ymup}\leq r$.
We claim that a subsequence of $(v_{\varepsilon,\mu_{m}})_{m\in\mathbb{N}}$
converges, as $m\to+\infty$, to some $v_{\varepsilon}\in C_{w_{0},L}^{2}(\mathbb{R}^{3})$,
with $w_{0}(z,x,y):=(1+y^{2})e^{\frac{\kappa}{2}z}$. First,
one can readily check that $C_{w_{0},L}^{2}(\mathbb{R}^{3})$ is complete
for $||\cdot||_{w_{0},2}$. Then, because $(v_{\varepsilon,\mu_{m}})_{m}$
is bounded in $C_{w,L}^{3}(\mathbb{R}^{3})$, for any $\delta>0$
there exist $z_{\delta},y_{\delta}\geq0$ such that for
all $m,n\in\mathbb{N}$ and $|\alpha|\leq 2$,
\[
\left|\left(D^\alpha v_{\varepsilon,\mu_{m}}(z,x,y)-D^\alpha v_{\varepsilon,\mu_{n}}(z,x,y)\right)w_{0}(z,y)\right|\leq\delta,\qquad\forall|z|\geq z_{\delta},\,\forall|y|\geq y_{\delta},\,\forall x\in[0,L].
\]
From there, redoing the proof of the Arzelà-Ascoli theorem, we prove that  with another extraction $(\mu_{m}^{\prime})_{m\in\mathbb{N}}$, we obtain for $m,n$ large enough
\[
\left|\left(D^\alpha v_{\varepsilon,\mu^\prime_{m}}(z,x,y)-D^\alpha v_{\varepsilon,\mu^\prime_{n}}(z,x,y)\right)w_{0}(z,y)\right|\leq\delta,\qquad\forall (z,x,y)\in [-z_{\delta},z_{\delta}]\times[0,L]\times[-y_{\delta},y_{\delta}].
\]
Consequently the sequence $v_{\ep,\mu^\prime_m}$ is uniformly Cauchy in $C_{w_{0},L}^{2}(\mathbb{R}^{
3})$, thus convergent to some 
$v_{\varepsilon}\in C_{w_{0},L}^{2}(\mathbb{R}^{3})$.

\medskip

\begin{proof}[Completion of the proof of Theorem \ref{thm:pulsating}] Let us fix $|\varepsilon|\leq\overline{\varepsilon}_{0}$.
By construction we have $\Fmu(\varepsilon,s_{\varepsilon,\mu_{m}^{\prime}},v_{\varepsilon,\mu_{m}^{\prime}})=0$
for all $m\in\mathbb{N}$. Passing to the limit as $m\to+\infty$,
thanks to the dominated convergence theorem, we obtain $\Ff(\varepsilon,s_{\varepsilon},v_{\varepsilon})=0$.
As a result, 
\[
u^{\varepsilon}(z,x,y)=U(z)n^{\varepsilon}(x,y)+v_{\varepsilon}(z,x,y),\qquad z=x-(c_{0}+s_{\varepsilon})t,
\]
solves (\ref{eq-r}) by construction, and satisfies (\ref{eq:ueps_cond}). Finally, from (\ref{eq:ep_to_0}) combined with
\[
||v_\ep||_{w_0,2} \leq \limsup_{m \to +\infty} ||v_{\ep,\mu^\prime_m}||_{w_0,2} \leq \sup_{m \in \mathbb{N}} ||v_{\ep,\mu^\prime_m}||_{\mathcal{Y}_\mu},
\]
we deduce that $|s_\ep|,||v_{\varepsilon}||_{w_{0},2}\to 0$ as $\ep \to 0$. This yields (\ref{eq:ueps_CV})
with $b=\frac{\kappa}{2}>0$. 
\end{proof}



\section{Insights of the results on the biological model}\label{s:numerics}

In this section, our goal is to discuss some biological implications of our mathematical analysis, completed by some numerical explorations, for a population facing a nonlinear environmental gradient. 

\medskip

Throughout this section, we assume $\theta\in C_b(\R)$ and $0<A<1$ so that $\lambda_0<0$, meaning Theorems \ref{thm:steady_state_eps}, \ref{thm:steady-per} and \ref{thm:steady-C0} hold. 
Letting $\alpha:=\sqrt{2A}$,  (\ref{terme-un}), (\ref{def-terme-un}) and (\ref{def:proba}) are recast
\begin{equation}\label{discussion}
n^\ep(x,y)\approx n^0(y)\left(1 + \ep \underbrace{A\, \rho_\alpha*\theta(x)}_{\text{ deformation}} y +\cdots\right), \quad \rho_\alpha(z):=\frac 12 \alpha e^{-\alpha\vert z\vert}.
\end{equation}
In the following, we discuss two types of error between $n^\ep(x,y)$ and $n^0(y)$: firstly, the so-called relative error, whose leading order term is $A \,\rho_\alpha *\theta(x)\, y=:D(x)y$; secondly, the absolute error, whose leading order term is given by $\overline{D}(x,y):=D(x)yn^0(y)$. 

We first present some general bounds on the two errors. Notice that $||D||_{L^\infty(\R)}\leq A||\theta||_{L^\infty(\R)}$, meaning $D(x)$ remains limited as $A\to 0$, and so does the relative error for bounded $y$. In other words, the shape of populations \lq\lq far from extinction'' ($A$ small) when $\ep=0$ is very robust: such species can dampen the perturbation when $|\ep|\neq 0$. As for $\overline{D}$, thanks to (\ref{def:n0}), we can compute
\begin{align}
\Vert \overline{D} \Vert_{L^\infty(\R^2)} = C (1-A) \Vert D \Vert_{L^{\infty}(\R)}\leq CA(1-A)||\theta||_{L^\infty(\R)}, \label{eq:abs-rela}
\end{align}
for some universal constant $C>0$. Note that, for any $x$, the maximum of $\vert \overline D(x,\cdot)\vert$ is attained at $y=\pm A^{-1/2}$, independently of $\theta$. From (\ref{eq:abs-rela}), the absolute error vanishes both far from extinction ($A\to 0$), and close to extinction ($A\to 1$). In the latter case, this is because $||n^0||_{L^\infty(\R)}$ itself goes to zero, see (\ref{def:n0}).

In the sequel, we shall mostly discuss on $D(x)$, which is tied to the relative error and which we call the {\it deformation}. On the other hand,  expansion (\ref{discussion}) has the advantage to be uniform in $y$ thanks to (\ref{objectif1}) and (\ref{eq:Y_space}), and  our numerical explorations will therefore mainly focus on the absolute error $\overline D(x,y)$.

\begin{example}[Test case]\label{ex:test} If $\theta \equiv 1$, then (\ref{discussion}) yields $n^{\ep}(x,y)\approx n^0(y)\left(1+\ep Ay +\cdots \right)$. On the other hand, in view of equation (\ref{eq:statio_state}), the solution is explicitly computed as (recall (\ref{def:n0}) and Proposition \ref{prop:basis_eigenfunctions})
$$
n^\ep(x,y)=n^\ep(y)=n_0(y-\ep)=\eta C_0 e^{-\frac 12 A (y-\ep)^2}=n_0(y)e^{\ep Ay-\ep ^{2}\frac A 2}\approx n_0(y) \left(1+\ep A y +\cdots\right). 
$$ 
We thus recover that $D(x)\equiv A$.
\end{example}

\subsection{Deformation of the steady state under localized perturbation}\label{ss:bio-loc}

\begin{example}[Localized prototype case] Consider $\theta(x):=\mathbf{1}_{(-\ell,\ell)}(x)$, with $\ell >0$. This $\theta$ is not continuous but we may consider a smooth compactly supported approximation so it does not matter much for our discussion. From (\ref{uniform}), the perturbation is localized  so that we only consider $\vert x\vert \leq \frac \ell 2$, for which we compute
$$
\rho_\alpha *\theta (x)= \frac 1 2\left(\int_{-\ell}^{x} \alpha e^{-\alpha(x-z)}dz+\int_x^{\ell}\alpha e^{-\alpha(z-x)}dz\right)=1-e^{-\alpha \ell }\cosh (\alpha x).
$$
In this case $D(x)=A\left(1-e^{-\sqrt{2A}\,\ell }\cosh (\sqrt{2A}\,x)\right)$ for $\vert x\vert \leq \frac \ell 2$, and
$$
C_{A,\ell}:=\Vert D\Vert _{L^{\infty}\left(-\frac \ell 2,\frac \ell 2\right)}=A\left(1-e^{-\sqrt{2A}\, \ell}\right).
$$
For a given $\ell >0$, $A\mapsto C_{A,\ell}$ is increasing on  $(0,1)$,  $C_{A,\ell} \to 0$ as $A\to 0$, whereas $C_{A,\ell}\to c_\ell:=1-e^{-\sqrt 2\,\ell }$ as $A\to 1$. We thus recover the fact that the population can dampen the perturbation \lq\lq far from extinction'' ($A$ small). On the other hand, populations \lq\lq hardly surviving'' ($A$ close to 1) when $\ep=0$ are  more sensitive to the  perturbation which they suffer with the coefficient $c_\ell$. Notice that letting $\ell \to +\infty$ yields $D(x)\to A$ and we naturally recover the above test case of Example \ref{ex:test}.
\end{example}

\begin{example}[\lq\lq Dirac'' case] Consider  $\theta(x)=\theta_{h}(x) := \frac 1{2h} \mathbf{1}_{(-h,h)}(x)$, with $h>0$. Again, this $\theta_h$ is not continuous, and since $\Vert \theta_h \Vert_{L^\infty(\R)}\to +\infty$ as $h\to 0$, we expect that $\ep_0=\ep_0(h)$ provided by Theorem \ref{thm:steady_state_eps} satisfies $\ep_0(h)\to 0$ as $h\to 0$. Nevertheless, we  formally obtain
\[
D(x)=D_h(x)\to A \rho_\alpha(x), \text{ as $h\to 0$}.
\]
Therefore, a large variation of the optimal trait on a very small spatial range ($h\to 0$) induces 
a deformation which is maximal at the singularity (here $x=0$), and varies like $A^{\frac 32}$.
\end{example}

\subsection{Deformation of the steady state under periodic perturbation}\label{ss:bio-steady}

\begin{example}[Periodic prototype case]\label{ex:periodic} Consider $\theta (x):=\sin (\frac x \ell)$, with $\ell>0$, which is $L=2\pi \ell$-periodic. Then
\begin{equation}\label{n1-convol}
\rho_\alpha *\theta (x)= \Im \int _\R  \rho_\alpha (x-z)e^{i\frac z \ell }dz= \Im e^{i\frac x\ell } \widehat{\rho_\alpha}\left(\frac 1\ell\right)=\frac{\ell ^2\alpha^{2}}{\ell ^2 \alpha ^{2}+1}\sin \left(\frac x \ell \right).
\end{equation}
In this case
\begin{equation}\label{def-C-A-ell}
D(x)=C_{A,\ell} \,\theta(x), \quad C_{A,\ell}:=\frac{2\ell ^2A^2}{2\ell ^2A+1}.
\end{equation}
Hence the deformation is proportional to the perturbation $\theta(x)$ itself. Also, for a given $\ell >0$, $A\mapsto C_{A,\ell}$ is increasing on  $(0,1)$,  $C_{A,\ell} \to 0$ as $A\to 0$, whereas $C_{A,\ell}\to c_\ell :=\frac {2\ell ^2}{2\ell ^2+1}$ as $A\to 1$. We thus recover the fact that the population can dampen the perturbation \lq\lq far from extinction'' ($A$ small). On the other hand, populations \lq\lq hardly surviving'' ($A$ close to 1) when $\ep=0$ are more sensitive to the perturbation which they suffer with the coefficient $c_\ell$. Notice also that $c_\ell \to 0$ as $\ell \to 0$ so that rapidly changing environments are rather harmless (in the sense that the deformation is small). On the other hand, $c_\ell \to 1$ as $\ell \to +\infty$ meaning that, in slowly changing environments, populations hardly surviving when $\ep=0$ fully suffer the perturbation.
\end{example}

\begin{rem}[Influence of $L$] Since the deformation $D(x)$ vanishes as $A\to 0$, let us assume here that $A\in (0,1)$ is fixed. We also fix a 1-periodic profile $\tilde{\theta}(x)$ and set $\theta_L(x):=\tilde{\theta}\left(\frac{x}{L}\right)$. We shall highlight how $D_L(x)$, the deformation corresponding to the perturbation $\theta _L(x)$, is affected by $L$. Firstly, $D_L$ is obviously $L$-periodic. Then, we have 
\begin{align*}
\tilde{D}_L(x):= A^{-1} D_L(Lx)&=(\rho_\alpha * \theta _L)(Lx)=  (\rho_{\tilde{\alpha}} * \tilde{\theta} )(x), \qquad \tilde{\alpha}:=L\alpha=L\sqrt{2A}.
\end{align*}

When $L\to 0$, one can check that $\tilde{D}_L$ converges uniformly to $\Theta:= \int_{0}^{1}\tilde{\theta}(x)dx=\frac 1L \int_0^L \theta _L(x)dx$, so that 
\[
D_L(x) \to A\Theta, \quad \text{uniformly as }L\to 0.
\]
Note that a deformation $A\Theta$ also corresponds to the deformation assuming $\theta _L(x)\equiv \Theta$, see Example \ref{ex:test}. In other words, in a rapidly changing environment, the population is deformed as if the optimal trait was uniformly equal to its average. In particular, if the average is zero, the steady state is not distorted at first order.

On the other hand, as $L\to +\infty$, $\rho_{\tilde \alpha}$ serves as an approximation of identity and  $||\tilde{D}_L-\tilde{\theta }||_{L^\infty(\R)}\to 0$, so that
\[
\Vert D_L - A\theta _L \Vert_{L^\infty(\R)} \to 0,\quad \text{as } L\to +\infty.
\]
Consequently, the deformation is directly proportional to the optimal trait, meaning the population fully suffers from the perturbation. Note that, since $\tilde{\theta }$ is continuous, the profile $\theta _L(x)=\tilde{\theta}\left(\frac{x}{L}\right)$ flattens as $L\to +\infty$. In particular, in the above limit, we could have replaced $\theta _L$ by $x\mapsto \frac{1}{2p} \int_{x-p}^{x+p}\theta _L(z)dz$ for any fixed $p>0$.
\end{rem}

We now present some numerics for the periodic prototype case of Example \ref{ex:periodic}. As mentioned above, we are mainly concerned with the \textit{absolute} error
\begin{align}
E^\ep(x,y) := n^\ep(x,y) - n^0(y) = \ep \overline{D}(x,y) + o(\ep).\label{eq:Err}
\end{align}
To compute $n^\ep(x,y)$ numerically, we consider the Cauchy problem with initial data $n^0(y)$, and retain the asymptotics $t\to +\infty$. The steady state $n^\ep(x,y)$ being unique in a neighborhood of $n^0(y)$, one can reasonably assume such an asymptotic state to be $n^\ep(x,y)$. This is confirmed by comparing with the expected theoretical result from Theorem \ref{thm:steady_state_eps}, see Figures \ref{fig:A09_L1} and \ref{fig:A09_L2}. 

\begin{figure}[h]
	\begin{center}
	\includegraphics[width=0.85\textwidth]{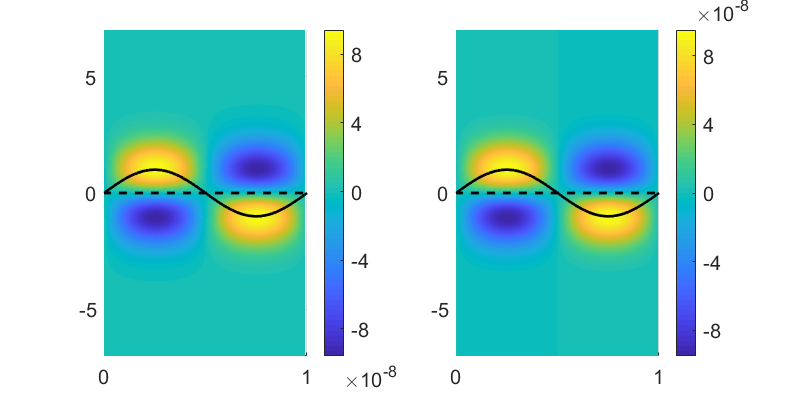}
	\caption{Left: absolute error $E^\ep(x,y) =n^\ep(x,y)-n^0(y)$, where $n^\ep(x,y)$ is determined numerically. Right: theoretical absolute error $\ep \overline{D}(x,y)$. In black, the function $\theta(x)=\sin(2\pi x)$, i.e. $L=1$. In dotted line, the optimal trait $\ep \theta(x)$. Here $A=0.9$ and $\ep=10^{-4}$.}
	\label{fig:A09_L1}
	\end{center}
\end{figure}

\begin{rem}[Absolute error vs. population distribution] In Figure \ref{fig:A09_L1} (and the ones that follow), we represent $\theta(x)$ with a solid, black line. Notice however that this does not correspond to the optimal trait at position $x$, given by $\ep\theta(x)$ and represented with a dotted line in Figure \ref{fig:A09_L1}.

The maximum of the absolute error $|\overline{D}|$ occurs in positions $x$ such that $|D(x)|$ is maximal and with trait $y=\pm y_A := \pm A^{-1/2}$, as mentioned above. As a consequence, at first order, the maximum of $\overline{D}$ occurs at traits $y=\pm y_A$ that do not depend on $\theta$, thus independently of the optimal traits. On the other hand, the positions $x$ where that maximum is attained directly depends on $\theta$ through $D(x)$.

Let us underline that this observation  concerns the absolute error $E^\ep(x,y)$, but not the population distribution $n^\ep(x,y)$ itself. For the latter, we observe numerically that its maximum remains close to $y=0$, for $|\ep|$ small enough. Moreover, thanks to (\ref{objectif1}) and (\ref{eq:Y_space}), we have
\[
\Vert n^\ep -n^0 -\ep D(x)yn^0(y)\Vert_Y = o(\ep),\quad \text{as }\ep \to 0,
\]
so that, keeping only the term corresponding to the index $D^\alpha = D_y$ in (\ref{eq:Y_norm}), and looking at $y=0$, we obtain
\[
\vert n^\ep_y(x,0) - \ep D(x)n^0(0)\vert = o(\ep), \quad \text{as }\ep\to 0.
\]
Consequently, for positions $x$ such that $D(x)\neq 0$, we see that, for $|\ep|$ small enough, $n^\ep_y(x,0)$ is non-zero and has same (opposite) sign as $D(x)$ when $\ep>0$ ($\ep<0$ respectively). In particular, the maximum of $n^\ep(x,y)$ is not attained for traits $y=0$. For those $x$, the maximum of the population size is typically shifted towards the optimal trait. Note that this also applies for non-periodic profile $\theta$.
\end{rem}

\begin{figure}[h]
	\begin{center}
	\includegraphics[width=0.7\textwidth]{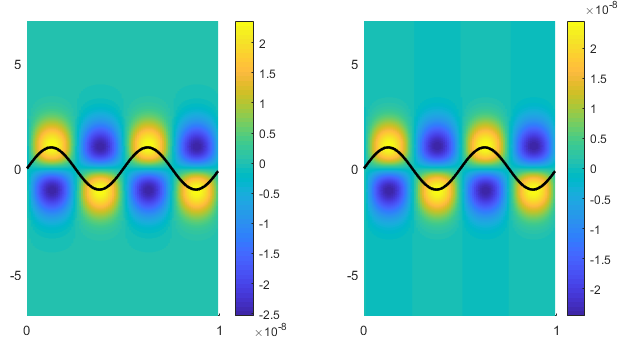}
		\caption{Left: absolute error $E^\ep(x,y) =n^\ep(x,y)-n^0(y)$, where $n^\ep(x,y)$ is determined numerically. Right: theoretical absolute error $\ep \overline{D}(x,y)$. In black, the function  $\theta(x)=\sin(4\pi x)$, i.e. $L=0.5$. Here $A=0.9$ and $\ep=10^{-4}$.}
	\label{fig:A09_L2}
	\end{center}
\end{figure}

Let us pursue with a few comments. Firstly, the error is small near $y=0$ since $\overline{D}(x,0)=0$. Also, we see that $E^\ep(x,y)$ has same sign as $y\theta(x)$, since here $D(x,y)=D(x)yn^0(y)=C_{A,\ell}\theta(x)yn^0(y)$. It can be checked that $\Vert E^\ep-\ep\overline{D} \Vert_{L^\infty(\R^2)}$ decays numerically like $O(\ep^2)$ as $\ep\to 0$.  Let us recall that, from (\ref{eq:abs-rela}),
	\[
	\Vert \overline{D} \Vert_{L^\infty(\R^2)} = C(1-A) C_{A,\ell} = C(1-A) \frac{2\ell^2A^2}{2\ell^2A+1},
	\]
for some universal $C>0$.  Therefore at first order, we expect $E_{max}^\ep :=\Vert E^\ep \Vert_{L^\infty(\R^2)}$ to be increasing with $\ell$, which is highlighted by a comparison of Figures \ref{fig:A09_L1} and \ref{fig:A09_L2} (notice the different scales). 
More generally, $\Vert \overline{D}\Vert_{L^\infty(\R^2)}$ is \lq\lq maximal'' for  $\ell \to +\infty$, $A=\frac 12$.

Last, in order to allow a better comparison with Example \ref{ex:periodic}, we also inquire for the numerical relative error. Note that for $x$ such that $D(x)\neq 0$, the relative error goes to infinity as $|y|\to +\infty$, as can be seen from (\ref{discussion}), and we therefore focus on small values of $y$. We refer to  Figure \ref{fig:Relative_A09_L1}. We have also computed the relative errors for $A\in\{0.8,0.9\}$ and $L\in\{0.5,1\}$. We observed that the numerical outcomes are in agreement with the results discussed in Example \ref{ex:periodic}.

\begin{figure}[h]
	\begin{center}
	\includegraphics[width=0.8\textwidth]{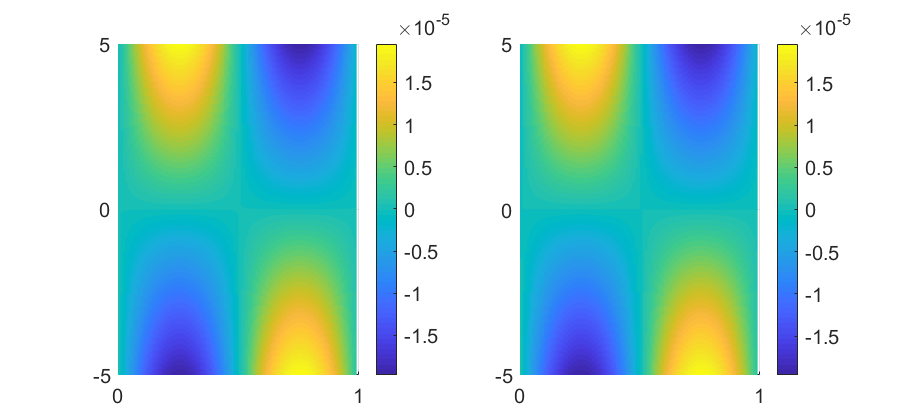}
	\caption{Left: relative error $\frac{n^\ep(x,y)-n^0(y)}{n^0(y)}$, where $n^\ep(x,y)$ is determined numerically. Right: theoretical relative error $\frac{\ep \overline{D}(x,y)}{ n^0(y)}$. Here $A=0.9$, $\ep=10^{-4}$ and $\theta(x)=\sin(2\pi x)$, i.e. $L=1$.}
	\label{fig:Relative_A09_L1}
	\end{center}
\end{figure}

\begin{example}[Influence of skewness] We here perform numerical simulations in   the 1-periodic step function case
\begin{align}
\theta(x)=\begin{cases}
+1 & \text{if } x\in \left(0,\frac{a}{2} \right)\cup\left(1-\frac{a}{2},1\right),\\
-1 & \text{if } x\in\left(\frac{a}{2},1-\frac{a}{2}\right),
\end{cases}\label{theta_marche}
\end{align}
where $0<a<1$ serves as a parameter which measures the asymmetry, or skewness, of the perturbation. Indeed, the optimal trait takes the values $y=+\ep$ and $y=-\ep$ with proportions (over a period) $a$ and $1-a$ respectively. 

In the balanced case $a=\frac 12$, the steady state is symmetrically distorted and, therefore, the location of the maximal absolute error switches between $y=\frac{1}{\sqrt{A}}$ and $y=- \frac {1}{\sqrt{A}}$, see Figure \ref{fig:marche5050}. 
\begin{figure}[h]
	\begin{center}
	\includegraphics[width=0.8\textwidth]{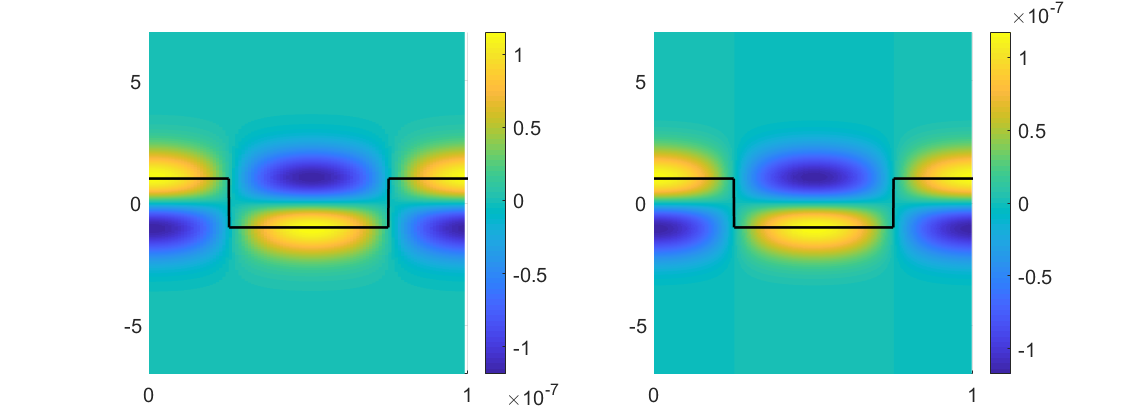}
	\caption{Left: $E^\ep(x,y)=n^\ep(x,y)-n^0(y)$, where $n^\ep(x,y)$ is determined numerically. Right: theoretical $\ep \overline{D}(x,y)$. In black, the 1-periodic function $\theta$ given by (\ref{theta_marche}) with $a=0.5$. Here $A=0.9$ and $\ep=10^{-4}$.}
	\label{fig:marche5050}
	\end{center}
\end{figure}

On the other hand, when $a\to 1$ (the $a\to 0$ case being similar), the $+\ep$ optimum is much more prevalent and, therefore, there is no switch of the maximal absolute error, and the population leans to the upper side, see Figure \ref{fig:marche8020} for $a=0.8$. In other words, there is little advantage for the population to invest on displacements to visit the lower side.
\begin{figure}[h]
	\begin{center}
	\includegraphics[width=0.8\textwidth]{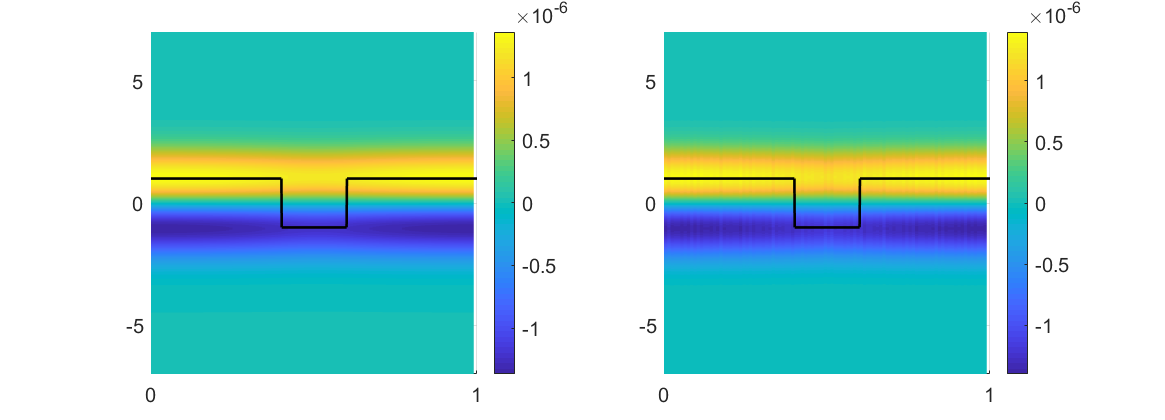}
	\caption{Left: $E^\ep(x,y)=n^\ep(x,y)-n^0(y)$, where $n^\ep(x,y)$ is determined numerically. Right: theoretical $\ep \overline{D}(x,y)$. In black, the 1-periodic function $\theta$ given by (\ref{theta_marche}) with $a=0.8$. Here $A=0.9$ and $\ep=10^{-4}$.}
	\label{fig:marche8020}
	\end{center}
\end{figure}

Last, we consider an intermediate case: Figure \ref{fig:marche5248}, for $a=0.52$, 
reveals that the population suffers less from the perturbation at positions $x$ where the optimal trait is $y=-\ep$ than at other positions.

\begin{figure}[h]
	\begin{center}
	\includegraphics[width=0.8\textwidth]{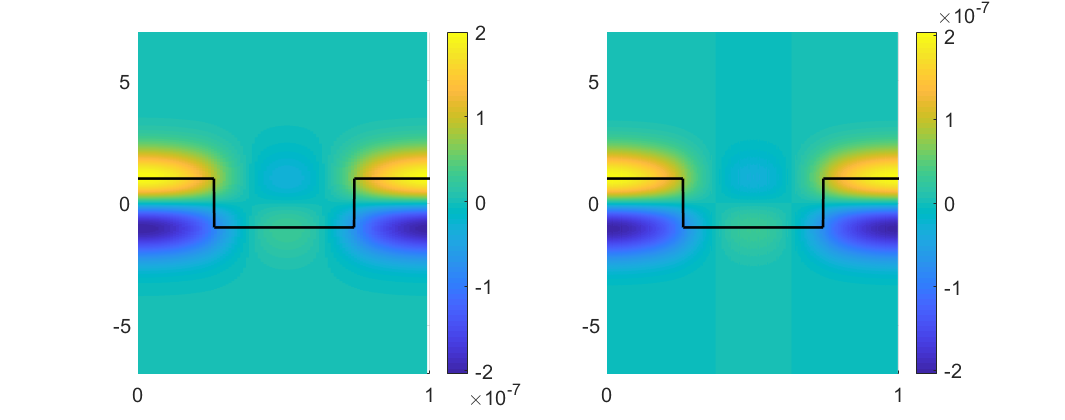}
	\caption{Left: $E^\ep(x,y)=n^\ep(x,y)-n^0(y)$, where $n^\ep(x,y)$ is determined numerically. Right: theoretical $\ep \overline{D}(x,y)$. In black, the 1-periodic function $\theta$ given by (\ref{theta_marche}) with $a=0.52$. Here $A=0.9$ and $\ep=10^{-4}$.}
	\label{fig:marche5248}
	\end{center}
\end{figure}

It is worth mentioning that Figures \ref{fig:marche5050} to \ref{fig:marche5248} highlight that the maximum absolute error increases (notice the different scales) with $\left|a-\frac 12\right|$ (i.e. the aforementioned skewness).

These remarks are consistent with the fact that the absolute error is $D(x)yn^0(y)=(\rho_\alpha * \theta)(x)yn^0(y)$. Indeed, in order to have a positive absolute error at the lower side of position $x$, one must have $D(x)<0$. In the balanced case, one obviously has $D(x)<0$ for all $x\in (0.25,0.75)$, and $|D(x)|$ is maximal at $x=0.5$. When $a=0.8$, we have $D(x)\geq D(0.5)>0$, so that the population always leans towards the upper side, albeit slightly less in $x=0.5$.

In fact, for any fixed $\alpha=\sqrt{2A}$, one can explicitly compute the value $a=a_\alpha$ such that $D(0.5)=0$. We omit the details (tedious but straightforward
cutting of the integral accordingly to the step function, computation of an infinite series and solving of a quadratic equation) and find
\[
a_\alpha=\frac{2}{\alpha}\ln\left(\frac{4}{e^{-\alpha}-1+\sqrt{(e^{-\alpha}-1)^{2}+16e^{-\alpha}}}\right)-1.
\]
Then for any $a>a_\alpha$, we have $D(x)\geq D(0.5)>0$, so that the population leans to the upper side everywhere. For $A=0.9$, we have $a_\alpha \approx 0.5273$, hence our choice of $a=0.52$ for the intermediate case.
\end{example}

\subsection{Deformation of the speed and profile of the front under periodic perturbation}\label{ss:bio-puls}

Here, we formally reproduce the arguments of subsection \ref{ss:proof-th2} (performed to analyse the perturbation of  the steady state) to analyse the perturbation of the pulsating front constructed through Section \ref{s:pulsating}, to which we refer for notations and definitions. We differentiate  $\Fmu(\ep,s_\ep,v_\ep)=0$ with respect to $\ep$ thanks to the chain rule and then evaluate at $\ep=0$ to get
\[
D_{\varepsilon}\Fmu(0,0,0)+\Lmu\left(\left.\frac{\partial s_{\varepsilon,\mu}}{\partial\varepsilon}\right|_{\varepsilon=0},\left.\frac{\partial v_{\varepsilon,\mu}}{\partial\varepsilon}\right|_{\varepsilon=0}\right)=0.
\]
From the expression of $D_{\varepsilon}\Fmu=D_{\varepsilon}\Fmu(\ep,s,v)$ and 
\[
n^{\varepsilon}(x,y)=n^{0}(y)+\varepsilon n^{1}(x,y)+o(\varepsilon)\qquad\text{in }Y,\text{ as }\varepsilon\to 0,
\]
we compute
$$
f(z,x,y):=D_{\varepsilon}\Fmu(0,0,0)  =2U'(z) n_x^{1}(x,y)+U(z)(1-U(z))n^{0}(y)\int_{\mathbb{R}}n^{1}(x,y^{\prime})dy^{\prime}=2U'(z) n_x^{1}(x,y),
$$
since we know from Theorem \ref{thm:steady_state_eps} that $n^{1}(x,y)$ is odd with respect to $y$. From the above and (\ref{def-terme-un}), we reach
$$
\left(\left.\frac{\partial s_{\varepsilon,\mu}}{\partial\varepsilon}\right|_{\varepsilon=0},\left.\frac{\partial v_{\varepsilon,\mu}}{\partial\varepsilon}\right|_{\varepsilon=0}\right)=(\Lmu)^{-1}(f)=(\Lmu)^{-1}\left(\sqrt{2A}\,\eta U'(z)(\rho_A*\theta)'(x)\Gamma _1(y)\right).
$$
Projecting on $(\Gamma _j)$ we thus have $f_j(z,x)=\sqrt{2A}\,\eta U'(z)(\rho_A*\theta)'(x)\, \delta_{j,1}$, where we use the Kronecker symbol. Now, the key point is that $\frac  1 L\int _0^{L}(\rho_A*\theta)'(x)\,dx=0$ so that the Fourier coefficient $f^{0}_{1}(z)\equiv 0$. As a result, recalling (\ref{eq:v00_recast}), $\Phi_f(z)\equiv 0$ so that $s=0$, where $s$ is given by (\ref{eq:s}). In our setting, the latter is recast $\left.\frac{\partial s_{\varepsilon,\mu}}{\partial\varepsilon}\right|_{\varepsilon=0}=0$. Formally letting $\mu \to 0$, this provides
$s_\ep=o(\ep)$ and, thus, (\ref{c1-egal-zero}). 

As explained above, (\ref{c1-egal-zero}) means that the perturbation of the speed of the front by the nonlinearity $\theta=\theta (x)$ is of the second order with respect to $\ep$. As far as the distortion of the profile of the front itself is involved, we focus on the following  example which sheds light on the amplitude of the deformation. 

\begin{example}[Amplitude of the deformation of the profile]\label{ex:periodic2} Here, following Example \ref{ex:periodic}, we consider $\theta(x):=\sin\left(\frac x \ell\right)$, with $\ell>0$, which is $L=2\pi \ell$-periodic. As a result, recalling (\ref{def-terme-un}), (\ref{def:proba}) and (\ref{n1-convol}), we reach
$$
f(z,x,y)=U'(z)\frac{\ell \alpha ^4}{\ell ^2\alpha ^2+1}\cos\left(\frac x\ell\right) yn_0(y)=U'(z)\frac{\ell \alpha ^3\eta}{\ell ^2\alpha ^2+1}\cos\left(\frac x\ell\right) \Gamma_1(y),
$$
where, as above, we use the shortcut $\alpha=\sqrt{2A}$. Projecting on $(\Gamma _j)$, we get
$$
f_j(z,x)= U'(z)\frac{\ell \alpha ^3\eta}{\ell ^2\alpha ^2+1}\cos\left(\frac x\ell\right)\, \delta_{j,1}\,,
$$
whose Fourier coefficients are
$$
f_j^n(z)= U'(z)\frac{\ell \alpha ^3\eta}{2(\ell ^2\alpha ^2+1)}\,\delta_{\vert n\vert,1}\,\delta_{j,1}=:  \mathcal{C} \eta U'(z) \,\delta_{\vert n\vert,1}\,\delta_{j,1}\,,
$$
and where
\begin{equation}
\label{def:C-ronde-A-ell}
\mathcal{C}=\mathcal{C}_{A,\ell}:=\frac{\ell A\sqrt{2A}}{2 \ell ^{2}A+1}.
\end{equation}
 In other words $
f_1^{1}(z)=f_1^{-1}(z)= \mathcal{C}\eta  U'(z)$ and all other coefficients vanish. As a result, the profile of the pulsating front is described by
$$
u^\ep(z,x,y) \approx U(z)n^\ep(x,y)+\ep \left(v_1^{-1}(z)e_{-1}(x)+v_1^1(z)e_1(x)\right)\Gamma_1(y)+\cdots,
$$
where $\mathcal{E}_{\pm 1,1,\mu}[v_{1}^{\pm 1}]=f_1^{\pm 1}(z)=\mathcal{C}\eta U'(z)$, see (\ref{eq:vnj}). Clearly, we have $\overline{v_1^{1}}=v_1^{-1}$ so that
\begin{eqnarray*}
u^\ep(z,x,y) &\approx & U(z)n^\ep(x,y)+\ep\, 2 \Re \left(v_1^1(z)e_1(x)\right)\Gamma_1(y)+\cdots\vspace{10pt}\\
& \approx & U(z)n^0(y)\Big(1+\ep\, C_{A,\ell} \theta(x) y+\ep\, 2 \frac{\Re \left(v_1^1(z)e_1(x)\right)}{\eta U(z)}\sqrt{2A}\, y +\cdots\Big).
\end{eqnarray*}
Here we have used Example \ref{ex:periodic}, in particular $C_{A,\ell}$ is given by (\ref{def-C-A-ell}). Next, since $\ell\sqrt{2A}\, \mathcal{C}_{A,\ell}=C_{A,\ell}$, we end up with
$$
u^\ep(z,x,y) \approx   U(z)n^0(y)\left(1+\ep\, C_{A,\ell} \left(\theta(x) + 2 \frac{\Re\, \left( \mathcal{L}_{1,1,\mu}^{-1}[U^{\prime}](z) e_1(x) \right)}{\ell U(z)}\right)y +\cdots\right).
$$

At this stage, since the term $w(z):=\mathcal{L}_{1,1,\mu}^{-1}[U^{\prime}](z)$ also depends on $A$ and $\ell$, the amplitude of (the leading order term of) the deformation of the profile of the front is not transparent. Nevertheless we can formally obtain some clues in some asymptotic regimes. Recall that, up to letting $\mu \to 0$, $w$ solves
\begin{equation}\label{bidule-truc}
w''+\left(\frac{2i}\ell+c_0\right)w'-\left(\lambda_1-\lambda_0U(z)+\frac 1{\ell ^{2}}\right)w=U'(z).
\end{equation}

Letting $\ell \to 0$, (\ref{bidule-truc}) formally provides $w(z)\sim -\ell ^{2}U'(z)$ so that   $\frac{\Re\, \left( w(z) e_1(x) \right)}{\ell U(z)}$ is of \lq\lq magnitude $\ell$'', 
and thus $u^{\ep}(z,x,y)\approx U(z)n^{0}(y)\left(1+\ep C_{A,\ell} \theta(x)y+\cdots\right)$. On the other hand, letting $\ell \to +\infty$, (\ref{bidule-truc}) formally shows that $w(z)$ is independent on $\ell$ so that $\frac{\Re\, \left( w(z) e_1(x) \right)}{\ell U(z)}$ is of \lq\lq magnitude $1 /\ell$'' and thus, again, $u^{\ep}(z,x,y)\approx U(z)n^{0}(y)\left(1+\ep C_{A,\ell} \theta(x)y+\cdots\right)$. As a result, at least in any of the asymptotic regimes $\ell \to 0$, $\ell \to +\infty$,   the amplitude of (the leading order term of) the deformation of the profile of the front is again measured by $C_{A,\ell}$, so that the biological insights are similar to those of Example \ref{ex:periodic}.
 
 On the other hand, letting $A\to 0$ or $A\to 1$, $w(z)$ formally becomes independent on $A$ and thus $$
u^{\ep}(z,x,y)\approx U(z)n^{0}(y)\Big(1+\ep C_{A,\ell} (\theta(x)+\Psi _\ell (z,x))y+\cdots\Big),
$$
so that an additional deformation term, denoted $\Psi_\ell (z,x)$, is involved.

\end{example}

\appendix

\section{Appendix}

\subsection{Proof of Lemma \ref{lem:homo_nj}\label{sec:Lemmaphi}}

We first need to construct solutions
of (\ref{eq:homo}) on $\mathbb{R}_{-}$ and $\mathbb{R}_{+}$. The proof mainly consists in rewriting the ordinary differential equation (\ref{eq:homo}) as a fixed point problem, and then to perform careful estimates by considering separately large values of $\min(\vert n \vert,j)$ from bounded values of $\min(\vert n \vert,j)$.

\begin{lem}[Fundamental system of (\ref{eq:homo}) on $\mathbb{R}_{-}$ and $\mathbb{R}_{+}$]
\label{lem:phi_tilde}Let $(n,j)\in\mathbb{Z}\times\mathbb{N}$ with
$(n,j)\neq(0,0)$, and $0\leq\mu<1$. On $\mathbb{R}_{-}$, we can
construct a system of fundamental solutions $(\tilde{\varphi}_{-},\tilde{\varphi}_{+})$
of (\ref{eq:homo}), such that 
\begin{equation}
\tilde{\varphi}_{\pm}(z)=\tilde{P}_{\pm}(z)e^{a_{n,j,\mu}^{\pm}z},\qquad\tilde{P}_{\pm}\in C_{b}^{2}(\mathbb{R}_{-},\mathbb{C}),\qquad\liminf_{z\to-\infty}\left|\tilde{P}_{-}(z)\right|>0,\label{eq:homo_sol_R-}
\end{equation}
with $a_{n,j,\mu}^{\pm}$ given by (\ref{eq:a_nj}). On $\mathbb{R}_{+}$,
we can construct a system of fundamental solutions $(\tilde{\psi}_{-},\tilde{\psi}_{+})$
of (\ref{eq:homo}) such that 
\begin{equation}
\tilde{\psi}_{\pm}(z)=\tilde{Q}_{\pm}(z)e^{b_{n,j,\mu}^{\pm}z},\qquad\tilde{Q}_{\pm}\in C_{b}^{2}(\mathbb{R}_{+},\mathbb{C}),\qquad\liminf_{z\to+\infty}\left|\tilde{Q}_{+}(z)\right|>0,\label{eq:homo_sol_R+}
\end{equation}
with $b_{n,j,\mu}^{\pm}$ given by (\ref{eq:b_nj}). Also, there is
$\tilde{R}_{max}>0$ such that
\begin{equation}
\sup_{(n,j)\neq(0,0)}\sup_{0\leq\mu<1}\sup_{\tilde{R}\in\{\tilde{P}_{\pm},\tilde{Q}_{\pm}\}}\left(||\tilde{R}||_{L^{\infty}}+||\tilde{R}^{\prime}||_{L^{\infty}}\right)\leq\tilde{R}_{max},\label{eq:PQtilde_bounds}
\end{equation}
where by convention the sup norm is taken over the domain of definition 
of $\tilde{R}$. 

Additionally, there exist $n_{0},j_{0}>0$ such that if $|n|\geq n_{0}$
or $j\geq j_{0}$, there holds for all $0\leq\mu<1$,
\begin{equation}
\left|\tilde{P}_{+}(0)-1\right|,\left|\tilde{Q}_{-}(0)-1\right|\leq\frac{1}{2},\qquad\tilde{P}_{-}(0)=\tilde{Q}_{+}(0)=1,\label{eq:PQtilde_0}
\end{equation}
\begin{equation}
\left|\tilde{P}_{+}^{\prime}(0)\right|,\left|\tilde{Q}_{-}^{\prime}(0)\right|\leq1,\qquad\tilde{P}_{-}^{\prime}(0)=\tilde{Q}_{+}^{\prime}(0)=0.\label{eq:PpQptilde0}
\end{equation}

Besides, denoting $\tilde{P}_{\pm}=\tilde{P}_{\pm}^{n,j,\mu}$ and
$\tilde{Q}_{\pm}=\tilde{Q}_{\pm}^{n,j,\mu}$, we have for any $N_{0},J_{0}>0$
\begin{equation}
\sup_{\substack{|n|\leq N_{0},j\leq J_{0}\\
(n,j)\neq(0,0)
}
}\sup_{\tilde{R}\in\{\tilde{P}_{\pm},\tilde{Q}_{\pm}\}}\left(\left|\tilde{R}^{n,j,\mu}(0)-\tilde{R}^{n,j,0}(0)\right|+\left|\left(\tilde{R}^{n,j,\mu}\right)^{\prime}(0)-\left(\tilde{R}^{n,j,0}\right)^{\prime}(0)\right|\right)\xrightarrow[\mu\to0]{}0.\label{eq:PQtilde_muto0}
\end{equation}

Next, by taking $\tilde{\mu}_{max}>0$ small enough, there exists
$W_{min}>0$ such that for all $(n,j)\neq(0,0)$ and $0\leq\mu<\tilde{\mu}_{max}$,
the Wronskians of $(\tilde{\varphi}_{-},\tilde{\varphi}_{+})$ and
$(\tilde{\psi}_{-},\tilde{\psi}_{+})$ in zero satisfy:
\begin{equation}
\begin{cases}
\left|W_{\tilde{\varphi}}\right|:=\left|\left[\tilde{\varphi}_{-}^{\prime}\tilde{\varphi}_{+}-\tilde{\varphi}_{+}^{\prime}\tilde{\varphi}_{-}\right](0)\right|\geq W_{min},\vspace{5pt}\\
\left|W_{\tilde{\psi}}\right|:=\left|\left[\tilde{\psi}_{-}^{\prime}\tilde{\psi}_{+}-\tilde{\psi}_{+}^{\prime}\tilde{\psi}_{-}\right](0)\right|\geq W_{min},
\end{cases}\label{eq:Wmin}
\end{equation}
and if $|n|\geq n_{0}$ or $j\geq j_{0}$, we have as well 
\begin{equation}
\left|W_{\tilde{\varphi}}\right|\geq\frac{1}{4}\left|a^{+}_{n,j,\mu}-a^{-}_{n,j,\mu}\right|+1,\qquad\left|W_{\tilde{\psi}}\right|\geq\frac{1}{4}\left|b^{+}_{n,j,\mu}-b^{-}_{n,j,\mu}\right|+1.\label{eq:Wmin_bignj}
\end{equation}

Furthermore, there exist $\zeta_{1},\zeta_{2}>0$ such that for all
$(n,j)\neq(0,0)$ and $0\leq\mu<\tilde{\mu}_{max}$
\begin{equation}
\int_{-\infty}^{0}|\tilde{\varphi}_{+}(z)|^{2}dz\geq\zeta_{1}e^{-\zeta_{2}\Re a^{+}_{n,j,\mu}}.\label{eq:int_phi_tilde}
\end{equation}
\end{lem}

\begin{proof}
In the context of this proof, we always assume $(n,j)\neq(0,0)$.
Also, for the sake of readability, we drop the ``tilde'' notations
for $\tilde{\varphi},\tilde{\psi},\tilde{P},\tilde{Q}$ and denote
$a_{n,j,\mu}^{\pm}=a^{\pm}$, $b_{n,j,\mu}^{\pm}=b^{\pm}$. We first
construct the solutions $\varphi_{\pm}$. Let us fix $n,j$ and $0\leq\mu<1$.
We only treat the case $j\geq1$, the proof for $j=0$ being similar.
Set $\varphi_{\pm}(z)=P_{\pm}(z)e^{a^{\pm}z}$ where $P_{\pm}\in C_{b}^{2}(\mathbb{R}_{-},\mathbb{C})$
is to be determined. Plugging it into (\ref{eq:homo}), we obtain
\begin{equation}
P_{\pm}^{\prime\prime}\pm rP_{\pm}^{\prime}-\lambda_{0}(1-U(z))P_{\pm}=0,\label{eq:___cauchypb}
\end{equation}
where $r:=2a^{+}+2in\sigma+c_{0}=a^{+}-a^{-}$, so that $\Re r>0$ from (\ref{eq:Re_a-a_sign}).

Let us first construct $\varphi_{+}$. Using a Sturm-Liouville approach, we may recast (\ref{eq:___cauchypb})
as
\[
(P_{+}^{\prime}e^{rz})^{\prime}-\lambda_{0}(1-U(z))P_{+}e^{rz}=0,
\]
so that, assuming $P_{+}^{\prime}(-\infty)=0$, we obtain after integration on $(-\infty,z)$,
\begin{equation}
P_{+}^{\prime}(z)=\lambda_{0}\int_{-\infty}^{z}e^{r(\omega-z)}(1-U(\omega))P_{+}(\omega)d\omega,\label{eq:___P+p}
\end{equation}
and thus, assuming $P_{+}(-\infty)=1$, after another integration and a ,
\begin{equation}
P_{+}(z)=1-\lambda_{0}\int_{-\infty}^{z}\frac{e^{r(\omega-z)}-1}{r}(1-U(\omega))P_{+}(\omega)d\omega.\label{eq:___P+}
\end{equation}
Hence, $P_{+}$ is written as the solution of a fixed-point problem.
Since $1-U\in L^{1}(\mathbb{R}_{-})$, for a given $z_{0}\leq0$,
the operator in the right-hand side of (\ref{eq:___P+}) is globally
Lipschitz continuous on $C_{b}\left((-\infty,z_{0}],\mathbb{C}\right)$
with Lipschitz constant $2\left|\lambda_{0}r^{-1}\right|\int_{-\infty}^{z_{0}}(1-U(\omega))d\omega$.
Hence, for $|z_{0}|$ large enough,
the fixed-point theorem yields the existence and uniqueness of a solution
$P_{+}\in C_{b}\left((-\infty,z_{0}],\mathbb{C}\right)$ to the problem
(\ref{eq:___P+}). One can readily check that $P_{+}$ indeed solves
(\ref{eq:___cauchypb}) and belongs to $C_{b}^{2}\left((-\infty,z_{0}],\mathbb{C}\right)$.
We extend it to $\mathbb{R}_{-}$ by solving the Cauchy problem associated
to (\ref{eq:___cauchypb}). We have therefore constructed a function
$\varphi_{+}(z)=P_{+}(z)e^{a^{+}z}$ that solves (\ref{eq:homo}).

We now construct $\varphi_{-}$. We can repeat the same procedure, and obtain that $\varphi_{-}(z)=P_{-}(z)e^{a^{-}z}$
solves (\ref{eq:homo}) if and only if $P_{-}$ satisfies 
\[
(P_{-}^{\prime}e^{-rz})^{\prime}-\lambda_{0}(1-U(z))P_{-}e^{-rz}=0.
\]
By integrating on $[z,z_{0}]$ instead of $(-\infty,z]$, and assuming
$P_{-}^{\prime}(z_{0})=0$, $P_{-}(z_{0})=1$, we deduce successively
that 
\begin{equation}
P_{-}^{\prime}(z)=-\lambda_{0}\int_{z}^{z_{0}}e^{-r(\omega-z)}(1-U(\omega))P_{-}(\omega)d\omega,\label{eq:___P-p}
\end{equation}
\begin{equation}
P_{-}(z)=1+\lambda_{0}\int_{z}^{z_{0}}\frac{1-e^{-r(\omega-z)}}{r}(1-U(\omega))P_{-}(\omega)d\omega,\label{eq:___P-}
\end{equation}
so that $P_{-}$ solves a fixed-point problem. Assuming $|z_{0}|$
large enough, there exists a unique solution $P_{-}\in C_{b}\left((-\infty,z_{0}],\mathbb{C}\right)$
by the fixed-point theorem. One can then readily check that $P_{-}\in C_{b}^{2}\left((-\infty,z_{0}],\mathbb{C}\right)$,
and after extending it to $\mathbb{R}_{-}$ by solving the Cauchy
problem associated to (\ref{eq:___cauchypb}), we obtain another solution
$\varphi_{-}(z)=P_{-}(z)e^{a^{-}z}$ of (\ref{eq:homo}) on $\mathbb{R}_{-}$.
Finally, the solutions $(\varphi_{-},\varphi_{+})$ are linearly independent
since $\Re a^{+}\neq\Re a^{-}$.

We shall now prove that there exist $n_{0},j_{0}>0$ such that $P_{\pm}$
satisfy (\ref{eq:PQtilde_0})---(\ref{eq:PpQptilde0}) when $|n|\geq n_{0}$
or $j\geq j_{0}$. Also, for those indexes $n,j$, we shall prove
that $P_{\pm}$ satisfy (\ref{eq:homo_sol_R-}) and (\ref{eq:PQtilde_bounds}).
Here we denote $r=r_{n,j,\mu}$. Since (\ref{eq:a-a_bound}) and (\ref{eq:Re_a-a_noMu})
hold, there exist $n_{0},j_{0}>0$ such that if $|n|\geq n_{0}$ or
$j\geq j_{0}$, we have
\begin{equation}
\left|r_{n,j,\mu}\right|\geq8,\quad\frac{|\lambda_{0}|}{\Re r_{n,j,\mu}}\leq\frac{2}{3},\quad\left|\frac{\lambda_{0}}{r_{n,j,\mu}}\right|\int_{-\infty}^{0}(1-U(\omega))d\omega\leq\frac{1}{6},\qquad\forall\mu\in[0,1).\label{eq:___rnj_bignj}
\end{equation}
Let us assume that $|n|\geq n_{0}$ or $j\geq j_{0}$. Then (\ref{eq:___P+})
and (\ref{eq:___P-}) hold for any $z\leq z_{0}=0$, independently
of $0\leq\mu<1$. Therefore $P_{-}(0)=1$, $P_{-}^{\prime}(0)=0$
and
\[
||P_{\pm}||_{C_{b}^{0}(\mathbb{R}_{-})}\leq1+\left(2\left|\frac{\lambda_{0}}{r_{n,j,\mu}}\right|\int_{-\infty}^{0}(1-U(\omega))d\omega\right)||P_{\pm}||_{C_{b}^{0}(\mathbb{R}_{-})}\leq1+\frac{1}{3}||P_{\pm}||_{C_{b}^{0}(\mathbb{R}_{-})},
\]
and thus $||P_{\pm}||_{C_{b}^{0}(\mathbb{R}_{-})}\leq\frac{3}{2}$.
Combining this bound with (\ref{eq:___P+p})--- (\ref{eq:___rnj_bignj}),
we deduce on the one hand,
\[
||P_{\pm}-1||_{C_{b}^{0}(\mathbb{R}_{-})}\leq2\left|\frac{\lambda_{0}}{r_{n,j,\mu}}\right|||P_{\pm}||_{C_{b}^{0}(\mathbb{R}_{-})}\int_{-\infty}^{0}(1-U(\omega))d\omega\leq\frac{1}{2},
\]
and on the other hand,
\[
||P_{\pm}^{\prime}||_{C_{b}^{0}(\mathbb{R}_{-})}\leq\left|\lambda_{0}\right|\int_{-\infty}^{0}e^{\omega \Re r_{n,j,\mu}}(1-U(\omega))\left|P_{\pm}(\omega)\right|d\omega\leq\frac{3\left|\lambda_{0}\right|}{2\Re r_{n,j,\mu}}\leq1.
\]
In conclusion, assuming $|n|\geq n_{0}$ or $j\geq j_{0}$, $P_{\pm}$
satisfy (\ref{eq:homo_sol_R-}) and (\ref{eq:PQtilde_bounds})---(\ref{eq:PpQptilde0})
with $\tilde{R}_{max}=\frac{5}{2}$.

Let us now fix $n,j$ such that $|n|\leq n_{0}$ and $j\leq j_{0}$.
We shall prove that, up to taking $\tilde{R}_{max}$ possibly larger,
$P_{\pm}$ satisfy (\ref{eq:homo_sol_R-}) and (\ref{eq:PQtilde_bounds}).
Note that
\[
|r_{n,j,\mu}|\leq|r_{n_{0},j_{0},1}|\eqqcolon r_{max}>0,
\]
while we also have, since $(n,j)\neq(0,0)$,
\begin{equation}
|r_{n,j,\mu}|\geq\Re r_{n,j,\mu}\geq\inf_{(n,j)\neq(0,0)}\Re r_{n,j,0}\geq\min(\Re r_{1,0,0},r_{0,1,0})\eqqcolon r_{min}>0.\label{eq:___rmin}
\end{equation}
We now select $z_{0}\leq0$ independent of $n,j,\mu$ such that 
\begin{equation}
|\lambda_{0}|\int_{-\infty}^{z_{0}}(1-U(\omega))d\omega\leq\min\left(\frac{2}{3},\frac{r_{min}}{6}\right),\label{eq:___z0_smallnj}
\end{equation}
and thus (\ref{eq:___P+}) and (\ref{eq:___P-}) hold for any $z\leq z_{0}$.
Similarly as above, we deduce
\begin{equation}
||P_{\pm}-1||_{C_{b}^{0}\left((-\infty,z_{0})\right)}\leq\frac{1}{2},\quad||P_{\pm}^{\prime}||_{C_{b}^{0}\left((-\infty,z_{0})\right)}\leq1.\label{eq:___P_bounds}
\end{equation}
From there, we recall that the functions $P_{\pm}$ are extended to
$\mathbb{R}_{-}$ by solving the Cauchy problem associated to (\ref{eq:___cauchypb}),
which we recast
\[
Y_{\pm}^{\prime}=A_{\pm}(z)Y_{\pm},\qquad Y_{\pm}=\left(\begin{array}{c}
P_{\pm}\\
P_{\pm}^{\prime}
\end{array}\right),\quad A_{\pm}(z)=\left(\begin{array}{cc}
0 & 1\\
\mp r_{n,j,\mu} & \lambda_{0}(1-U(z))
\end{array}\right).
\]
If we denote $||\cdot||_{\infty}$ both the supremum norm on $\mathbb{C}^{2}$
and its associated subordinate norm on $M_{2}(\mathbb{C})$, we thus
obtain 
\[
||Y_{\pm}^{\prime}(z)||_{\infty}\leq||A_{\pm}(z)||_{\infty}||Y_{\pm}(z)||_{\infty}\leq M||Y_{\pm}(z)||_{\infty},\qquad M:=\left\Vert \left(\begin{array}{cc}
0 & 1\\
r_{max} & \left|\lambda_{0}\right|
\end{array}\right)\right\Vert _{\infty}.
\]
Using the Gronwall's Lemma, this leads to, for all $z_{0}\le z\leq0$:
\[
||Y_{\pm}(z)||_{\infty}\leq||Y_{\pm}(z_{0})||_{\infty}e^{M(z-z_{0})}\leq\frac{3}{2}e^{M|z_{0}|},\qquad\forall|n|\leq n_{0},\,\forall j\leq j_{0},\,\forall\mu\in[0,1).
\]
In conclusion, combining this paragraph and the previous one, we deduce
that $P_{\pm}$ satisfy (\ref{eq:homo_sol_R-}) and (\ref{eq:PQtilde_bounds}) 
for $\tilde{R}_{max}=\max\left(\frac{5}{2},\frac{3}{2}e^{M|z_{0}|}\right)$. Note that $n_0,j_0,z_0$ do not depend on $\mu$, so that this is also the case for $\tilde{R}_{max}$.

Let us prove that $W_{\tilde{\varphi}}$ satisfies (\ref{eq:Wmin_bignj}). Given that
\begin{equation}
W_{\tilde{\varphi}}=\left[(a^{-}-a^{+})P_{+}P_{-}+P_{-}^{\prime}P_{+}-P_{+}^{\prime}P_{-}\right](0),\label{eq:___Wphi}
\end{equation}
and using (\ref{eq:PQtilde_0})---(\ref{eq:PpQptilde0}) with (\ref{eq:___rnj_bignj}),
we have, for all $0\leq\mu<1$, 
\[
\left|W_{\tilde{\varphi}}\right|\geq\frac{1}{2}\left|a^{+}-a^{-}\right|-1\geq\frac{1}{4}|a^{+}-a^{-}|+1,\qquad\text{ if }|n|\geq n_{0}\text{ or }j\geq j_{0},
\]
so that (\ref{eq:Wmin_bignj}) holds. 

We now fix $N_{0},J_{0}>0$ and show that $P_{\pm}$ satisfies (\ref{eq:PQtilde_muto0}).
We first consider fixed indexes $|n|< n_{0}$ and $j< j_{0}$.
Let us recall that for those $n,j$ we selected $z_{0}\leq0$
such that (\ref{eq:___z0_smallnj}) holds for all $0\leq\mu<1$, which
means (\ref{eq:___P+}) and (\ref{eq:___P-}) hold for any $z\leq z_{0}$.
To begin with, we first prove that
\begin{equation}
\sup_{|n|\leq n_{0},j\leq j_{0}}\left(\left|P_{\pm}^{n,j,\mu}(z_{0})-P_{\pm}^{n,j,0}(z_{0})\right|+\left|\left(P_{\pm}^{n,j,\mu}\right)^{\prime}(z_{0})-\left(P_{\pm}^{n,j,0}\right)^{\prime}(z_{0})\right|\right)\xrightarrow[\mu\to0]{}0,\label{eq:___Ptilde_muto0_z0}
\end{equation}
where we denoted $P_{\pm}=P_{\pm}^{n,j,\mu}$. Let us mention that
$P_{-}$ satisfies (\ref{eq:___Ptilde_muto0_z0}) since by construction
$P_{-}^{n,j,\mu}(z_{0})=1$ and $(P_{-}^{n,j,\mu})^{\prime}(z_{0})=0$
for all $0\leq\mu<1$. It thus suffices to show that $P_{+}^{n,j,\mu}$
satisfies (\ref{eq:___Ptilde_muto0_z0}). Fix $\varepsilon>0$. We
set
\[
g_{n,j,\mu}(z):=\frac{e^{r_{n,j,\mu}z}-1}{r_{n,j,\mu}},\quad0\leq\mu<1,\quad z\leq0.
\]
Note that, due to (\ref{eq:___rmin}), we have $||g_{n,j,\mu}||_{\infty}\leq\frac{2}{r_{min}}$.
Also, we fix $z_{\varepsilon}\leq z_{0}$ such that $
\int_{-\infty}^{z_{\varepsilon}}(1-U(\omega))d\omega\leq\varepsilon$. 
Consequently, we have for all $z\leq z_{0}$
\begin{align*}
\left|P_{+}^{n,j,\mu}(z)-P_{+}^{n,j,0}(z)\right| & \leq\left|\lambda_{0}\right|\left|\int_{-\infty}^{z_{\varepsilon}}\left[(g_{n,j,\mu}-g_{n,j,0})(\omega-z)\right](1-U(\omega))P_{+}^{n,j,\mu}(\omega)d\omega\right|\\
 & \quad+\left|\lambda_{0}\right|\left|\int_{z_{\varepsilon}}^{z}\left[(g_{n,j,\mu}-g_{n,j,0})(\omega-z)\right](1-U(\omega))P_{+}^{n,j,\mu}(\omega)d\omega\right|\boldsymbol{1}_{(z_{\varepsilon},z_{0}]}(z)\\
 & \quad+\left|\lambda_{0}\right|\left|\int_{-\infty}^{z}g_{n,j,0}(\omega-z)(1-U(\omega))\left[P_{+}^{n,j,\mu}(\omega)-P_{+}^{n,j,0}(\omega)\right]d\omega\right|\\
 & \leq2\varepsilon\left|\lambda_{0}\right|||g_{n,j,\mu}||_{\infty}||P_{+}^{n,j,\mu}||_{L^\infty(\R^-)}\\
 & \quad+\left|\lambda_{0}\right|||P_{+}^{n,j,\mu}||_{L^\infty(\R^-)}\sup_{z_{\varepsilon}\leq\omega\leq0}\left|(g_{n,j,\mu}-g_{n,j,0})(\omega)\right|\int_{-\infty}^{0}(1-U(\omega))d\omega\\
 & \quad+\left|\lambda_{0}\right|||g_{n,j,0}||_{\infty}\int_{-\infty}^{z}(1-U(\omega))\left|P_{+}^{n,j,\mu}(\omega)-P_{+}^{n,j,0}(\omega)\right|d\omega.
\end{align*}
Also, one can readily check that 
\[
\sup_{|n|< n_{0},j< j_{0}}\sup_{z_{\varepsilon}\leq\omega\leq0}\left|(g_{n,j,\mu}-g_{n,j,0})(\omega)\right|\xrightarrow[\mu\to0]{}0.
\]
Therefore there exists $\mu_{\varepsilon}>0$ such that for any $|n|< n_{0}$,
$j< j_{0}$, $0\leq\mu\leq\mu_{\varepsilon}$ and $z\leq z_{0}$,
there holds
\[
\left|P_{+}^{n,j,\mu}(z)-P_{+}^{n,j,0}(z)\right|\leq C\varepsilon+D\int_{-\infty}^{z}(1-U(\omega))\left|P_{+}^{n,j,\mu}(\omega)-P_{+}^{n,j,0}(\omega)\right|d\omega,
\]
\[
C:=\left|\lambda_{0}\right|\tilde{R}_{max}\left(\frac{4}{r_{min}}+\int_{-\infty}^{0}(1-U(\omega))d\omega\right)>0,\qquad D:=\frac{2\left|\lambda_{0}\right|}{r_{min}}>0.
\]
From
the Gronwall's Lemma, we obtain
\[
\left|P_{+}^{n,j,\mu}(z)-P_{+}^{n,j,0}(z)\right|\leq C\varepsilon\exp\left(D\int_{-\infty}^{0}(1-U(\omega))d\omega\right).
\]
Since $\varepsilon$ is arbitrary, we see that $\left|P_{+}^{n,j,\mu}(z_{0})-P_{+}^{n,j,0}(z_{0})\right|\to0$
as $\mu\to0$ uniformly in $|n|< n_{0}$ and $j< j_{0}$. The
proof for $(P_{+}^{n,j,\mu})^{\prime}$ is similar and is thus omitted.
Therefore (\ref{eq:___Ptilde_muto0_z0}) holds. We are now ready to
prove (\ref{eq:PQtilde_muto0}) for indexes $|n|< n_{0}$ and $j< j_{0}$.
Let us recall that $P_{\pm}^{n,j,\mu}$ is extended to $\mathbb{R}_{-}$
by solving the Cauchy problem associated to (\ref{eq:___cauchypb})
with initial data taken at $z=z_{0}$. Because $z_{0}$ does not
depend on $\mu$, we deduce from classical results of ODEs and continuous
dependency of the solutions with respect to the parameter $\mu$,
that
\[
\sup_{|n|\leq n_{0},j\leq j_{0}}\left(\left|P_{\pm}^{n,j,\mu}(0)-P_{\pm}^{n,j,0}(0)\right|+\left|\left(P_{\pm}^{n,j,\mu}\right)^{\prime}(0)-\left(P_{\pm}^{n,j,0}\right)^{\prime}(0)\right|\right)\xrightarrow[\mu\to0]{}0.
\]
We now consider indexes $(n,j)$ such that $n_{0}\leq |n|\leq N_{0}$
or $j_{0}\leq j\leq J_{0}$, assuming $N_0\geq n_0$ or $J_0\geq j_0$. Let us recall that for such indexes, (\ref{eq:___rnj_bignj})
holds for all $0\leq\mu<1$, which means (\ref{eq:___P+}) and (\ref{eq:___P-})
hold for any $z\leq0$. Then using the same arguments, we have that
(\ref{eq:___Ptilde_muto0_z0}) holds where $(n_{0},j_{0},z_{0})$
are replaced by $(N_{0},J_{0},0)$, and thus $P_{\pm}^{n,j,\mu}$
satisfies (\ref{eq:PQtilde_muto0}).

We are now ready to prove (\ref{eq:Wmin}). We first consider fixed
indexes $|n|\leq n_{0}$ and $j\leq j_{0}$. Then one can readily
check that
\[
\sup_{|n|\leq n_{0},j\leq j_{0}}|a_{n,j,\mu}^{\pm}-a_{n,j,0}^{\pm}|\xrightarrow[\mu\to0]{}0.
\]
Since (\ref{eq:PQtilde_muto0}) holds with $(N_{0},J_{0})=(n_{0},j_{0})$,
we deduce from (\ref{eq:___Wphi}) that
\[
\sup_{|n|\leq n_{0},j\leq j_{0}}\left|W_{\tilde{\varphi}}^{n,j,\mu}-W_{\tilde{\varphi}}^{n,j,0}\right|\xrightarrow[\mu\to0]{}0,
\]
where we denoted $W_{\tilde{\varphi}}=W_{\tilde{\varphi}}^{n,j,\mu}$.
We have $W_{\tilde{\varphi}}^{n,j,0}\neq0$ for all $n,j$, since
it is the Wronskian of $(\varphi_{-},\varphi_{+})$ when $\mu=0$.
Therefore there exists $m>0$ such that $\inf_{|n|\leq n_{0},j\leq j_{0}}\left|W_{\tilde{\varphi}}^{n,j,0}\right|\geq m$.
Thus taking $\tilde{\mu}_{max}>0$ small enough, we obtain for any
$0\leq\mu<\tilde{\mu}_{max}$
\[
\inf_{|n|\leq n_{0},j\leq j_{0}}\left|W_{\tilde{\varphi}}^{n,j,\mu}\right|\geq\frac{m}{2}.
\]
Combining this with (\ref{eq:Wmin_bignj}), we obtain (\ref{eq:Wmin})
with $W_{min}:=\min\left(1,\frac{m}{2}\right)$.

Finally, let us prove (\ref{eq:int_phi_tilde}). Let us consider indexes
$(n,j)$ such that $|n|\geq n_{0}$ or $j\geq j_{0}$. Then $|P_{+}(0)|\geq\frac{1}{2}$
from (\ref{eq:PQtilde_0}). Set $\rho=(4\tilde{R}_{max})^{-1}>0$.
Then by the mean value inequality we have for all $\mu\in[0,1)$
$$
\int_{-\infty}^{0}\left|\varphi_{+}(z)\right|^{2}dz  \geq\int_{-\rho}^{0}\left|P_{+}(z)\right|^{2}e^{2\Re a^{+}z}dz \geq\left(\frac{1}{2}-\tilde{R}_{max}\rho\right)^{2}\int_{-\rho}^{0}e^{2\Re a^{+}z}dz\geq\frac{\rho}{16}e^{-2\Re a^{+}\rho}.
$$
Let us now assume that $|n|\leq n_{0}$
and $j\leq j_{0}$. Then from (\ref{eq:___P_bounds}) we have $|P_{+}(z_{0})|\geq\frac{1}{2}$,
with $z_{0}\leq0$ independent of $n,j,\mu$, see (\ref{eq:___z0_smallnj}).
Redoing the same calculations with $\int_{-\rho}^{0}$ being replaced
by $\int_{z_{0}-\rho}^{z_{0}}$, we obtain
\[
\int_{-\infty}^{0}\left|\varphi_{+}(z)\right|^{2}dz\geq\frac{\rho}{16}e^{2\Re a^{+}(z_{0}-\rho)}.
\]
Combining those estimates, we obtain (\ref{eq:int_phi_tilde}).

As for the construction of solutions $(\psi_{-},\psi_{+})$ on $\mathbb{R}_{+}$,
the procedure is similar and leads to the construction of a system
of fundamental solutions $(\psi_{-},\psi_{+})$ of (\ref{eq:homo})
on $\mathbb{R}_{+}$ such that (\ref{eq:homo_sol_R+})---(\ref{eq:Wmin_bignj})
hold. Let us however underline a key difference: there may happen
that $\Re b^{+}\leq0$ since (\ref{eq:sign_b+}) holds.
Despite that, $r\coloneqq 2b^{+}+2in\sigma+c_{0}=b^{+}-b^{-}$ still satisfies
$\Re r>0$ from (\ref{eq:Re_a-a_sign}). If however $(n,j,c_0)=(0,0,c^\ast)$, then $r=0$ and the above proof does not work, which is why we excluded the case $n=j=0$.
\end{proof}

We are now in the position to prove Lemma \ref{lem:homo_nj} concerning a fundamental system of solutions to (\ref{eq:homo}) on $\R$.

\begin{proof}[Proof of Lemma \ref{lem:homo_nj}]
In the context of this proof, we always assume $(n,j)\neq(0,0)$,
and for the sake of readability, we denote $a_{n,j,\mu}^{\pm}=a^{\pm}$,
$b_{n,j,\mu}^{\pm}=b^{\pm}$. From Lemma \ref{lem:phi_tilde}, for
all $n,j$ and $0\leq\mu<1$, we are equipped with the functions 
\[
\tilde{\varphi}_{\pm}(z)=\tilde{P}_{\pm}(z)e^{a^{\pm}z},\qquad\tilde{\psi}_{\pm}(z)=\tilde{Q}_{\pm}(z)e^{b^{\pm}z},
\]
that we extend to $\mathbb{R}$ by solving the Cauchy problem associated
to (\ref{eq:homo}). Let us fix $n,j$ and $0\leq\mu<\tilde{\mu}_{max}$,
where $\tilde{\mu}_{max}$ is obtained from Lemma \ref{lem:phi_tilde}.
Let us first assume that $(n,j,\mu)$ are such that $\tilde{\varphi}_{+},\tilde{\psi}_{-}$
are linearly independent. Then we set $(\varphi_{-},\varphi_{+}):=(\tilde{\psi}_{-},\tilde{\varphi}_{+})$,
which is indeed a fundamental system of solutions. In particular,
there exist $c_{-}\in\mathbb{C}$ and $c_{+}\in\mathbb{C}\backslash\{0\}$
such that for any $z\geq0$, there holds
\[
\varphi_{+}(z)=c_{-}\tilde{\psi}_{-}(z)+c_{+}\tilde{\psi}_{+}(z)=c_{-}\tilde{Q}_{-}(z)e^{b^{-}z}+c_{+}\tilde{Q}_{+}(z)e^{b^{+}z}.
\]
Setting 
\begin{equation}
P_{+}(z)=\tilde{P}_{+}(z),\qquad Q_{+}(z)=c_{-}\tilde{Q}_{-}(z)e^{(b^{-}-b^{+})z}+c_{+}\tilde{Q}_{+}(z),\label{eq:___P+Q+_expr}
\end{equation}
yields that $\varphi_{+}$ satisfies (\ref{eq:phi_pm}) with $P_{+}\in C_{b}^{2}(\mathbb{R}_{-})$
and $Q_{+}\in C_{b}^{2}(\mathbb{R}_{+})$, thanks to (\ref{eq:Re_a-a_sign}) and (\ref{eq:homo_sol_R-})---(\ref{eq:homo_sol_R+}). Also, since $c_{+}\neq0$, we have
$\liminf_{z\to+\infty}|Q_{+}(z)|\geq|c_{+}|\liminf_{z\to+\infty}|\tilde{Q}_{+}(z)|>0$
from (\ref{eq:homo_sol_R+}). We can prove in the same way that $\varphi_{-}$
satisfies (\ref{eq:phi_pm}) with $P_{-},Q_{-}$ belonging to $C_{b}^{2}(\mathbb{R}_{-}),C_{b}^{2}(\mathbb{R}_{+})$
respectively, and $\liminf_{z\to-\infty}|P_{-}(z)|>0$. 

Next, we claim that $\tilde{\varphi}_{+},\tilde{\psi}_{-}$ are never
collinear. Let us assume by contradiction that there exist $(n,j)\neq(0,0)$
and $0\leq\mu<\tilde{\mu}_{max}$ such that $\tilde{\varphi}_{+},\tilde{\psi}_{-}$
are collinear. Then adapting the above proof yields that $(\psi_{\infty},\psi_{0}):=(\tilde{\varphi}_{-},\tilde{\varphi}_{+})$
is a fundamental system of solutions of (\ref{eq:homo}) such that
\begin{equation}
\psi_{\infty}(z)=\begin{cases}
P_{-}(z)e^{a^{-}z} & z\leq0,\\
Q_{+}(z)e^{b^{+}z} & z\geq0,
\end{cases}\qquad\psi_{0}(z)=\begin{cases}
P_{+}(z)e^{a^{+}z} & z\leq0,\\
Q_{-}(z)e^{b^{-}z} & z\geq0,
\end{cases}\label{eq:___psi_0inf}
\end{equation}
with $P_{\pm}\in C_{b}^{2}(\mathbb{R}_{-})$, $Q_{\pm}\in C_{b}^{2}(\mathbb{R}_{+})$,
$\liminf_{z\to-\infty}|P_{-}(z)|>0$ and $\liminf_{z\to+\infty}|Q_{+}(z)|>0$.
Next, set $0<\delta<1$, and let $\mathcal{L}_{n,j,\mu}^{\delta,\rho}$
be the operator defined by (\ref{eq:Lnj_rho}), with $\rho = \frac{c_0}{2}$. We shall prove that
$\mathcal{L}_{n,j,\mu}^{\delta,\rho}$ is surjective. For any $f\in C_{\rho}^{0,\delta}(\mathbb{R},\mathbb{C})$,
using the variation of the constant, we set
\[
v(z):=\psi_{\infty}(z)\int_{-\infty}^{z}\frac{1}{W(\omega)}\psi_{0}(\omega)f(\omega)d\omega-\psi_{0}(z)\int_{0}^{z}\frac{1}{W(\omega)}\psi_{\infty}(\omega)f(\omega)d\omega,
\]
with the Wronskian $W(\omega):=\left[\psi_{\infty}^{\prime}\psi_{0}-\psi_{0}^{\prime}\psi_{\infty}\right](\omega)\neq0$.
By construction, $v\in C^{2}(\mathbb{R},\mathbb{C})$ and satisfies
$\mathcal{E}_{n,j,\mu}[v]=f$. Thus to conclude, it suffices to prove
that $v\in C_{\rho}^{2,\delta}(\mathbb{R},\mathbb{C})$, or equivalently that $\overline{v}(z):= v(z)e^{\rho z}\in C^{2,\delta}(\mathbb{R},\mathbb{C})$.
Firstly, we clearly have $\overline{v}\in C^{2}(\mathbb{R},\mathbb{C})$.
Also, notice that since $(\psi_{\infty},\psi_{0})$ solve (\ref{eq:homo}),
there holds
\[
W(\omega)=W(0)e^{-(c_{0}+2in\sigma)\omega}=W(0)e^{(a^{+}+a^{-})\omega}=W(0)e^{(b^{+}+b^{-})\omega}.
\]
Let us prove that $\overline{v}\in C_{b}^{2}(\mathbb{R},\mathbb{C})$.
Setting $\overline{f}:= fe^{\rho z}\in C^{0,\delta}(\mathbb{R},\mathbb{C})$,
for any $z\leq0$, there holds 
\begin{align*}
\left|\overline{v}(z)\right| & = \left|P_{-}(z)e^{(a^{-}+\rho)z}\int_{-\infty}^{z}\frac{P_{+}(\omega)}{W(0)}e^{-(a^{-}+\rho)\omega}\overline{f}(\omega)d\omega+P_{+}(z)e^{(a^{+}+\rho)z}\int_{z}^{0}\frac{P_{-}(\omega)}{W(0)}e^{-(a^{+}+\rho)\omega}\overline{f}(\omega)d\omega\right|\\
 & \leq\frac{1}{|W(0)|}||P_{-}||_{\infty}||P_{+}||_{\infty}||\overline{f}||_{\infty}\left(e^{(\text{Re }a^{-}+\rho)z}\int_{-\infty}^{z}e^{-(\text{Re }a^{-}+\rho)\omega}d\omega+e^{(\text{Re }a^{+}+\rho)z}\int_{z}^{0}e^{-(\text{Re }a^{+}+\rho)\omega}d\omega\right)\\
 & \leq\frac{1}{|W(0)|}||P_{-}||_{\infty}||P_{+}||_{\infty}||\overline{f}||_{\infty}\left(-\frac{1}{\text{Re }a^{-}+\rho}+\frac{1-e^{(\Re a^{+}+\rho)z}}{\text{Re }a^{+}+\rho}\right)\\
 & \leq\frac{1}{|W(0)|}||P_{-}||_{\infty}||P_{+}||_{\infty}||\overline{f}||_{\infty}\left(-\frac{1}{\text{Re }a^{-}+\rho}+\frac{1}{\text{Re }a^{+}+\rho}\right),
\end{align*}
where the last two lines of calculation are valid since $\Re a^{\pm}+\rho=\pm\frac{1}{2}\Re(a^{+}-a^{-})$ and (\ref{eq:Re_a-a_sign}) holds. Therefore $\overline{v}$ is uniformly bounded on $\mathbb{R}_{-}$,
and similarly on $\mathbb{R}_{+}$.
Next, because 
\[
\overline{v}^{\prime}(z)=\psi_{\infty}^{\prime}(z)\int_{-\infty}^{z}\frac{1}{W(\omega)}\psi_{0}(\omega)f(\omega)d\omega-\psi_{0}^{\prime}(z)\int_{0}^{z}\frac{1}{W(\omega)}\psi_{\infty}(\omega)f(\omega)d\omega,
\]
and $P_{\pm},Q_{\pm}\in C_{b}^{1}(\mathbb{R},\mathbb{C})$, we prove
in the same way that $\overline{v}\in C_{b}^{1}(\mathbb{R},\mathbb{C})$.
Finally, plugging $v=\overline{v}e^{-\rho z}$ into $\mathcal{E}_{n,j,\mu}[v]=f$,
we deduce that 
\[
\mathcal{\overline{E}}_{n,j,\mu}[\overline{v}]:=\overline{v}^{\prime\prime}+2in\sigma\overline{v}^{\prime}+\left[-\frac{c_{0}^{2}}{4}-\left(\lambda_{j}-(1+\delta_{0j})\lambda_{0}U(z)+(1+\mu)n^{2}\sigma^{2}\right)\right]\overline{v}=\overline{f}.
\]
Therefore $\overline{v}\in C_{b}^{2}(\mathbb{R},\mathbb{C})$, and
since $\overline{f}\in C^{0,\delta}(\mathbb{R},\mathbb{C})$, we deduce
immediately that $\overline{v}\in C^{2,\delta}(\mathbb{R},\mathbb{C})$,
so that $v\in C_{\rho}^{2,\delta}(\mathbb{R},\mathbb{C})$. Therefore
$\mathcal{L}_{n,j,\mu}^{\delta,\rho}$ is surjective. From Lemma \ref{lem:Fredholm_weight}
in Appendix \ref{sec:Fredholm}, this operator is also Fredholm of
index zero, and is thus injective. We shall obtain a contradiction
by showing that $\psi_{0}\in\ker\mathcal{L}_{n,j,\mu}^{\delta,\rho}$.
First, setting $\overline{\psi}_{0}(z):=\psi_{0}(z)e^{\rho z}$,
we have $\overline{\psi}_{0}\in C_{b}^{2}(\mathbb{R},\mathbb{C})$
from (\ref{eq:___psi_0inf}).
Then, since $\mathcal{\overline{E}}_{n,j,\mu}[\overline{\psi}_{0}]=0$,
we also have $\overline{\psi}_{0}\in C^{2,\delta}(\mathbb{R},\mathbb{C})$.
Therefore $\psi_{0}\in C_{\rho}^{2,\delta}(\mathbb{R},\mathbb{C})$
satisfies $\mathcal{E}_{n,j,\mu}[\psi_{0}]=0$, which means $\psi_{0}\in\ker\mathcal{L}_{n,j,\mu}^{\delta,\rho}$.
Consequently our assumption that $\tilde{\varphi}_{+},\tilde{\psi}_{-}$
are collinear is absurd. Therefore for all $(n,j)\neq(0,0)$ and $0\leq \mu<\tilde{\mu}_{max}$, the functions $(\varphi_{-},\varphi_{+})$
defined by (\ref{eq:phi_pm}) form a fundamental system of solutions
of (\ref{eq:homo}).

Let us prove that $P_{\pm},Q_{\pm}$ satisfy (\ref{eq:PQ_bounds}).
We shall only prove it for $P_{+},Q_{+}$, the equivalent for $P_{-},Q_{-}$
being similar. Since $P_{+}=\tilde{P}_{+}$ satisfies (\ref{eq:PQtilde_bounds}),
$P_{+}$ satisfies (\ref{eq:PQ_bounds}) if $R_{max}\geq\tilde{R}_{max}$.
As for $Q_{+}$, we first need upper bounds on $c_{\pm}$. Since $\varphi_{+},\varphi_{+}^{\prime}$
are continuous at $z=0$, we obtain from (\ref{eq:phi_pm}) and (\ref{eq:___P+Q+_expr}),
the following linear system in $(c_{-},c_{+})$
\begin{equation}
\begin{cases}
c_{-}\tilde{Q}_{-}(0)+c_{+}\tilde{Q}_{+}(0) & =\tilde{P}_{+}(0),\\
c_{-}\left(\tilde{Q}_{-}^{\prime}(0)+b^{-}\tilde{Q}_{-}(0)\right)+c_{+}\left(\tilde{Q}_{+}^{\prime}(0)+b^{+}\tilde{Q}_{+}(0)\right) & =\tilde{P}_{+}^{\prime}(0)+a^{+}\tilde{P}_{+}(0).
\end{cases}\label{eq:___lin_syst}
\end{equation}
Let us recall that we denoted $W_{\tilde{\psi}}$ the Wronskian in
$z=0$ of the family $(\tilde{\psi}_{-},\tilde{\psi}_{+})$, and there
holds (\ref{eq:Wmin}) since $\mu<\tilde{\mu}_{max}$. Therefore the
system (\ref{eq:___lin_syst}) admits a unique solution for all $n,j$
and $0\leq\mu<\tilde{\mu}_{max}$, given by
\begin{equation}
c_{-}=-\frac{\left(\tilde{Q}_{+}^{\prime}(0)+b^{+}\tilde{Q}_{+}(0)\right)\tilde{P}_{+}(0)-\left(\tilde{P}_{+}^{\prime}(0)+a^{+}\tilde{P}_{+}(0)\right)\tilde{Q}_{+}(0)}{W_{\tilde{\psi}}},\label{eq:___c-}
\end{equation}

\begin{equation}
c_{+}=-\frac{\left(\tilde{P}_{+}^{\prime}(0)+a^{+}\tilde{P}_{+}(0)\right)\tilde{Q}_{-}(0)-\left(\tilde{Q}_{-}^{\prime}(0)+b^{-}\tilde{Q}_{-}(0)\right)\tilde{P}_{+}(0)}{W_{\tilde{\psi}}}.\label{eq:___c+}
\end{equation}
Let $n_{0},j_{0}>0$ being given by Lemma \ref{lem:phi_tilde} and
fix $n,j$ such that $|n|\geq n_{0}$ or $j\geq j_{0}$. Then we deduce
from (\ref{eq:a-b_bound}), (\ref{eq:PQtilde_0})---(\ref{eq:PpQptilde0})
and (\ref{eq:Wmin_bignj}) that
\begin{align*}
|c_{-}| & \leq\frac{\frac{3}{2}\left|b^{+}-a^{+}\right|+1}{\frac{1}{4}|b^{+}-b^{-}|+1}\leq\frac{6\overline{C}+4}{|b^{+}-b^{-}|+4},
\end{align*}
\[
|c_{+}|\leq\frac{\frac{9}{4}\left|a^{+}-b^{-}\right|+3}{\frac{1}{4}|b^{+}-b^{-}|+1}\leq\frac{9\left(|b^{+}-b^{-}|+\overline{C}\right)+12}{|b^{+}-b^{-}|+4}.
\]
Thus there exists $C>0$ such that for all $n,j$ and $0\leq\mu<\tilde{\mu}_{max}$,
\begin{equation}
|c_{-}|\leq\frac{C}{1+|b^{+}-b^{-}|},\qquad|c_{+}|\leq C,\qquad\text{if }|n|\geq n_{0}\text{ or }j\geq j_{0}.\label{eq:___cpm_bounds}
\end{equation}
This leads to, thanks to (\ref{eq:PQtilde_bounds}) and (\ref{eq:___P+Q+_expr}),
\begin{equation}
||Q_{+}||_{L^{\infty}(\mathbb{R}_{+})}\leq2C\tilde{R}_{max},\qquad\text{if }|n|\geq n_{0}\text{ or }j\geq j_{0},\label{eq:___Q+_bignj}
\end{equation}
\begin{equation}
||Q_{+}^{\prime}||_{L^{\infty}(\mathbb{R}_{+})}\leq\frac{C\left(\tilde{R}_{max}+|b^{+}-b^{-}|\tilde{R}_{max}\right)}{1+|b^{+}-b^{-}|}+C\tilde{R}_{max}\leq2C\tilde{R}_{max},\qquad\text{if }|n|\geq n_{0}\text{ or }j\geq j_{0},
\end{equation}
so that $Q_{+}$ satisfies (\ref{eq:PQ_bounds}) for any such $n,j$.
We now consider fixed indexes $|n|\leq n_{0}$ and $j\leq j_{0}$.
It is clear that there exists $M>0$ such that 
\[
\max\left(|b_{n,j,\mu}^{\pm}|,|a_{n,j,\mu}^{\pm}|\right)\leq M,\qquad\forall|n|\leq n_{0},\quad\forall j\leq j_{0},\quad\forall0\leq\mu<\tilde{\mu}_{max}.
\]
Then combining this with (\ref{eq:PQtilde_bounds}) and (\ref{eq:Wmin}),
we have
\[
|c_{\pm}|\leq\frac{2\left(M+1\right)\tilde{R}_{max}^{2}}{W_{min}}\eqqcolon K,
\]
so that we deduce 
\begin{equation}
||Q_{+}||_{L^{\infty}(\mathbb{R}_{+})}\leq2K\tilde{R}_{max},\qquad\forall|n|\leq n_{0},\,\forall j\leq j_{0}.
\end{equation}
\begin{equation}
||Q_{+}^{\prime}||_{L^{\infty}(\mathbb{R}_{+})}\leq K\left(\tilde{R}_{max}+|b^{+}-b^{-}|\tilde{R}_{max}\right)+K\tilde{R}_{max}\leq2K\left(1+M\right)\tilde{R}_{max},\qquad\forall|n|\leq n_{0},\,\forall j\leq j_{0}.\label{eq:___Q+p_smallnj}
\end{equation}
Consequently, from (\ref{eq:___Q+_bignj})---(\ref{eq:___Q+p_smallnj}),
we obtain that $Q_{+}$ satisfies (\ref{eq:PQ_bounds}) for all $n,j$
and $0\leq\mu<\tilde{\mu}_{max}$.

Let us now prove that $W_{\varphi}$ satisfies (\ref{eq:W0_bignj}).
Let us consider indexes $n,j$ such that $|n|\geq N_{0}\geq n_{0}$
\textit{or} $j\geq J_{0}\geq j_{0}$, where $N_{0},J_{0}>0$ are large
enough so that, thanks to (\ref{eq:a-a_bound}), we have
\begin{equation}
\left|b_{n,j,\mu}^{+}-b_{n,j,\mu}^{-}\right|\geq\max\left(2\overline{C}+24\,,\,702\left(2R_{max}^{2}+1\right)\right),\quad\forall\mu\in[0,1),\label{eq:___N0J0_1}
\end{equation}
\begin{equation}
\frac{\frac{1}{12}\left|b_{n,j,\mu}^{+}-b_{n,j,\mu}^{-}\right|-\frac{3}{2}C}{\left|b_{n,j,\mu}^{+}-b_{n,j,\mu}^{-}\right|+1}\geq\frac{1}{13},\quad\forall\mu\in[0,1).\label{eq:___N0J0_2}
\end{equation}
Since $N_{0}\geq n_{0}$ and $J_{0}\geq j_{0}$, we also have thanks
to (\ref{eq:PQtilde_0})---(\ref{eq:PpQptilde0})
\[
\left|W_{\tilde{\psi}}\right|\leq\frac{3}{2}|b^{+}-b^{-}|+1\leq\frac{3}{2}\left(1+|b^{+}-b^{-}|\right),
\]
and thus, using (\ref{eq:a-b_bound}), (\ref{eq:PQtilde_0})---(\ref{eq:PpQptilde0})
and (\ref{eq:___c+}), we deduce
\[
|c_{+}|\geq\frac{\frac{1}{4}|b^{+}-b^{-}|-\frac{1}{4}\overline{C}-3}{\frac{3}{2}\left(1+|b^{+}-b^{-}|\right)}\geq\frac{\frac{1}{12}|b^{+}-b^{-}|}{1+|b^{+}-b^{-}|},
\]
where we used (\ref{eq:___N0J0_1}) for the last inequality. Combining
this lower bound with (\ref{eq:PQtilde_0}), (\ref{eq:___P+Q+_expr})
and (\ref{eq:___cpm_bounds}), we have
\begin{align*}
\left|Q_{+}(0)\right| & \geq|c_{+}|-\frac{3}{2}|c_{-}|\geq\frac{1}{13},
\end{align*}
where we used (\ref{eq:___N0J0_2}) for the last inequality. Let us
also recall that, by construction, we have $|Q_{-}(0)|=|\tilde{Q}_{-}(0)|\geq\frac{1}{2}$
from (\ref{eq:PQtilde_0}) and because $|n|\geq N_{0}\geq n_{0}$.
Combining that with (\ref{eq:PQ_bounds}), we obtain 
\begin{align}
W_{\varphi} & := \left[(b^{-}-b^{+})Q_{-}Q_{+}+Q_{-}^{\prime}Q_{+}-Q_{+}^{\prime}Q_{-}\right](0) \label{eq:Wphi_expr}\\
|W_{\varphi}| & \geq\frac{1}{26}\left|b^{+}-b^{-}\right|-2R_{max}^{2} \nonumber \\
 & \ge\frac{1}{27}\left|b^{+}-b^{-}\right|+1,\nonumber
\end{align}
since (\ref{eq:___N0J0_1}) holds. Thanks to (\ref{eq:a-a_bound}),
one can readily show that $W_{\varphi}$ satisfies (\ref{eq:W0_bignj})
with $C_{W}=\min\left(1,\frac{1}{27}\underline{C}\right)^{-1}$.

Let us prove that $W_{\varphi}=W_{\varphi}^{n,j,\mu}$ satisfies (\ref{eq:W0})
for indexes $|n|\leq N_{0}$ and $j\leq J_{0}$. From (\ref{eq:Wphi_expr}), since $Q_{-}=\tilde{Q}_{-}$
by construction, since $Q_{+}$ satisfies (\ref{eq:___P+Q+_expr})
with $c_{\pm}$ given by (\ref{eq:___c-})---(\ref{eq:___c+}),
and because $W_{\tilde{\psi}}$ is given by (\ref{eq:Wmin}), there
exists a polynomial function $\mathcal{P}$ such that for all
$n,j,\mu$, we have
\[
W_{\varphi}^{n,j,\mu} = \frac{1}{W_{\tilde{\psi}}^{n,j,\mu}} \,\mathcal{P}\left(\tilde{Q}_{-}(0),\tilde{Q}_{-}^{\prime}(0),\tilde{Q}_{+}(0),\tilde{Q}_{+}^{\prime}(0),\tilde{P}_{+}(0),\tilde{P}_{+}^{\prime}(0),a^{+},b^{+},b^{-}\right)\eqqcolon \frac{1}{W_{\tilde{\psi}}^{n,j,\mu}} \,\mathcal{P}^{n,j,\mu}.
\]
One can readily check that $\left|a_{n,j,\mu}^{\pm}-a_{n,j,0}^{\pm}\right|,\left|b_{n,j,\mu}^{\pm}-b_{n,j,0}^{\pm}\right|\to0$
as $\mu\to0$ uniformly in $|n|\leq N_{0}$ and $j\leq J_{0}$. Combining
this with (\ref{eq:PQtilde_muto0}), we obtain 
\[
\sup_{|n|\leq N_{0},j\leq J_{0}}\left|\mathcal{P}^{n,j,\mu}-\mathcal{P}^{n,j,0}\right|\xrightarrow[\mu\to0]{}0.
\]
Meanwhile, the same holds for $W_{\tilde{\psi}}^{n,j,\mu}$, as we showed in the proof of (\ref{eq:Wmin}). Since (\ref{eq:Wmin}) holds, we have
\[
\sup_{|n|\leq N_{0},j\leq J_{0}}\left|W_{\varphi}^{n,j,\mu}-W_{\varphi}^{n,j,0}\right|\xrightarrow[\mu\to0]{}0.
\]
Also, there holds $W_{\varphi}^{n,j,0}\neq0$ for all $n,j$
since $\varphi_{-},\varphi_{+}$ are linearly independent. Thus there
exists $m>0$ such that 
\[
\inf_{|n|\leq N_{0},j\leq J_{0}}\left|W_{\varphi}^{n,j,0}\right|\geq m.
\]
Therefore, and because $N_{0},J_{0}$ are uniform in $\mu\in[0,1)$
from (\ref{eq:___N0J0_1})---(\ref{eq:___N0J0_2}), we deduce the
existence of $0<\mu_{max}\leq\tilde{\mu}_{max}$ such that for all
$0\leq\mu<\mu_{max}$ we have
\[
\inf_{|n|\leq N_{0},j\leq J_{0}}\left|W_{\varphi}^{n,j,\mu}\right|\ge\frac{m}{2}.
\]
Combining this with (\ref{eq:W0_bignj}), we see that (\ref{eq:W0})
holds with $W_{0}=\min(\frac{1}{C_{W}},\frac{m}{2})$.

Finally, (\ref{eq:int_phi}) simply follows from (\ref{eq:int_phi_tilde})
since $\varphi_{+}=\tilde{\varphi}_{+}$ by construction. The proof
of Lemma \ref{lem:homo_nj} is thus complete.
\end{proof}

\subsection{Fredholm analysis\label{sec:Fredholm}}

We recall below \cite[Theorem 2.4, p. 366]{Vol_11}.
\begin{thm}[Fredholm property in $C^{2,\delta}(\mathbb{R},\mathbb{R}^{d})$]
\label{thm:Fredholm}Fix $0<\delta<1$ and $d\in\mathbb{N}_+$.
Consider the operator $L\colon C^{2,\delta}(\mathbb{R},\mathbb{R}^{d})\to C^{0,\delta}(\mathbb{R},\mathbb{R}^{d})$
defined by
\[
Lu:=\alpha(x)u^{\prime\prime}+\beta(x)u^{\prime}+\gamma(x)u,
\]
where the coefficients $\alpha(x),\beta(x),\gamma(x)$ are smooth
$d\times d$ matrices and there is $\alpha_{0}>0$ such that $\langle\alpha(x)\xi,\xi\rangle\geq\alpha_{0}|\xi|^{2}$
for any $\xi\in\mathbb{R}^{d}$. Assume that $\alpha(x),\beta(x),\gamma(x)$
have finite limits as $x\to\pm\infty$, denoted respectively $\alpha_{\pm},\beta_{\pm},\gamma_{\pm}$.
Finally, we define the limiting operators
\[
L^{\pm}u=\alpha_{\pm}u^{\prime\prime}+\beta_{\pm}u^{\prime}+\gamma_{\pm}u,
\]
and assume that 
\[
\forall\xi\in\mathbb{R},\quad T^{\pm}(\xi)=-\alpha_{\pm}\xi^{2}+\beta_{\pm}i\xi+\gamma_{\pm}\quad\text{is an invertible matrix}.
\]
Then $L$ is a Fredholm operator, and its index is given by $\ind L=k_{+}-k_{-}$,
where 
\[
k_{\pm}=\text{Sp}(M^{\pm})\cap\left\{ z\in\mathbb{C}: \Re z>0\right\} ,\quad M^{\pm}=\left(\begin{array}{cc}
0 & -I_{d}\\
\alpha_{\pm}^{-1}\gamma_{\pm} & \alpha_{\pm}^{-1}\beta_{\pm}
\end{array}\right)\in M_{2d\times2d}(\mathbb{R}).
\]
\end{thm}

We now apply Theorem \ref{thm:Fredholm} to our case.
\begin{lem}[Fredholm property on weighted spaces]
\label{lem:Fredholm_weight}Let $(n,j)\in\mathbb{Z}\times\mathbb{N}$
with $(n,j)\neq(0,0)$, and $0\leq\mu<1$. Set $0<\delta<1$, $\rho:=\frac{c_{0}}{2}>0$
and
\begin{alignat}{2}
\mathcal{L}_{n,j,\mu}^{\delta,\rho}\colon & C_{\rho}^{2,\delta}(\mathbb{R},\mathbb{C}) &  & \to C_{\rho}^{0,\delta}(\mathbb{R},\mathbb{C})\label{eq:Lnj_rho}\\
 & u &  & \mapsto\mathcal{E}_{n,j,\mu}[u],\nonumber 
\end{alignat}
where $\mathcal{E}_{n,j,\mu}[u]$ is given by (\ref{eq:vnj}), and
for any $k\in\mathbb{N}$, we set
\[
C_{\rho}^{k,\delta}(\mathbb{R},\mathbb{C}):=\left\{ f\in C^{k}(\mathbb{R},\mathbb{C}): ||f||_{C_{\rho}^{k,\delta}(\mathbb{R},\mathbb{C})}<\infty\right\} ,\qquad||f||_{C_{\rho}^{k,\delta}(\mathbb{R},\mathbb{C})}:=\left\Vert z\mapsto f(z)e^{\rho z}\right\Vert _{C^{k,\delta}(\mathbb{R},\mathbb{C})}.
\]
Then $\mathcal{L}_{n,j,\mu}^{\delta,\rho}$ is Fredholm with index
zero.
\end{lem}

\begin{proof}
We cannot apply Theorem \ref{thm:Fredholm} to $\mathcal{L}_{n,j,\mu}^{\delta,\rho}$
because of the weighted space and the functions involved being complex.
Thus we define successively the real counterpart of $\mathcal{L}_{n,j,\mu}^{\delta,\rho}$
by
\begin{alignat*}{2}
L_{n,j,\mu}^{\delta,\rho} \colon & C_{\rho}^{2,\delta}(\mathbb{R},\mathbb{R}^{2}) &  & \to C_{\rho}^{0,\delta}(\mathbb{R},\mathbb{R}^{2})\\
 & u=(u_{1},u_{2}) &  & \mapsto\left(\begin{array}{c}
u_{1}^{\prime\prime}+c_{0}u_{1}^{\prime}-2n\sigma u_{2}^{\prime}-\left(\lambda_{j}-(1+\delta_{0j})\lambda_{0}U(z)+(1+\mu)n^{2}\sigma^{2}\right)u_{1}\vspace{0.1cm}\\
u_{2}^{\prime\prime}+c_{0}u_{2}^{\prime}+2n\sigma u_{1}^{\prime}-\left(\lambda_{j}-(1+\delta_{0j})\lambda_{0}U(z)+(1+\mu)n^{2}\sigma^{2}\right)u_{2}
\end{array}\right),
\end{alignat*}
as well as the operator $M_{n,j,\mu}^{\delta,\rho}\colon C^{2,\delta}(\mathbb{R},\mathbb{R}^{2})\to C^{0,\delta}(\mathbb{R},\mathbb{R}^{2})$
with
\begin{align*}
M_{n,j,\mu}^{\delta,\rho}u & =e^{\rho z}L_{n,j,\mu}^{\delta,\rho}\left(ue^{-\rho z}\right)\\
 & =\left(\begin{array}{c}
u_{1}^{\prime\prime}+\left[\rho^{2}-\rho c_{0}-\left(\lambda_{j}-(1+\delta_{0j})\lambda_{0}U(z)+(1+\mu)n^{2}\sigma^{2}\right)\right]u_{1}-2n\sigma\left(u_{2}^{\prime}-\rho u_{2}\right)\vspace{0.1cm}\\
u_{2}^{\prime\prime}+\left[\rho^{2}-\rho c_{0}-\left(\lambda_{j}-(1+\delta_{0j})\lambda_{0}U(z)+(1+\mu)n^{2}\sigma^{2}\right)\right]u_{2}+2n\sigma\left(u_{1}^{\prime}-\rho u_{1}\right)
\end{array}\right).
\end{align*}
We may rewrite the above operator as
\[
M_{n,j,\mu}^{\delta,\rho}u=\left(\begin{array}{cc}
1 & 0\\
0 & 1
\end{array}\right)u^{\prime\prime}+\left(\begin{array}{cc}
0 & -2n\sigma\\
2n\sigma & 0
\end{array}\right)u^{\prime}+\left(\begin{array}{cc}
q(z) & 2n\sigma\rho\\
-2n\sigma\rho & q(z)
\end{array}\right)u\eqqcolon I_{2}u^{\prime\prime}+\beta u^{\prime}+\gamma(z)u,
\]
with, given that $\rho=\frac{c_{0}}{2}$,
\[
q(z):=-\frac{c_{0}^{2}}{4}-\left(\lambda_{j}-(1+\delta_{0j})\lambda_{0}U(z)+(1+\mu)n^{2}\sigma^{2}\right).
\]
We also set 
\[
\gamma_{\pm}=\lim_{z\to\pm\infty}\gamma(z)=\left(\begin{array}{cc}
q_{\pm} & n\sigma c_{0}\\
-n\sigma c_{0} & q_{\pm}
\end{array}\right),
\]
\[
\mathbb{R}\ni q_{\pm}=\lim_{z\to\pm\infty}q(z)=-\frac{c_{0}^{2}}{4}-\left(\lambda_{j}+(1+\mu)n^{2}\sigma^{2}\right)+\begin{cases}
(1+\delta_{0j})\lambda_{0}, & \text{ if } q_{\pm}=q_{-},\\
0, & \text{ if  } q_{\pm}=q_{+}.
\end{cases}
\]
Finally, we define
\[
T_{\rho}^{\pm}\colon\xi\in\mathbb{R}\mapsto-\xi^{2}I+i\xi\beta+\gamma_{\pm}\in M_{2\times2}(\mathbb{R}).
\]
In order to apply Theorem \ref{thm:Fredholm}, we have to check if
$T_{\rho}^{\pm}(\xi)$ is invertible for any $\xi\in\mathbb{R}$.
We compute
\begin{align*}
\det T_{\rho}^{\pm}(\xi) & =\left|\begin{array}{cc}
-\xi^{2}+q_{\pm} & -2in\sigma\xi+2n\sigma\rho\\
2in\sigma\xi-2n\sigma\rho & -\xi^{2}+q_{\pm}
\end{array}\right|\\
 & =\left(-\xi^{2}+q_{\pm}\right)^{2}+\left(2in\sigma\xi-2n\sigma\rho\right)^{2}\\
 & =P_{\mathbb{R}}(\xi)-8in^{2}\sigma^{2}\rho\xi,
\end{align*}
with $P_{\mathbb{R}}(\xi)$ a real polynomial function of $\xi$.
Assume by contradiction that there exists $\xi\in\mathbb{R}$ such
that $\det T_{\rho}^{\pm}(\xi)=0$. Then necessarily $n\xi=0$. Then
$n$ has to be zero, for otherwise we would have $\xi=0$, while
\[
\det T_{\rho}^{\pm}(0)=q_{\pm}^{2}+(2n\sigma\rho)^{2}>0.
\]
Therefore $n=0$, but this would yield
\[
\det T_{\rho}^{\pm}(\xi)=\left(-\xi^{2}+q_{\pm}\right)^{2}.
\]
However, since $(n,j)\neq(0,0)$
\[
q_{+}=-\frac{c_{0}^{2}}{4}-\lambda_{j}<-\frac{1}{4}\left(c_{0}^{2}+4\lambda_{0}\right)\leq0,\qquad q_{-}=q_{+}+\lambda_{0}<q_{+}\leq0,
\]
which means that if $n=0$, we again have $\det T_{\rho}^{\pm}(\xi)\neq0$
for any $\xi\in\mathbb{R}$, which contradicts our assumption. Consequently,
$T_{\rho}^{\pm}(\xi)$ is invertible for any $\xi\in\mathbb{R}$. 

Hence, from Theorem \ref{thm:Fredholm}, we deduce that $L_{n,j,\mu}^{\delta}$
is Fredholm with
\[
\ind M_{n,j,\mu}^{\delta,\rho}=k^{+}-k^{-},\qquad k^{\pm}=\text{Sp}M^{\pm}\cap\left\{ z: \Re z>0\right\} ,
\]
where 
\begin{align*}
M^{\pm} & :=\left(\begin{array}{cc}
0 & -I_{2}\\
\gamma_{\pm} & \beta
\end{array}\right)=\left(\begin{array}{cccc}
0 & 0 & -1 & 0\\
0 & 0 & 0 & -1\\
q_{\pm} & n\sigma c_{0} & 0 & -2n\sigma\\
-n\sigma c_{0} & q_{\pm} & 2n\sigma & 0
\end{array}\right).
\end{align*}
Let us first determine $k^{+}$. We compute
\begin{align*}
\det(M^{+}-XI_{4}) & =\left|\begin{array}{cc}
q_{+}+X^{2} & n\sigma c_{0}+2n\sigma X\\
-n\sigma c_{0}-2n\sigma X & q_{+}+X^{2}
\end{array}\right|\\
 & =(q_{+}+X^{2})^{2}+\left(n\sigma c_{0}+2n\sigma X\right)^{2}\\
 & =\left(X^{2}+q_{+}-in\sigma c_{0}-2in\sigma X\right)\left(X^{2}+q_{+}+in\sigma c_{0}+2in\sigma X\right),
\end{align*}
from which we deduce that the eigenvalues of $M^{+}$ are
\[
X_{\pm}^{1}=\frac{2in\sigma\pm\sqrt{(2in\sigma)^{2}-4(q_{+}-in\sigma c_{0})}}{2},\quad 
X_{\pm}^{2}=\frac{-2in\sigma\pm\sqrt{(2in\sigma)^{2}-4(q_{+}+in\sigma c_{0})}}{2},
\]
and thus
\begin{align*}
\Re\, (2X_{\pm}^{1}) & =\pm\Re\sqrt{-4n^{2}\sigma^{2}-4q_{+}+4in\sigma c_{0}}\\
 & =\pm\Re\sqrt{c_{0}^{2}+4\lambda_{j}+4in\sigma c_{0}+4\mu n^{2}\sigma^{2}}\\
 & =\pm\Re\, (b_{n,j,\mu}^{+}-b_{n,j,\mu}^{-}).
\end{align*}
Notice that, since $(n,j)\neq(0,0)$, we have $\Re\,(b_{n,j,\mu}^{+}-b_{n,j,\mu}^{-})>0$,
so that $\Re X_{-}^{1}<0<\Re X_{+}^{1}$. Similar calculations yield
\[
\Re\, (2X_{\pm}^{2})=\pm\Re\sqrt{c_{0}^{2}+4\lambda_{j}-4in\sigma c_{0}+4\mu n^{2}\sigma^{2}}=\pm\Re\,(b_{-n,j,\mu}^{+}-b_{-n,j,\mu}^{-}),
\]
therefore we also have $\Re X_{-}^{2}<0<\Re X_{+}^{2}$. As a result
$k^{+}=2$.

Let us now turn our attention to $k^{-}$. Similar calculations yield
that the eigenvalues of $M^{-}$ are
\[
Y_{\pm}^{1}=\frac{2in\sigma\pm\sqrt{(2in\sigma)^{2}-4(q_{-}-in\sigma c_{0})}}{2},\quad
Y_{\pm}^{2}=\frac{-2in\sigma\pm\sqrt{(2in\sigma)^{2}-4(q_{-}+in\sigma c_{0})}}{2},
\]
which yields 
\begin{align*}
\Re\,(2Y_{\pm}^{1}) & =\pm\Re\sqrt{-4n^{2}\sigma^{2}+4in\sigma c_{0}-4q_{-}}\\
 & =\pm\Re\sqrt{c_{0}^{2}+4in\sigma c_{0}+4\lambda_{j}+4\mu n^{2}\sigma^{2}-4(1+\delta_{0j})\lambda_{0}}\\
 & =\pm\Re\, (a_{n,j,\mu}^{+}-a_{n,j,\mu}^{-}).
\end{align*}
Notice that $\Re\,(a_{n,j,\mu}^{+}-a_{n,j,\mu}^{-})>0$, so that $\Re Y_{-}^{1}<0<\Re Y_{+}^{1}$,
and similarly $\Re Y_{-}^{2}<0<\Re Y_{+}^{2}$. Hence, $k^{-}=2$,
which means $\ind M_{n,j,\mu}^{\delta,\rho}=0$. Since the operator $S_{\rho}\colon u\in C^{2,\delta}(\mathbb{R},\mathbb{R}^{2})\mapsto ue^{-\rho z}\in C_{\rho}^{2,\delta}(\mathbb{R},\mathbb{R}^{2})$
is continuously invertible with $S_{\rho}^{-1}\colon u\in C_{\rho}^{2,\delta}(\mathbb{R},\mathbb{R}^{2})\mapsto ue^{\rho z}\in C^{2,\delta}(\mathbb{R},\mathbb{R}^{2})$,
we have that $L_{n,j,\mu}^{\delta,\rho}u=S_{\rho}M_{n,j,\mu}^{\delta,\rho}S_{\rho}^{-1}u$
shares the same Fredholm property and index as $M_{n,j,\mu}^{\delta,\rho}$.
From there, we prove similarly that the operator $\mathcal{L}_{n,j,\mu}^{\delta,\rho}$
defined by (\ref{eq:Lnj_rho}) is also Fredholm and satisfies $\ind\mathcal{L}_{n,j,\mu}^{\delta,\rho}=0$.
The proof is thus complete.
\end{proof}

\medskip
 \noindent{\bf Acknowledgements.} 
 Both authors are supported by the 
ANR \textsc{i-site muse}, project \textsc{michel} 170544IA (n$^{\circ}$ ANR \textsc{idex}-0006).

\bibliographystyle{plain}
\bibliography{./biblio}

\begin{thebibliography}{10}

\bibitem{AlfBerRao_17}
Matthieu Alfaro, Henri Berestycki, and Ga{\"e}l Raoul.
\newblock The effect of climate shift on a species submitted to dispersion,
  evolution, growth, and nonlocal competition.
\newblock {\em SIAM Journal on Mathematical Analysis}, 49(1):562--596, 2017.

\bibitem{AlfCov_12}
Matthieu Alfaro and J{\'e}r{\^o}me Coville.
\newblock Rapid traveling waves in the nonlocal {F}isher equation connect two
  unstable states.
\newblock {\em Appl. Math. Lett.}, 25(12):2095--2099, 2012.

\bibitem{AlfCovRao_13}
Matthieu Alfaro, J{\'e}r{\^o}me Coville, and Ga{\"e}l Raoul.
\newblock Travelling waves in a nonlocal reaction-diffusion equation as a model
  for a population structured by a space variable and a phenotypic trait.
\newblock {\em Communications in Partial Differential Equations},
  38(12):2126--2154, 2013.

\bibitem{AlfGri_18}
Matthieu Alfaro and Quentin Griette.
\newblock Pulsating fronts for fisher--kpp systems with mutations as models in
  evolutionary epidemiology.
\newblock {\em Nonlinear Analysis: Real World Applications}, 42:255--289, 2018.

\bibitem{AlfVer_19}
Matthieu Alfaro and Mario Veruete.
\newblock Evolutionary branching via replicator--mutator equations.
\newblock {\em Journal of Dynamics and Differential Equations},
  31(4):2029--2052, 2019.

\bibitem{AroWei_75}
D.~G. Aronson and H.~F. Weinberger.
\newblock Nonlinear diffusion in population genetics, combustion, and nerve
  pulse propagation.
\newblock In {\em Partial differential equations and related topics ({P}rogram,
  {T}ulane {U}niv., {N}ew {O}rleans, {L}a., 1974)}, pages 5--49. Lecture Notes
  in Math., Vol. 446. Springer, Berlin, 1975.

\bibitem{AroWei_78}
D.~G. Aronson and H.~F. Weinberger.
\newblock Multidimensional nonlinear diffusion arising in population genetics.
\newblock {\em Adv. in Math.}, 30(1):33--76, 1978.

\bibitem{BagMar_09}
Micha\"{e}l Bages and Patrick Martinez.
\newblock Existence of pulsating waves of advection-reaction-diffusion
  equations of ignition type by a new method.
\newblock {\em Nonlinear Anal.}, 71(12):e1880--e1903, 2009.

\bibitem{Bay_16}
Michael Baym, Tami~D Lieberman, Eric~D Kelsic, Remy Chait, Rotem Gross, Idan
  Yelin, and Roy Kishony.
\newblock Spatiotemporal microbial evolution on antibiotic landscapes.
\newblock {\em Science}, 353(6304):1147--1151, 2016.

\bibitem{BenCalMeuVoi_12}
Olivier Benichou, Vincent Calvez, Nicolas Meunier, and Raphael Voituriez.
\newblock Front acceleration by dynamic selection in {F}isher population waves.
\newblock {\em Physical Review E}, 86(4):041908, 2012.

\bibitem{BerHam_02}
Henri Berestycki and Fran{\c{c}}ois Hamel.
\newblock Front propagation in periodic excitable media.
\newblock {\em Comm. Pure Appl. Math.}, 55(8):949--1032, 2002.

\bibitem{BerHamRoq_05_2}
Henri Berestycki, Fran{\c{c}}ois Hamel, and Lionel Roques.
\newblock Analysis of the periodically fragmented environment model. {II}.
  {B}iological invasions and pulsating travelling fronts.
\newblock {\em J. Math. Pures Appl. (9)}, 84(8):1101--1146, 2005.

\bibitem{BerJinSil_16}
Henri Berestycki, Tianling Jin, and Luis Silvestre.
\newblock Propagation in a non local reaction diffusion equation with spatial
  and genetic trait structure.
\newblock {\em Nonlinearity}, 29(4):1434--1466, 2016.

\bibitem{BerNadPerRyz_09}
Henri Berestycki, Gr{\'e}goire Nadin, Benoit Perthame, and Lenya Ryzhik.
\newblock The non-local {F}isher-{KPP} equation: travelling waves and steady
  states.
\newblock {\em Nonlinearity}, 22(12):2813--2844, 2009.

\bibitem{BerMouRao_15}
Nathana{\"e}l Berestycki, Cl{\'e}ment Mouhot, and Ga{\"e}l Raoul.
\newblock Existence of self-accelerating fronts for a non-local
  reaction-diffusion equations.
\newblock {\em arXiv preprint arXiv:1512.00903}, 2015.

\bibitem{BouCalMeuMirPerRao_12}
Emeric Bouin, Vincent Calvez, Nicolas Meunier, Sepideh Mirrahimi, Beno{\^\i}t
  Perthame, Ga{\"e}l Raoul, and Rapha{\"e}l Voituriez.
\newblock Invasion fronts with variable motility: phenotype selection, spatial
  sorting and wave acceleration.
\newblock {\em Comptes Rendus Mathematique}, 350(15-16):761--766, 2012.

\bibitem{BouHenRhy_16}
Emeric Bouin, Christopher Henderson, and Lenya Ryzhik.
\newblock Super-linear spreading in local and non-local cane toads equations.
\newblock {\em Journal de math{\'e}matiques Pures et Appliqu{\'e}es},
  108(5):724--750, 2017.

\bibitem{ConDomRoqRyz_06}
Peter Constantin, Komla Domelevo, Jean-Michel Roquejoffre, and Lenya Ryzhik.
\newblock Existence of pulsating waves in a model of flames in sprays.
\newblock {\em J. Eur. Math. Soc. (JEMS)}, 8(4):555--584, 2006.

\bibitem{Dav_05}
Margaret~B Davis, Ruth~G Shaw, and Julie~R Etterson.
\newblock Evolutionary responses to changing climate.
\newblock {\em Ecology}, 86(7):1704--1714, 2005.

\bibitem{DinHamZha_14}
Weiwei Ding, Fran\c{c}ois Hamel, and Xiao-Qiang Zhao.
\newblock Bistable pulsating fronts for reaction-diffusion equations in a
  periodic habitat.
\newblock {\em Indiana Univ. Math. J.}, 66(4):1189--1265, 2017.

\bibitem{Dup_12}
Anne Duputi{\'e}, Fran{\c{c}}ois Massol, Isabelle Chuine, Mark Kirkpatrick, and
  Oph{\'e}lie Ronce.
\newblock How do genetic correlations affect species range shifts in a changing
  environment?
\newblock {\em Ecology letters}, 15(3):251--259, 2012.

\bibitem{Ett_08}
Julie~R Etterson, Daniel~E Delf, Timothy~P Craig, Yoshino Ando, and Takayuki
  Ohgushi.
\newblock Parallel patterns of clinal variation in solidago altissima in its
  native range in central usa and its invasive range in japan.
\newblock {\em Botany}, 86(1):91--97, 2008.

\bibitem{FifMac_77}
Paul~C. Fife and J.~B. McLeod.
\newblock The approach of solutions of nonlinear diffusion equations to
  travelling front solutions.
\newblock {\em Arch. Ration. Mech. Anal.}, 65(4):335--361, 1977.

\bibitem{Fis_37}
R.~A. Fisher.
\newblock The wave of advance of advantageous genes.
\newblock {\em Ann. of Eugenics}, 7:355--369, 1937.

\bibitem{Gar_11}
Jimmy Garnier.
\newblock Accelerating solutions in integro-differential equations.
\newblock {\em SIAM Journal on Mathematical Analysis}, 43(4):1955--1974, 2011.

\bibitem{Gri_06}
Timothy~M Griffith and Maxine~A Watson.
\newblock Is evolution necessary for range expansion? manipulating reproductive
  timing of a weedy annual transplanted beyond its range.
\newblock {\em The American Naturalist}, 167(2):153--164, 2006.

\bibitem{Ham_08}
Fran{\c{c}}ois Hamel.
\newblock Qualitative properties of monostable pulsating fronts: exponential
  decay and monotonicity.
\newblock {\em J. Math. Pures Appl. (9)}, 89(4):355--399, 2008.

\bibitem{HamRoq_10}
Fran{\c{c}}ois Hamel and Lionel Roques.
\newblock Fast propagation for {KPP} equations with slowly decaying initial
  conditions.
\newblock {\em Journal of Differential Equations}, 249(7):1726--1745, 2010.

\bibitem{HamRoq_11}
Fran{\c{c}}ois Hamel and Lionel Roques.
\newblock Uniqueness and stability properties of monostable pulsating fronts.
\newblock {\em J. Eur. Math. Soc. (JEMS)}, 13(2):345--390, 2011.

\bibitem{HamRyz_14}
Fran{\c{c}}ois Hamel and Lenya Ryzhik.
\newblock On the nonlocal {F}isher-{KPP} equation: steady states, spreading
  speed and global bounds.
\newblock {\em Nonlinearity}, 27(11):2735--2753, 2014.

\bibitem{Hen_14}
Christopher Henderson.
\newblock Pulsating fronts in a 2{D} reactive {B}oussinesq system.
\newblock {\em Comm. Partial Differential Equations}, 39(8):1555--1595, 2014.

\bibitem{Her_12}
Rutger Hermsen, J~Barrett Deris, and Terence Hwa.
\newblock On the rapidity of antibiotic resistance evolution facilitated by a
  concentration gradient.
\newblock {\em Proceedings of the National Academy of Sciences},
  109(27):10775--10780, 2012.

\bibitem{HudZin_95}
W.~Hudson and B.~Zinner.
\newblock Existence of traveling waves for reaction diffusion equations of
  {F}isher type in periodic media.
\newblock In {\em Boundary value problems for functional-differential
  equations}, pages 187--199. World Sci. Publ., River Edge, NJ, 1995.

\bibitem{Kel_08}
Stephen~R Keller and Douglas~R Taylor.
\newblock History, chance and adaptation during biological invasion: separating
  stochastic phenotypic evolution from response to selection.
\newblock {\em Ecology Letters}, 11(8):852--866, 2008.

\bibitem{Key_06}
Juan~E Keymer, Peter Galajda, Cecilia Muldoon, Sungsu Park, and Robert~H
  Austin.
\newblock Bacterial metapopulations in nanofabricated landscapes.
\newblock {\em Proceedings of the National Academy of Sciences},
  103(46):17290--17295, 2006.

\bibitem{KolPetPis_37}
A.~N. Kolmogorov, I.~G. Petrovsky, and N.~S. Piskunov.
\newblock Etude de l'\'equation de la diffusion avec croissance de la
  quantit\'e de mati\`{e}re et son application \`{a} un probl\`{e}me
  biologique.
\newblock {\em Bull. Univ. Etat Moscou}, S\'er. Inter. A 1:1--26, 1937.

\bibitem{MirRao_13}
Sepideh Mirrahimi and Ga{\"e}l Raoul.
\newblock Dynamics of sexual populations structured by a space variable and a
  phenotypical trait.
\newblock {\em Theoretical population biology}, 84:87--103, 2013.

\bibitem{Pec_98}
Joel~R Peck, Jonathan~M Yearsley, and David Waxman.
\newblock Explaining the geographic distributions of sexual and asexual
  populations.
\newblock {\em Nature}, 391(6670):889--892, 1998.

\bibitem{Pel_20}
Gwena\"{e}l Peltier.
\newblock Accelerating invasions along an environmental gradient.
\newblock {\em J. Differential Equations}, 268(7):3299--3331, 2020.

\bibitem{Pol_05}
Jitka Polechov{\'a} and Nicholas~H Barton.
\newblock Speciation through competition: a critical review.
\newblock {\em Evolution}, 59(6):1194--1210, 2005.

\bibitem{Pre_04}
C{\'e}line Prevost.
\newblock {\em Applications des {\'e}quations aux d{\'e}riv{\'e}s partielles
  aux probl{\`e}mes de dynamique des populations et traitement num{\'e}rique}.
\newblock PhD thesis, Orl{\'e}ans, 2004.

\bibitem{ShiKawTer_86}
Nanako Shigesada, Kohkichi Kawasaki, and Ei~Teramoto.
\newblock Traveling periodic waves in heterogeneous environments.
\newblock {\em Theoret. Population Biol.}, 30(1):143--160, 1986.

\bibitem{Vol_11}
Vitaly Volpert.
\newblock {\em Elliptic partial differential equations. {V}olume 1: {F}redholm
  theory of elliptic problems in unbounded domains}, volume 101 of {\em
  Monographs in Mathematics}.
\newblock Birkh\"{a}user/Springer Basel AG, Basel, 2011.

\bibitem{Wei_02}
Hans~F. Weinberger.
\newblock On spreading speeds and traveling waves for growth and migration
  models in a periodic habitat.
\newblock {\em J. Math. Biol.}, 45(6):511--548, 2002.

\bibitem{Xin_00}
Jack Xin.
\newblock Front propagation in heterogeneous media.
\newblock {\em SIAM Rev.}, 42(2):161--230, 2000.

\bibitem{Xin_93}
Jack~X. Xin.
\newblock Existence and nonexistence of traveling waves and reaction-diffusion
  front propagation in periodic media.
\newblock {\em J. Statist. Phys.}, 73(5-6):893--926, 1993.

\bibitem{Xin_91}
Xue Xin.
\newblock Existence and stability of traveling waves in periodic media governed
  by a bistable nonlinearity.
\newblock {\em J. Dynam. Differential Equations}, 3(4):541--573, 1991.

\bibitem{Zei_86}
Eberhard Zeidler.
\newblock {\em Nonlinear functional analysis and its applications. {I}}.
\newblock Springer-Verlag, New York, 1986.
\newblock Fixed-point theorems, Translated from the German by Peter R. Wadsack.

\end{thebibliography}

\end{document}